\documentclass[oneside,a4paper,english, 11pt]{amsart}
\makeatletter
 \theoremstyle{plain}
\newtheorem{thm}{Theorem}[section]
\theoremstyle{plain}
  \newtheorem{prop}[thm]{Proposition}
\theoremstyle{plain}
\newtheorem{conj}[thm]{Conjecture}
\theoremstyle{plain}
 \newtheorem{lemma}[thm]{Lemma}
\theoremstyle{plain}
 \newtheorem*{claim}{Claim}
\theoremstyle{plain}

\theoremstyle{plain}
\newtheorem{cor}[thm]{Corollary}
\theoremstyle{definition}
  \newtheorem{defn}[thm]{Definition}
    \newtheorem{const}[thm]{Construction}
    
    \newtheorem{notation}[thm]{Notation}
 \theoremstyle{definition}
  
  \newtheorem{exam}[thm]{Example}
\theoremstyle{remark}
\newtheorem{rmk}[thm]{Remark}
\numberwithin{equation}{section}

\usepackage{amsmath}
\usepackage{amssymb}
\usepackage{amscd}
\usepackage{amsthm}
\usepackage{array}
\usepackage[arrow,matrix,tips]{xy}

\usepackage{stmaryrd}
\usepackage{lscape}
\usepackage{caption}
\usepackage{tikz-cd}
\usepackage{url}
\usepackage{makecell, multirow, tikz-cd, mathtools}

\pdfpageheight\paperheight
\pdfpagewidth\paperwidth
\let\det\relax
\DeclareMathOperator{\det}{det}

\newcommand{\rec}{\operatorname{rec}}
\newcommand{\cInd}{\text{c-Ind}}

\newcommand{\St}{\mathrm{St}}

\newcommand{\trace}{\operatorname{tr}}

\newcommand{\lbar}[1]{\overline{#1}}
\newcommand{\fs}{\mathfrak{s}}
\newcommand{\fp}{\mathfrak{p}}

\newcommand{\Z}{\mathbb{Z}}

\newcommand{\Q}{\mathbb{Q}}

\newcommand{\Qp}{\mathbb{Q}_p}
\newcommand{\R}{\mathbb{R}}
\newcommand{\C}{\mathbb{C}}
\newcommand{\bG}{\mathbb{G}}

\newcommand{\F}{\mathbb{F}}
\newcommand{\N}{\mathbb{N}}

\newcommand{\J}{\mathrm{J}}

\newcommand{\fM}{\mathfrak{M}}

\newcommand{\fS}{\mathfrak{S}}

\newcommand{\fm}{\mathfrak{m}}

\newcommand{\bA}{\mathbb{A}}
\newcommand{\A}{\mathbb{A}}

\newcommand{\bF}{\mathbb{F}}

\newcommand{\bN}{\mathbb{N}}

\newcommand{\bP}{\mathbb{P}}
\newcommand{\bQ}{\mathbb{Q}}
\newcommand{\bR}{\mathbb{R}}

\newcommand{\bT}{\mathbb{T}}

\newcommand{\bZ}{\mathbb{Z}}

\newcommand{\cC}{\mathcal{C}}

\newcommand{\cE}{\mathcal{E}}

\newcommand{\cG}{\mathcal{G}}
\newcommand{\cH}{\mathcal{H}}

\newcommand{\cJ}{\mathcal{J}}
\newcommand{\cK}{\mathcal{K}}
\newcommand{\cL}{\mathcal{L}}

\newcommand{\cN}{\mathcal{N}}
\newcommand{\cO}{\mathcal{O}}
\newcommand{\cP}{\mathcal{P}}
\newcommand{\cQ}{\mathcal{Q}}
\newcommand{\cR}{\mathcal{R}}
\newcommand{\cS}{\mathcal{S}}
\newcommand{\cT}{\mathcal{T}}
\newcommand{\cU}{\mathcal{U}}
\newcommand{\cV}{\mathcal{V}}

\newcommand{\cX}{\mathcal{X}}
\newcommand{\cY}{\mathcal{Y}}
\newcommand{\cZ}{\mathcal{Z}}

\newcommand{\ovl}{\overline}
\newcommand{\un}{\underline}

\newcommand{\Gal}{\mathrm{Gal}}
\newcommand{\Hom}{\mathrm{Hom}}

\newcommand{\Ind}{\mathrm{Ind}}

\newcommand{\GL}{\mathrm{GL}}

\newcommand{\Gm}{\mathbb{G}_m}
\newcommand{\ad}{\mathrm{ad}}

\newcommand{\red}{\mathrm{red}}

\newcommand{\Spec}{\mathrm{Spec}\ }

\newcommand{\JH}{\mathrm{JH}}
\newcommand{\Rep}{\mathrm{Rep}}

\newcommand{\Spf}{\mathrm{Spf}}

\DeclareMathOperator{\codim}{codim}

\DeclareMathOperator{\rhobar}{\overline{\rho}}
\newcommand{\rbar}{\overline{r}}

\usepackage{xcolor}
\newcommand{\Brandon}[1]{\textcolor{orange}{#1}}

\newcommand{\ra}{\rightarrow}

\newcommand{\iarrow}{\hookrightarrow}

\newcommand{\tld}{\widetilde}
\makeatletter
\let\@wraptoccontribs\wraptoccontribs
\makeatother

\makeatother

\author{Robin Bartlett}
\address{School of Mathematical Sciences,
 Queen Mary University of London,
London,
E1 4NS, 
United Kingdom}
\email{robin.bartlett.math@gmail.com}

\author{Bao V.~Le Hung}
\address{Department of Mathematics,
Northwestern University, 
2033 Sheridan Road\\
Evanston, IL 60208, USA}
\email{lhvietbao@googlemail.com}

\author{Brandon Levin}
\address{Department of Mathematics,
Rice University, 
6100 Main Street,
Houston, Texas 77005, USA}
\email{bl70@rice.edu}

\contrib[with an appendix by]{Andrea Dotto}

\title{Resolutions of spaces of crystalline representations and modularity}

\begin{document}

\begin{abstract} We introduce a new partial resolution of crystalline spaces of Galois representations when the gaps in Hodge--Tate weights are smaller than $p$, with no bound on ramification.  Furthermore, when $n =3$ in the case of minimal regular weight, we are able to show that the resolution is normal (assuming the ramification index is divisible by 3).  Employing base change techniques and further analysis of the resolution, we are able to show that all the components of the crystalline deformation rings are potentially diagonalizable. As a consequence, we deduce automorphy lifting, the weight part of Serre's conjecture, and the Breuil-M\'ezard conjecture in dimension three for minimal regular weight. 
\end{abstract}

\maketitle

\tableofcontents

\section{Introduction}

Since the landmark work of \cite{mffgs}, it has been understood that the Taylor--Wiles patching method reduces automorphy lifting theorems to the study of deformation spaces of Galois representations of $\ell$-adic fields with $p$-complete coefficients. The deepest subtleties arise when $\ell = p$, where one must analyze loci cut out by $p$-adic Hodge-theoretic conditions. These conditions are naturally formulated only after inverting $p$, and as a result, a precise understanding of their integral variation inside deformation spaces remains largely out of reach. Fortunately, \cite{BLGGT14} isolated a more accessible property---\emph{potential diagonalizability}---that suffices for applications to automorphy. A local Galois representation is potentially diagonalizable if, after restriction to a finite-index subgroup, it lies on the same irreducible component of the crystalline deformation space as a direct sum of characters. Combined with automorphic base change, this notion allowed \cite{BLGGT14} to prove the strongest known automorphy lifting theorems in dimensions greater than two. It is therefore of central importance to determine whether all potentially crystalline representations are potentially diagonalizable.

Despite its importance, potential diagonalizability has so far been established only in relatively restricted settings. Beyond the ordinary and Fontaine--Laffaille representations treated in \cite{BLGGT14,GL14,B20}, the only other complete case is that of two-dimensional potentially crystalline representations with Hodge--Tate weights in $[0,1]$ (the potentially Barsotti--Tate case). This was established in \cite{gee-kisin}, building heavily on \cite{mffgs}. This case played a crucial role in the proof of the Breuil--M\'ezard conjecture for potentially Barsotti--Tate deformation spaces and the weight part of Serre's conjecture for $\GL_2$.

Recent work has led to significant advances in understanding potentially crystalline deformation spaces (see \cite{LLLM-shapes, LLLM-models, B23, B24}), but existing methods typically impose restrictions on the ramification index. Such restrictions are a fundamental obstacle to proving potential diagonalizability, which inherently requires handling arbitrarily  ramified extensions.  In this paper, we introduce methods that are insensitive to the ramification index, allowing us to establish potential diagonalizability for three-dimensional representations of minimal regular weight:

\begin{thm} \label{thma}
  Let $K$ be a finite extension of $\bQ_p$ with $p \geq 5$. Then any potentially crystalline representation $\rho\colon G_K \rightarrow \operatorname{GL}_3(\overline{\bQ}_p)$ with minimal regular Hodge--Tate weights $(2,1,0)$ is potentially diagonalizable.
\end{thm}

Combining Theorem~\ref{thma} with the methods of \cite{BLGGT14}, we immediately obtain:

\begin{thm} \label{thmc} (Automorphy Lifting) 
Let $F$ be an imaginary CM field with maximal totally real subfield $F^+$, and let
$c \in \Gal(F/F^+)$ denote the non-trivial element.  Assume $p \geq 5$ and $F$ is split at all places of $F^+$ above $p$. Let $\iota:\overline{\Q}_p\cong \C$, and $r : G_F \to \GL_3(\overline{\Q}_p)$ be a continuous irreducible representation with residual representation $\overline{r}:G_F \rightarrow \operatorname{GL}_3(\overline{\mathbb{F}}_p)$. Assume that:
\begin{enumerate}
\item \emph{(odd essential conjugate self-duality)}  
There is a character $\chi : G_{F^+} \to \overline{\Q}_p^\times$ with $\chi(c_v) = -1$ for all
$v \mid \infty$ and an isomorphism $r^c \cong r^\vee \otimes \chi$.

\item \emph{(unramified almost everywhere)}  
The representation $r$ is ramified at only finitely many primes.

\item \emph{(minimal regular potentially crystalline)}  
For all places $v|p$, the restriction $r|_{G_{F_v}}$ is potentially crystalline with $\kappa$-Hodge--Tate weights $(2,1,0)$ for all $\kappa:F_v \hookrightarrow \overline{\bQ}_p$. 
\item \emph{(adequate)} $\rbar|_{ G_{F(\zeta_p)}}$ is irreducible and $\rbar( G_{F(\zeta_p)})$ is an adequate subgroup of $\GL_3(\overline{\F})$; and
\item \emph{(residual modularity)} $\overline{r} \cong \overline{r}_\iota(\pi)$ for some $\pi$ a regular algebraic conjugate essentially self-dual cuspidal  automorphic (RAECSDC) representation of $\GL_3(\A_F)$. 
\end{enumerate}
Then $r$ is automorphic, i.e.~$r \cong r_\iota(\pi')$ for some $\pi'$ a RAECSDC automorphic representation of $\GL_3(\A_F)$.
\end{thm}

As a second application of Theorem~\ref{thma} we prove the Breuil--M\'ezard conjecture in minimal weight for potentially crystalline representations of dimension three, following the strategy laid out in dimension two in \cite{gee-kisin, EG14}. We state a version for the Emerton--Gee stack, which implies the corresponding statement for deformation spaces (see \cite[\S 8.3]{EGstack}). 

Let $\cX_3$ denote the moduli stack introduced in \cite{EGstack}, parametrizing $3$-dimensional $p$-adic
representations of $G_K$. The irreducible components of the underlying reduced stack $\cX_{3,\red}$ are naturally labeled by Serre weights $\sigma$ (i.e.\ irreducible $\overline{\F}_p$-representations of $\GL_3(\cO_K)$); we write $\cC_\sigma$ for the component corresponding to $\sigma$, and $\Z[\cX_{3,\red}]$ for the free abelian group generated by these components. If
$\tau : I_K \rightarrow \GL_3(\overline{\Q}_p)$ is an inertial type for
$K$, let $\sigma(\tau)$ denote the $\GL_3(\cO_K)$-representation
associated to $\tau$ by the inertial local Langlands correspondence
\cite{CEGS+16}, and write $\overline{\sigma}(\tau)$ for the
semisimplification of the mod $p$ reduction of a lattice in $\sigma(\tau)$. Let
$\cX^{\tau}_{\eta} \subset \cX_3$ denote the substack parametrizing
potentially crystalline representations of inertial type $\tau$ and
$\kappa$-Hodge--Tate weights $(2,1,0)$ for each
$\kappa : K \hookrightarrow \overline{\Q}_p$.

\begin{thm} \label{thm:introBMconj} (Geometric Breuil--M\'ezard in minimal weight) For each Serre weight $\sigma$ for $\GL_3(\cO_K)$,   there exist a cycle $\cZ_\sigma \in \Z[\cX_{3, \mathrm{red}}]$ such that for every inertial type $\tau$ for $K$, we have
\begin{equation} \label{eqn:BM}
\sum_\sigma \mu_\sigma(\cX^{\tau}_{\eta, \overline{\F}_p}) \cC_\sigma = \sum_{\sigma} [\ovl{\sigma}(\tau):\sigma] \cZ_\sigma.
\end{equation}
Here $\mu_\sigma(\cX^{\tau}_{\eta, \overline{\F}_p})$ denotes the multiplicity of $\cC_\sigma$ as an irreducible component of $\cX^{\tau}_{\eta, \overline{\F}_p}$ and $[\ovl{\sigma}(\tau):\sigma]$ is the multiplicity of $\sigma$ as a Jordan--H\"older factor of $\ovl{\sigma}(\tau)$.

Moreover, the cycles $\cZ_{\sigma}$ are unique and effective.   
\end{thm}

\begin{rmk} \label{rmk:BMC}
\begin{enumerate}
\item Theorem~\ref{thm:introBMconj} is the first result towards the
Breuil--M\'ezard conjecture in dimension greater than two that allows arbitrary (in particular, wildly ramified) inertial types. Although the theorem
is restricted to minimal regular Hodge--Tate weights, the
$\overline{\sigma}(\tau)$ span the Grothendieck group of mod~$p$
representations of $\GL_3(\cO_K)$ rationally
(cf.\ Appendix~\ref{appendix}), so the system of
equations~\eqref{eqn:BM} determines the cycles $\cZ_{\sigma}$
uniquely. 
\item As explained in \cite{GHS}, the existence of the cycles
$\cZ_\sigma$ gives an (a priori inexplicit) formulation of the weight part of Serre's conjecture: a global mod~$p$ automorphic Galois
representation $\rbar$ is automorphic of weight $\{\sigma_v\}$ if and
only if its local components $\rbar_v$ lie in the support of
$\cZ_{\sigma_v}$ for each $v | p$. In Section~\ref{sec:WPSC}, we
prove this result under standard Taylor--Wiles
hypotheses, using the connection with patching already present in
the proof of Theorem~\ref{thm:introBMconj}.
\item  To make the weight part of Serre's conjecture explicit, one needs to compute the support of $\cZ_\sigma$.  When $K=\Qp$, the second two authors, together with Daniel Le and Stefano Morra, determine that the cycles for $\sigma = F(a,b,c)$ with $a \ge b \ge c$ are given by 

\[
\cZ_\sigma =
\begin{cases}
\cC_\sigma
& \text{if } a-b \neq p-1,\; b-c \neq p-1, \\[6pt]

\cC_\sigma + \cC_{F(b,b,c)}
& \text{if } a-b = p-1,\; b-c \neq p-1, \\[6pt]

\cC_\sigma + \cC_{F(a,b,b)}
& \text{if } a-b \neq p-1,\; b-c = p-1, \\[6pt]

\cC_{\sigma}
+ \cC_{F(b,b,c)}
+ \cC_{F(a,b,b)}
+ \cC_{F(c,c,c)}
& \text{if } a-b = b-c = p-1 .
\end{cases}
\]

This result will appear in forthcoming work which will further determine the cycles whenever $K/\Qp$ is unramified.     

\end{enumerate}
\end{rmk}

We now sketch the proof of Theorem~\ref{thma}, which builds on the strategy introduced in \cite{mffgs}. In modern terms, Kisin's approach revolves around the construction of a partial resolution  \[\cY^{\operatorname{cr}}_{\lambda} \rightarrow \cX_\lambda^{\operatorname{cr}}\] of the crystalline locus with Hodge--Tate weights $\lambda$ inside the Emerton--Gee stack of $G_K$-representations. More precisely, this resolution is defined as the $\bZ_p$-flat part of a fibre product
\[
\cY^{\mathrm{cr}}_{\lambda}
:= \left( \cX_\lambda^{\operatorname{cr}} \times_{\mathcal R_n} Y_{\le\lambda} \right)^{\operatorname{fl}},
\]
where $Y_{\le\lambda}$ denotes the moduli stack of rank $n$ Breuil--Kisin modules of height $\leq \lambda$, and the morphism $\cX_\lambda^{\operatorname{cr}} \rightarrow \cR_n$ corresponds to restriction to $G_{K_\infty} \subset G_K$ for an extension $K_\infty=K(\pi^{1/p^{\infty}})$ obtained by adjoining a fixed system of $p$-th power roots of a uniformizer $\pi \in K$. 

The morphism $\cY^{\mathrm{cr}}_{\lambda}\rightarrow \cX_\lambda^{\operatorname{cr}}$ becomes an isomorphism after inverting $p$, justifying the idea that the former is a (partial) resolution of the latter. In the two-dimensional Barsotti--Tate case, or more generally whenever $\lambda$ consists of minuscule coweights, the situation simplifies considerably because $\cY^{\mathrm{cr}}_{\lambda} \cong Y_{\le\lambda}$. In other words, Breuil--Kisin modules of height $\leq \lambda$ correspond exactly to lattices in crystalline representations of weight $\lambda$ in this situation. The geometry of $\cY^{\mathrm{cr}}_{\lambda}$ is then governed by that of $Y_{\le\lambda}$, whose local models are (mixed characteristic) affine Schubert varieties. These are known to be normal in great generality \cite{PR03, PZ13, Levin16}. 

Since $\cY^{\operatorname{cr}}_{\lambda}$ is normal when $\lambda$ is minuscule, the irreducible components of the crystalline deformation space of a representation $\rhobar \in \cX_\lambda^{\operatorname{cr}}(\overline{\F}_p)$ correspond to connected components of the fiber of $\cY^{\operatorname{cr}}_\lambda \rightarrow \cX_\lambda^{\operatorname{cr}}$ above $\rhobar$. This fiber can be analyzed using the explicit description of $Y_{\le\lambda}$, and ultimately potential diagonalizability reduces to finding a point on each of its connected components that admits a diagonal lift to characteristic zero. We stress that the normality of the resolution $\cY^{\operatorname{cr}}_{\lambda}$ is the key ingredient of the entire strategy.


Beyond the minuscule setting, $G_K$-stable lattices in crystalline representations only come from special kinds of Breuil--Kisin modules, and the relationship between $\cY^{\mathrm{cr}}_{\lambda}$ and $Y_{\le\lambda}$ becomes significantly more subtle. In particular, unlike $Y_{\le\lambda}$, it is not known in general whether $\cY^{\mathrm{cr}}_{\lambda}$ is normal. A central innovation of this paper is a framework to restore control over the geometry whenever $\lambda$ is concentrated in degree $[0,p-1]$ (in other words, all Hodge--Tate weights lie in this interval), but crucially, without imposing any restriction on the ramification index $e$.

The key idea is to enhance the moduli problem $Y_{\le\lambda}$ by equipping a Breuil--Kisin module $\fM$ (which we recall is a projective module over $\fS=W(k)[\![u]\!]$ with Frobenius structure $\varphi_{\fM}$, where $k$ is the residue field of $K$) with two additional structures:
\begin{enumerate}
    \item \label{item: truncated monodromy}An $\fS$-linear endomorphism $N_0: \fM/u^{e+1}\fM\rightarrow \fM/u^{e+1}\fM$. This induces a derivation $N_0^{\varphi}$ of $\varphi^*\fM/u^{pe+1}\varphi^*\fM$, which we refer to as \emph{truncated monodromy} and obeys certain stability and commutation relation with Frobenius (cf Definition \ref{def-BK+monodromy stack}). 
    \item \label{item: convolution}A convolution structure on $\varphi^*\fM$, i.e.\ an  $e$-step filtration
    $$
    \fM_e := \varphi^{-1}_{\fM}(E(u)^{p-1}\fM) \subset \fM_{e-1} \subset \ldots \subset \fM_0 := \varphi^*\fM.
    $$
\end{enumerate}
These two extra structures have natural origins: the data of $N_0$ in (\ref{item: truncated monodromy}) is the (linearization of the) structure needed to extend the $G_{K_\infty}$-representation associated to $\fM$ to a $G_K$-representation, while the convolution data in (\ref{item: convolution}) is designed to capture an integral version of the Hodge filtration (or rather, the Nygaard filtration lifting it).

Incorporating convolution structures into the resolution on the Breuil--Kisin module (and taking $\Z_p$-flat parts) yields a refinement $\cY^{\mathrm{cr},\mathrm{conv}}_{\lambda}$ of our previous resolution $\cY^{\mathrm{cr}}_{\lambda}$.
In characteristic $p$, there is a natural map $\iota: \cY^{\operatorname{cr},\operatorname{conv}}_{\lambda} \otimes_{\overline{\mathbb{Z}}_p} \overline{\mathbb{F}}_p
\rightarrow  Y^{\nabla,\operatorname{conv}}_{\leq \lambda}$,  where $ Y^{\nabla,\operatorname{conv}}_{\leq \lambda}$ is  the algebraic stack over $\operatorname{Spec}\overline{\bF}_p$ parametrizing Breuil--Kisin modules $\fM$ of height $\leq \lambda$, with $N_0$ as in (\ref{item: truncated monodromy}) and a convolution structure as in (\ref{item: convolution}) satisfying the compatibility $u^{i-1}N_0^{\varphi}(\fM_i) \subset \fM_i$. Furthermore, $\iota$ factors through a closed substack $ Y^{\nabla,\operatorname{conv}}_{\lambda}\subset   Y^{\nabla,\operatorname{conv}}_{\leq \lambda}$ isolating the ``elementary divisor $\lambda$-part'' as elaborated on below. 


For general $\lambda$, we do not at present have enough control over $Y^{\nabla,\operatorname{conv}}_{\lambda}$ to decide whether $\cY^{\operatorname{cr},\operatorname{conv}}_{\lambda}= Y^{\nabla,\operatorname{conv}}_{\lambda}$ or to analyze the singularities of $Y^{\nabla,\operatorname{conv}}_{\lambda}$. However, when $\lambda_{\kappa} = (2,1,0)$ for all $\kappa:K\iarrow \overline{\Qp}$, a series of numerical coincidences—ultimately reflecting the fact that $\lambda$ is quasi-minuscule—allow us to prove the following key technical input which underlies Theorem~\ref{thma}.

\begin{thm}\label{thmb}
    Assume $p \geq 5$ and $e \equiv 0$ modulo $3$ and $\lambda_{\kappa} =(2,1,0)$ for all $\kappa$. Then $Y^{\nabla,\operatorname{conv}}_\lambda$
    is Cohen--Macaulay and generically reduced, hence reduced. Furthermore
\[\cY^{\operatorname{cr},\operatorname{conv}}_{\lambda} \otimes_{\overline{\bZ}_p} \overline{\bF}_p \cong
    Y^{\nabla,\operatorname{conv}}_\lambda.\]
In particular, $\cY^{\operatorname{cr},\operatorname{conv}}_{\lambda}$ is normal.

\end{thm}
The assumptions $p \geq 5$ and $e \equiv 0 \pmod{3}$ arise from specific dimension estimates in the proof, and it appears likely they can be relaxed.

 With Theorem~\ref{thmb} in hand, the strategy outlined above applies. After a base change, Theorem~\ref{thma} reduces to analyzing the fibres of
\[
\cY^{\operatorname{cr},\operatorname{conv}}_{\lambda} \longrightarrow \cX^{\operatorname{cr}}_{\lambda}
\]
above the trivial representation $\overline{\rho}$. This fiber is an explicit substack of $Y^{\nabla,\operatorname{conv}}_\lambda$, which is equipped with an action of a maximal torus $T$ centralizing $\rhobar$. This action together with the explicit relation to convolution fibers allows us to show each connected component contains a special $T$-fixed point which lifts to characteristic zero, thus yielding enough diagonalizable points (see Section~6).


The proof of Theorem~\ref{thmb} takes up the main part of this paper, and involves three key steps, which we elaborate on in the remainder of the introduction.

\textit{Step 1 - Truncated monodromy and descent from $G_{K_\infty}$}:
In Section~2, we make precise the role of truncated monodromy in extending a $G_{K_\infty}$-action to $G_K$ for any Hodge--Tate weights concentrated in degree $[0,p-1]$. 

Given a Breuil--Kisin module $\fM$, the base change 
\[
\fM_{\mathrm{inf}} = \fM \otimes_{\fS} W(\mathcal{O}_{C^\flat})
\]
gives rise to an étale $(\varphi,G_{K_\infty})$-module corresponding to a $\Z_p[G_{K_\infty}]$-representation $V$. Extending $V$ to a crystalline $G_K$-representation is equivalent to equipping $\fM_{\mathrm{inf}}$ with a compatible $G_K$-action. Taking the logarithmic derivative of this action gives a $\varphi$-equivariant derivation $\cN$ on a suitable completion of $\fM[\frac{1}{p}]$. When  the Hodge--Tate weights belong to $[0,p-1]$, we can sufficiently control the denominators occurring in $\cN$ to ensure the truncation of  $\cN$ and $\varphi^*\cN$ modulo $u^{e+1}$ and $u^{ep +1}$ respectively are integral. Reducing these integral truncations modulo $p$ produces the operators $N_0$ and $N_0^{\varphi}$ parametrised in (1). The main outcome is Theorem~\ref{thm-embedding} which produces a closed immersion
\begin{equation}\label{eq-intro embedding}
\cY^{\operatorname{cr}}_{\lambda} \otimes_{\overline{\bZ}_p} \overline{\bF}_p\rightarrow Y^{\nabla}_{\leq \lambda}
\end{equation}
where $Y^{\nabla}_{\leq \lambda}$ denotes an algebraic stack over $\operatorname{Spec}\overline{\bF}_p$ classifying Breuil--Kisin modules $\fM$ of height $\leq \lambda$ and an $N_0$ as in (1).

\textit{Step 2 - Convolution structures and Pl\"ucker equations:}
Sections~3 and~4 introduce convolution structures and then impose equations cutting out the exact Hodge--Tate weights.

Heuristically, in characteristic $0$ the condition that the Hodge--Tate weights are $\leq \lambda$ distributed along completions at the distinct roots of $E(u)$, which fuse as one degenerates to characteristic $p$ (in particular, see \cite{B24,B23}). 
Convolution structures have the effect of undoing this fusion and generally make singularities milder: even in the minuscule case, $\cY^{\operatorname{cr},\operatorname{conv}}_{\lambda}$ is smooth (cf \cite{B23b}), whereas $\cY^{\operatorname{cr}}_{\lambda}$ is usually singular.

In non-minuscule situations, one has to contend with a further problem, namely how to isolate the condition that the Hodge--Tate weights are exactly $\lambda$, as opposed to just bounded by $ \lambda$. This is delicate: while the latter 
 is a closed condition on the moduli of Breuil--Kisin modules, the former condition is only locally closed.  To isolate the condition of having weights \emph{exactly} $\lambda$, we analyze the interaction between the monodromy operator and the convolution filtration. 

For simplicity, assume $K$ is totally ramified and choose an ordering of its embeddings into $\overline{\bQ}_p$ so that the Hodge--Tate weights are given by an $e$-tuple of elementary divisors $(\lambda_i)_{1\leq i \leq e}$. Let $(\fM,\fM_i )\in  Y^{\nabla,\operatorname{conv}}_\lambda(\overline{\Z}_p)$ be a Breuil--Kisin module coming from a crystalline representation, together with its associated convolution structure. The crystallinity of $\fM$ guarantees that, for each $1 \leq i \leq e$, one has the stability
\begin{equation}\label{eq-intro monodromy}
\bigg( \prod_{j=1}^{i} (u-\kappa_j(\pi)) \bigg)\varphi^*\cN(\fM_i) \subset \fM_i[\tfrac{1}{p}],
\end{equation}
(more precisely, the right hand side should be replaced by a suitable completion of $\fM_i[\frac{1}{p}]$). For a fixed $1 \leq i \leq e$ we can, after trivializing $\fM_{i-1}$, interpret the lattice $\fM_i \subset \fM_{i-1}$ as a point $x_i$ in the Schubert variety $\operatorname{Gr}_{\leq \lambda_i}$ of the affine Grassmannian for $\operatorname{GL}_d$, with parameter $u-\kappa_i(\pi)$. The stability in \eqref{eq-intro monodromy} then admits the following equivalent reformulation: after inverting $p$, the point $x_i$ is fixed by an $L^+G$-conjugate of the loop rotation $\bG_m$-action on $\operatorname{Gr}_{\leq \lambda_i}$. Specifically, the conjugating element in $L^+G$ is obtained from the restriction of $\bigg( \prod_{j=1}^{i-1} (u-\kappa_j(\pi)) \bigg)\varphi^*\cN$ to $\fM_{i-1}$.

The fixed points under this $\bG_m$-action are well known: they are the $L^+G$-translates of a disjoint union of flag varieties $\bigcup_{\mu \leq \lambda_i} \operatorname{FL}_{\mu}$ embedded into $\operatorname{Gr}_{\leq \lambda_i}$. However, only the translate of the top flag variety $\operatorname{FL}_{\lambda_i} \subset \operatorname{Gr}_{\lambda_i}$ is relevant to Breuil--Kisin modules of weight $(\lambda_i)_{1 \leq i \leq e}$. To single it out we use the Pl\"ucker embedding
$$
\operatorname{Gr}_{\leq \lambda_i} \hookrightarrow \mathbb{P}(V),
$$
 and characterise $\operatorname{FL}_{\lambda_i}$ as the locus where $\big( \prod_{j=1}^i (u-\kappa_j(\pi))\big)\varphi^*\cN$ acts
with certain explicit eigenvalues. When $\lambda_i$ is concentrated in degree $[0,p-1]$, $N_0$ is a sufficiently good approximation of $\cN$ to interpret these conditions as equations mod $p$, where they cut out the stack $Y^{\nabla,\operatorname{conv}}_{\lambda}$ inside $Y^{\mathrm{conv}}_{\leq \lambda}\otimes_{\overline{\bZ}_p} \overline{\bF}_p$.

Up to this point, everything works for $\lambda$ concentrated in degree $[0,p-1]$. However, unless $\lambda$ is quasi-minuscule the process above produces a large number of equations which we are not generally able to control. When $\lambda$ is quasi-minuscule, a numerical coincidence ensures the number of Pl\"ucker equations obtained as above is exactly the codimension of
$\operatorname{FL}_{\lambda_i}$ inside $\operatorname{Gr}_{\leq \lambda_i}$. The upshot is that if the inequality $\dim Y^{\nabla,\operatorname{conv}}_{\lambda}\geq \dim \cY^{\operatorname{cr},\operatorname{conv}}_{\lambda} \otimes_{\overline{\bZ}_p} \overline{\bF}_p$ is an equality, then $Y^{\nabla,\operatorname{conv}}_{\lambda}$ is a local complete intersection inside an affine bundle over a convolution
product of the $\operatorname{Gr}_{\leq \lambda_i}$'s. Since the latter is known to be Cohen--Macaulay, this implies the same
for $Y^{\nabla,\operatorname{conv}}_{\lambda}$. Additionally if one knows that $Y^{\nabla,\operatorname{conv}}_{\lambda}$ is generically reduced and irreducible, then one concludes $\cY^{\operatorname{cr},\operatorname{conv}}_{\lambda} \otimes_{\overline{\bZ}_p} \overline{\bF}_p$ coincides with $Y^{\nabla,\operatorname{conv}}_{\lambda}$ for dimension reasons, and both are reduced.

\textit{Step 3 - Dimension estimates and irreducibility:} 
The last step is to control the dimension of $Y^{\nabla,\operatorname{conv}}_{\lambda}$, and show it is generically reduced and irreducible. We do this in Section 5, specializing to the case $d=3$ and $\lambda_{\kappa} = (2,1,0)$ for each $\kappa:K\hookrightarrow \overline{\Q}_p$. This is done by analyzing the base and the fiber of the natural map $\pi: Y^{\nabla,\operatorname{conv}}_{\lambda}\rightarrow Y^{\nabla}_{\leq \lambda}$ given by forgetting the convolution structure.  

First, we analyze $Y^{\nabla}_{\leq \lambda}$, which is the moduli of pairs $(\fM,N_0)$ such that $N_0^{\varphi}$ is compatible with Frobenius and satisfies Griffiths transversality. Concretely, after choosing a basis of $\fM$, this is the space of pairs $(X,\cN)\in L\GL_3\times u \mathfrak{gl}_3[u]/u^{e+1}$ such that $X$ has elementary divisors $\leq e(2,1,0)$ and obey the relation
\[
u^e \Big( X \varphi(\cN) X^{-1} - c(u) \frac{d}{du}(X) X^{-1} \Big) \equiv c(u) \cN \mod u^{e+1}
\]
for a suitable $c(u)\in \overline{\F}_p[u]$ depending on $E(u)$. In particular, we can stratify $Y^{\nabla}_{\leq \lambda}$ into strata $Y^{\nabla}(\mu)$ according to the elementary divisor $\mu$ of $X$. An explicit but technical analysis of the $Y^{\nabla}(\mu)$ yields the key estimate 
\[\dim Y^{\nabla}(\mu)\leq 3e- \frac{1}{2}(\dim \operatorname{Gr}_{\leq \lambda}-\dim \operatorname{Gr}_{\leq \mu}),\]
and shows the open stratum $Y^{\nabla}((2e,e,0))$ contains a unique generically reduced irreducible component of maximal dimension $3e$.

Finally, we analyze the fibers of $\pi: Y^{\nabla,\operatorname{conv}}_{\lambda}\rightarrow Y^{\nabla}_{\leq \lambda}$ which are closed subschemes of the fibers for the convolution map 
\[\mathrm{conv}: \operatorname{Gr}_{\leq (2,1,0)}\widetilde{\times}  \operatorname{Gr}_{\leq (2,1,0)}\cdots \widetilde{\times}  \operatorname{Gr}_{\leq (2,1,0)}\rightarrow  \operatorname{Gr}_{\leq (2e,e,0)}.\]
The semi-smallness of $\mathrm{conv}$ now shows that $\dim \pi^{-1}(Y^{\nabla}(\mu))\leq 3e$, so it remains to show that the inequality is strict when $\mu< e(2,1,0)$. We do this by exhibiting for each such $\mu$ and each top dimensional irreducible component $\cC\subset \mathrm{conv}^{-1}( \operatorname{Gr}_\mu)$ an explicit point in $\cC \setminus \pi^{-1}(Y^{\nabla}(\mu))$, by exploiting the description of $\cC$ in terms of (convolution of) Mirkov-Vilonen cycles as in \cite{MV}.

\subsection{Acknowledgments}  
This project originated from conversations at Max Planck Institute and the Hausdorff Center of Mathematics in Bonn which we thank for their excellent working conditions during the special program in Summer 2023. The motivation for the constructions here comes from related constructions in the tamely potentially crystalline setting which were first explored in discussions with Daniel Le and Stefano Morra. 
R.B.~was supported by the EPSRC grant EP-R034826-1 and the European Union’s Horizon 2020
research and innovation programme under the Marie Skłodowska-Curie grant
agreement No 101204730.
B.LH.~acknowledges
support from the National Science Foundation under grants Nos.~DMS-1952678 and DMS-2302619 and the Alfred P.~Sloan Foundation.
B.L.~was supported by National Science Foundation grants Nos.~DMS-2306369 and DMS-2237237 and the Alfred P.~Sloan Foundation.

 AI tools (GPT 5.5 and Gemini) were used during the preparation of this manuscript for proofreading and stylistic polishing to improve the clarity of the exposition.

\section{Stacks of Breuil--Kisin modules}\label{sec-stacks of BK modules}

\subsection{Setup}\label{sec-setup}
\begin{itemize}
	\item Fix a finite extension $K$ of $\bQ_p$ with residue field $k$ and a compatible system $\pi^{1/p^\infty}$ of $p$-th power roots of a uniformiser $\pi \in K$ inside a completed algebraic closure $C$ of $K$. Write $E(u) \in W(k)[u]$ for the minimal polynomial of $\pi$ over $W(k)$, whose degree equals the ramification degree $e = e(K/\mathbb{Q}_p)$ of $K$ over $\mathbb{Q}_p$.
	\item For any $\bZ_p$-algebra $A$ we write $\fS_A$  for the $E(u)$-adic completion of $(W(k) \otimes_{\bZ_p} A)[u]$. If $A$ is $p$-adically complete then $\fS_A \cong (W(k) \otimes_{\bZ_p} A)[[u]]$ and we write $\varphi$ for the $A$-linear Frobenius on $\fS_A$ lifting the $p$-th power map on $k$ and sending $u \mapsto u^p$. If $A$ is $\bZ_p$-finite, then  $\fS_A = \fS \otimes_{\bZ_p} A$ for $\fS := \fS_{\bZ_p}$.
	\item Let $C^\flat := \varprojlim_{x\mapsto x^p} C$ denote the tilt of $C$ and $\cO_{C^\flat}$ its ring of integers. Consider $W(C^\flat)$ and the subring $A_{\operatorname{inf}} = W(\cO_{C^\flat})$, viewed as an $\fS$-algebra via $u \mapsto [\pi^{1/p^\infty}]$, and write $\varphi$ for the Witt vector Frobenius which extends $\varphi$ on $\fS$. We write $\mu = [\epsilon_\infty] -1 \in A_{\operatorname{inf}}$ for a fixed compatible system $\epsilon_\infty = (1,\epsilon_1,\ldots) \in \bZ_p(1)$ of primitive $p$-th power roots of unity in $C$. 
	\item If $A$ is $p$-adically complete and topologically of finite type over $\bZ_p$ (i.e.\ $A \otimes_{\bZ_p} \bF_p$ is finite type over $\bF_p$) then we can form the $\fS_A$-algebras $A_{\operatorname{inf},A}$ and $W(C^\flat)_A$ as in \cite[\S 2.2]{EGstack}. If $A$ is $\bZ_p$-finite, then  $A_{\operatorname{inf},A} = A_{\operatorname{inf}} \otimes_{\bZ_p} A$ and $W(C^\flat)_A = W(C^\flat) \otimes_{\bZ_p} A$. 
    \item Write $c(u) = \frac{\varphi(E(u)) - u^{ep}}{p} \in \fS$ and set $c^*(u) = \frac{\varphi(E(u)) - E(u)^p}{p} = c(u) + \frac{u^{ep} - E(u)^p}{p}$.
\end{itemize}
\subsection{Crystalline Breuil--Kisin modules}

	If $A$ is any $p$-adically complete $\bZ_p$-algebra,  then a Breuil--Kisin module of rank $d$ over $A$ is a projective $\fS_A$-module $\fM$ of rank $d$ equipped with an isomorphism
	$$
	 \varphi_{\fM} : \fM \otimes_{\varphi,\fS_A} \fS_A[\tfrac{1}{E(u)}] \xrightarrow{\sim} \fM[\tfrac{1}{E(u)}] 
	$$
	Define $\varphi^*\fM := \fM \otimes_{\varphi,\fS_A} \fS_A$. 
	We say $\fM$ has height $\leq h$ if $E(u)^h\fM \subset \varphi_{\fM}(\varphi^*\fM) \subset \fM$. If $A$ is furthermore topologically of finite type over $\bZ_p$ then a crystalline $G_K$-action on $\fM$ is a $\varphi$-equivariant, continuous, and $A_{\operatorname{inf},A}$-semilinear action of $G_K$ on $\fM \otimes_{\fS_A} A_{\operatorname{inf},A}$ satisfying 
	\begin{equation}\label{eq-Galois action}
		\sigma(m) - m \in \fM \otimes_{\fS_A} u \varphi^{-1}(\mu) A_{\operatorname{inf},A}, \qquad \sigma_\infty(m) = m
	\end{equation}
	for all $m \in \fM$ and all $\sigma \in G_K$ and all  $\sigma_\infty \in G_{K_\infty}$. 

\begin{defn}\label{def-BKcryststack}
	Fix $d \geq 1$. For any $p$-adically complete topologically finite type $\bZ_p$-algebra $A$  we write $\cY_{\leq h}(A)$ for the groupoid of rank $d$ Breuil--Kisin modules over $A$ of height $\leq h$ equipped with a crystalline $G_K$-action. We write $\cY_{\leq h}$ for the resulting limit preserving category fibred over $\operatorname{Spf}\bZ_p$. This is a finite type $p$-adic formal algebraic stack over $\operatorname{Spf}\bZ_p$ in the sense of \cite[A.7]{EGstack}. See \cite[4.5.20]{EGstack}. We write $\cY^{\operatorname{cr}}_{\leq h}$ for the $\bZ_p$-flat substack of $\cY_{\leq h}$ characterised uniquely by the property that $\cY^{\operatorname{cr}}_{\leq h}(A) = \cY_{\leq h}(A)$ whenever $A$ is finite flat over $\bZ_p$.  See \cite[Proposition 4.8.2]{EGstack} or \cite[Prop 10.7]{B24}.  
\end{defn}

 To work with $\cY^{\operatorname{cr}}_{\leq h}$, we will use an additional level of control on the $G_K$-action in \eqref{eq-Galois action}. Specifically, we claim that for all $n \geq 1$, one has
\begin{equation}\label{eq-refined Galois}
	(\tau-1)^n(m) \in \fM \otimes_{\fS_A} u \varphi^{-1}(\mu)^n A_{\operatorname{inf},A}
\end{equation}
whenever $m \in \fM$ and $\tau \in G_K$ is such that $\chi_{\operatorname{cyc}}(\tau) =1$ and $\epsilon(\tau):=\tau(u)/u \in \bZ_p(1)$ is a $\bZ_p$-generator. To see this one first notes that, as explained in e.g. the proof of \cite[Lemma 2.2.8]{B20}, \eqref{eq-refined Galois} cuts out a closed substack of $\cY^{\operatorname{cr}}_{\leq h}$. By \cite[Proposition 20.10]{B23} this closed substack has the same $A$-valued points as $\cY^{\operatorname{cr}}_{\leq h}$ whenever $A$ is finite flat over $\bZ_p$. The characterising property of $\cY^{\operatorname{cr}}_{\leq h}$ therefore ensures this closed substack coincides with $\cY^{\operatorname{cr}}_{\leq h}$.

\begin{rmk}\label{rem-p=2}
	Note that $\tau$ as in \eqref{eq-refined Galois} always exists when $p>2$. This is not always the case when $p=2$, but it is shown in \cite[Lemma 2.1]{Wan22} that $\pi$ can be chosen so that such $\tau$ exist. Thus, when $p=2$ we assume throughout that such a $\pi$ has been chosen.
\end{rmk}

\begin{lemma}\label{lem-lifting points}
    If $A$ is a finite $\bF_p$-algebra and $\fM \in \cY^{\operatorname{cr}}_{\leq h}(A)$ then there exists a finite flat $\bZ_p$-algebra $A^\circ$ with $A = A^\circ\otimes_{\bZ_p}\bF_p$ and $\fM^\circ \in \cY^{\operatorname{cr}}_{\leq h}(A^\circ)$ with $\fM^\circ \otimes_{\bZ_p} \bF_p = \fM$.
\end{lemma}
\begin{proof}
    This follows as in \cite[Lemma 10.9]{B24}.
\end{proof}

\subsection{Breuil--Kisin modules with monodromy}\label{sec-BK with mono}

One difficulty controlling the geometry of $\cY^{\operatorname{cr}}_{\leq h}$ arises from the lack of explicit coordinates on this stack. We address this when $h \leq p-1$ by interpreting the Galois action on crystalline Breuil--Kisin modules in terms of coherent data.

To simplify notation we write $\fM_e := \varphi^{-1}_\fM(E(u)^h\fM) \subset \varphi^*\fM$ whenever $\fM$ is a Breuil--Kisin module of height $\leq h$.
\begin{defn}\label{def-BK+monodromy stack}  Assume $h \leq p$. 
	Let $Y^{\nabla}_{\leq h}$ denote the finite type algebraic stack over $\operatorname{Spec}\bF_p$ whose $A$-points classify rank $d$ Breuil--Kisin modules $\fM$ of height $\leq h$ over $A$ equipped with an $\fS_A$-linear $N_0:\fM/u^{e+1}\fM \rightarrow \fM/u^{e+1}\fM$ with $N_0 \equiv 0$ modulo $u$ and satisfying the following conditions:
	\begin{itemize}
		\item  If $N_0^{\varphi}$ denotes the operator on $\varphi^*\fM/u^{ep+1}\varphi^*\fM
$ given by $m \otimes f \mapsto N_0(m) \otimes f + m \otimes c(u)u\frac{d}{du}(f)$ then $u^{e-1} N_0^\varphi$ stabilises $\fM_e/ u^{ep+e} \varphi^*\fM$.  
		\item  For any operator $N$ on $\fM$ lifting $N_0$ and satisfying the Leibniz rule $N(fm) = fN(m) + m u^{e+1}\frac{d}{du}(f)$ for $m\in \fM,f \in \fS_A$, one has 
		$$
		 u^e N_0^{\varphi} \equiv \varphi_{\fM}^{-1} \circ  c(u)N \circ \varphi_{\fM} \mod u^{e+1}\fM_e		$$
		as operators on $\fM_e /u^{ep+e+1}\varphi^*\fM$. 
	\end{itemize}
    Note $N_0^{\varphi}$ is well defined since $\frac{d}{du}$ kills $\varphi(\fS_{\bF_p})$. We also use that $h\leq p$ to ensure $u^{ep}\varphi^*\fM \subset \fM_e$  so that $u^{e+1}\fM_e/ u^{ep+e+1}\varphi^*\fM$ makes sense. Finally, we point out that the second condition is independent of the chosen lift $N$.
\end{defn}

\begin{rmk}
    To parse the conditions in Definition~\ref{def-BK+monodromy stack} it is helpful to consider $N_0$ as the reduction modulo $u^{e+1}$ of a derivation over $E(u)u\frac{d}{du}$. We would then like to interpret $N_0^\varphi$ as the Frobenius pullback of this derivation and so interpret the second bullet point as a commutation between $N_0$ and Frobenius. However, this does not make sense over $p$-torsion coefficients---such a pullback would necessarily be a derivation over $\frac{\varphi(E(u))}{p} u$. To circumvent this we instead define $N_0^\varphi$ directly, using $c(u)$ as an integral truncation of $\frac{\varphi(E(u))}{p}$.  
\end{rmk}

\begin{rmk}
    See Section~\ref{sec-explicit equations} for an explicit description of the conditions from Definition~\ref{def-BK+monodromy stack} after choosing bases. 
\end{rmk}

The following motivates this construction and is the main result of this section. We emphasize that, while the definition of $Y^{\nabla}_{\leq h}$ makes sense whenever $h \leq p$, we are only able to relate the resulting object to $\cY^{\operatorname{cr}}_{\leq h}$ when $h \leq p-1$.
\begin{thm}\label{thm-embedding}
	Fix $\tau \in G_K$ as in \eqref{eq-refined Galois}; when $p=2$ assume $\pi$ is chosen as in Remark~\ref{rem-p=2}. If $h \leq p-1$ then there is a closed immersion 
	$$
	\cY^{\operatorname{cr}}_{\leq h} \otimes_{\bZ_p} \bF_p \rightarrow  Y^{\nabla}_{\leq h} 
	$$
	of algebraic stacks over $\operatorname{Spec}\bF_p$. On points valued in a finite type $\bF_p$-algebra $A$ this morphism is given by $\fM \mapsto (\fM,N_0)$ with $N_0$ determined by the formula 
	\begin{equation}\label{eq-formula for monodromy}
		N_0(m) \equiv \sum_{n \geq 1} \Bigg( \frac{(\tau-1)^n(m)}{\varphi^{-1}(\mu)^n} \Bigg)  c_n \mod \fM \otimes_{\fS_A} u^{e+1} A_{\operatorname{inf},A}
	\end{equation}
	for a sequence $c_n \in A_{\operatorname{inf}}$ converging $p$-adically to zero and depending only on the choices of $\pi \in K$, $\tau \in G_K$, and $\epsilon_\infty \in \bZ_p(1)$. 
\end{thm}

\begin{rmk}
    In \eqref{eq-formula for monodromy} we are using \eqref{eq-refined Galois} to ensure this sum converges to an element inside $\fM \otimes_{\fS_A} A_{\operatorname{inf},A}$. Part of the content of the theorem is then that this sum is, modulo $u^{e+1}A_{\operatorname{inf},A}$, inside $\fM$.
\end{rmk}
\subsection{Truncated monodromy}\label{sec-creating integral monodromy}
To prepare for the proof of Theorem~\ref{thm-embedding} we consider a finite flat and local $\bZ_p$-algebra $A$ and a Breuil--Kisin module $\fM$ over $A$ equipped with a crystalline $G_K$-action. There is then an associated crystalline representation $T = (\fM \otimes_{\fS} W(C^\flat))^{\varphi=1}$ of $G_K$ on a projective $A$-module (as explained in e.g. \cite[Theorem 2.1.12]{B20}) and an $A$-linear $\varphi$-equivariant comparison
\begin{equation}\label{eq-M and D comparison}
	\fM \otimes_{\fS} \cO^{\operatorname{rig}}[\tfrac{1}{\lambda}] \cong D \otimes_{K_0} \cO^{\operatorname{rig}}[\tfrac{1}{\lambda}]
\end{equation}
where $D = (T \otimes_{\bZ_p} B_{\operatorname{crys}})^{G_K}$ and $\cO^{\operatorname{rig}} \subset K_0[[u]]$ is the subring of series convergent on the open unit disk, containing $\lambda = \prod_{n=0}^\infty \varphi^n(\frac{E(u)}{E(0)})$. This allows us to define a derivation $N_\nabla$ on $\fM \otimes_{\fS} \cO^{\operatorname{rig}}[\tfrac{1}{\lambda}]$ via the formula $N_\nabla(d \otimes f) = d  \otimes E(u) u\frac{d}{du}(f)$ whenever $d \in D, f \in \cO^{\operatorname{rig}}[\frac{1}{\lambda}] \otimes_{\bZ_p} A$. We also write $\varphi^* N_\nabla$ for the derivation on $\varphi^*\fM \otimes_{\fS} \cO^{\operatorname{rig}}[\tfrac{1}{\varphi(\lambda)}]$ given by 
\begin{equation}\label{eq-extending monodromy}
m \otimes f \mapsto N_\nabla(m) \otimes f  + m \otimes \tfrac{\varphi(E(u))}{p} u \tfrac{d}{du}(f)	
\end{equation}
 which satisfies the relation
		\begin{equation}\label{eq-Nphi = p phi N}
		E(u) \varphi^* N_\nabla = \varphi_{\fM}^{-1} \circ \tfrac{\varphi(E(u))}{p}N_\nabla \circ \varphi_{\fM}
		\end{equation}
		when viewed as operators on $\varphi^*\fM \otimes_{\fS} \mathcal{O}^{\operatorname{rig}}[\frac{1}{\varphi(\lambda)}]$. This follows from the definition of $N_\nabla$ and the observation that $\varphi \circ  E(u)u\frac{d}{du} =\frac{\varphi(E(u))}{p}u\frac{d}{du} \circ \varphi$ in $K_0[[u]]$. 
        
        All the above is essentially formal, with such operators existing for any Breuil--Kisin over $A$ (see \cite[\S1]{Kis05}). However, we also need two properties which are specific to the crystalline situation. Firstly,  $N_\nabla$ stabilises $\fM \otimes_{\fS_A} \cO^{\operatorname{rig}}[\tfrac{1}{\varphi(\lambda)}]$, see \cite[Corollary 1.3.15]{Kis05}.  Given \eqref{eq-Nphi = p phi N} and the fact that $\varphi^*N_\nabla \equiv 0$ modulo $u$, this stability is equivalent to 
\begin{equation}\label{eq-Griffiths trans}
E(u)\varphi^*N_\nabla( \fM_e) \subset u\fM_e \otimes_{\fS} \mathcal{O}^{\operatorname{rig}}[\tfrac{1}{\varphi(\lambda)}]
\end{equation}
for $\fM_e := \varphi^{-1}_\fM(E(u)^h\fM)$. This holds for any Breuil--Kisin module associated to a $G_{K_\infty}$-stable $\bZ_p$-lattice inside a crystalline representation. Secondly, as $\fM$ furthermore comes from a $G_K$-stable $\bZ_p$-lattice inside a crystalline representation, $N_\nabla$ enjoys the following integrality:
\begin{equation}\label{eq-mono in Smax}
    N_\nabla(m) \in \fM \otimes_{\fS} uS_{\operatorname{max}}
\end{equation}
for all $m \in \fM$ and where $S_{\operatorname{max}} = W(k)[[u,\frac{u^e}{p}]]\cap \cO^{\operatorname{rig}}[\frac{1}{\varphi(\lambda)}]$.
This follows from \cite[Proposition 20.1]{B23}, bearing in mind that what is here called $N_\nabla$ is denoted $E(u)N_\nabla$ in loc.\ cit., as well as the fact $E(u) \in p S_{\operatorname{max}}$.

\begin{lemma}\label{lem-Smax containment}
    We have the following containments:
    \begin{enumerate}
        \item $S_{\operatorname{max}} \subset \frac{u^e}{p}S_{\operatorname{max}} + \fS$. 
        \item $\frac{u^{ep}}{p}S_{\operatorname{max}} \subset \frac{E(u)^p}{p} S_{\operatorname{max}} + u^{e(p-1)} \fS + p\fS$.
        \item $(\frac{u^e}{p} S_{\operatorname{max}} + E(u)S_{\operatorname{max}}) \cap \fS = p\fS + u^e \fS$
    \end{enumerate}
\end{lemma}
\begin{proof}
        First, consider the statements with $S_{\operatorname{max}}$ replaced by the subring $A \subset K_0[[u]]$ consisting of series $\sum_{i \geq 0} \big( \frac{u^e}{p} \big)^i f_i(u)$ with $f_i \in W(k)[u]$ converging $p$-adically to zero. Part (1) is then clear, and part (2) reduces to finding, for each $n \geq p$, elements $f_1 \in A$ and $f_2,f_3 \in W(k)[u]$ so that $p^{p-1} \big( \frac{u^{e}}{p} \big)^n = \frac{E(u)^p}{p}f_1 + u^{e(p-1)}f_2 + pf_3$. For this, express $u^e = E(u) + pd(u)$ for some $d(u) \in W(k)[u]$, and so expand
    $$
    p^{p-1} \big(\frac{u^e}{p} \big)^n = \sum_{i=0}^n \big( \frac{E(u)}{p} \big)^i  p^{p-1} d_i(u) 
    $$
    for some $d_i(u) \in W(k)[u]$. The $i$-th term in the sum lies in $\frac{E(u)^p}{p}A$ whenever $i \geq p$, in $pW(k)[u]$ whenever $i < p-1$, and in $u^{e(p-1)} W(k)[u]+ pW(k)[u]$ whenever $i=p-1$. For part (3), note that (1) asserts $A/\frac{u^e}{p}A \cong \fS/u^e\fS$. If $f =\sum_{i \geq 0} (\frac{u}{p})^i f_i(u)$ with $f_i \in W(k)[u]$ of degree $< e$ then $A \rightarrow A/\frac{u^e}{p}A \cong \fS/u^e\fS$ sends $f$ onto $f_0(u)$ mod $u^e\fS$. It follows that the image of $E(u)A$ under this composite is the ideal of $\fS/u^e\fS$ generated by $p$. Consequently, the kernel of $\fS \rightarrow A \rightarrow A/(\frac{u^e}{p},E(u))A$ is the ideal of $\fS$ generated by $p$ and $u^e$, as required.

    Clearly $S_{\operatorname{max}} \subset A$ and so (3) in the lemma follows immediately from the discussion above. To deduce the (1) and (2) in the lemma from the above we need to show that if $f\in A$ and either $\frac{u^e}{p}f \in \mathcal{O}^{\operatorname{rig}}$ or $\frac{E(u)^p}{p}f \in \cO^{\operatorname{rig}}$ then $f \in \cO^{\operatorname{rig}}$. This is clear if $\frac{u^e}{p}f \in \mathcal{O}^{\operatorname{rig}}$ as $\sum_{i\geq 0} f_i u^i \in \cO^{\operatorname{rig}}$ if and only if $v_p(f_i) + ri \rightarrow \infty$ for all $r >0$. If $\frac{E(u)^p}{p}f \in \cO^{\operatorname{rig}}$ then, after extending scalars so that $E(u)$ splits into linear factors, we can argue inductively and replace $E(u)$ by one such factor $u-\pi$. One then argues just as in the first case, the key point being that the assumption $f\in A$ ensures the existence of an $u-\pi$-adic expansion $f = \sum_{i \geq 0} f_i(u-\pi)^i$. Since $v_p(\pi)>0$, the same coefficient growth criterion for $\mathcal{O}^{\operatorname{rig}}$ applies to such $u-\pi$-adic expansions. 
\end{proof}

The main consequence of Lemma~\ref{lem-Smax containment} is that $N_\nabla$ and $\varphi^*N_\nabla$ respectively admit $u$-adic and $E(u)$-adic truncations which are integral, and relate to each other modulo $p$:

\begin{prop}\label{prop-create monodromies}
	There exists:
	\begin{enumerate}
		\item An operator $N$ on $\fM$ with $N\equiv 0$ modulo $u\fM$ and satisfying $N(fm) = fN(m) + mE(u)u\frac{d}{du}(f)$ for $f \in \fS_A, m\in \fM$ and  $N \equiv N_\nabla$ modulo $\fM \otimes_{\fS} \tfrac{u^{e+1}}{p}S_{\operatorname{max}}$.
		\item An operator $N^\varphi$ on $\varphi^*\fM$ satisfying $N^\varphi(fm) = fN^\varphi(m) + mc^*(u)u\frac{d}{du}(f)$ for $f \in \fS_A, m\in \varphi^*\fM$ and with $N^\varphi \equiv \varphi^*N_\nabla$ modulo $\varphi^*\fM \otimes_{\fS} u\tfrac{E(u)^p}{p}S_{\operatorname{max}}$.
	\end{enumerate}
    Furthermore, $N$ and $N^\varphi$ can be chosen so that 
        $$
        \varphi^*N \equiv N^\varphi \mod \tfrac{u^{ep+1}}{p}\varphi^*\fM + u^{e(p-1)+p} \varphi^*\fM + pu\varphi^*\fM
        $$
        for $\varphi^*N$ defined as in \eqref{eq-extending monodromy}.\footnote{The $\frac{1}{p}$ appears in this congruence because the formation of $\varphi^*N$ necessarily introduces denominators due to the appearance of $\frac{\varphi(E(u))}{p}$.}
\end{prop}

\begin{proof}
    Since $A$ is local $\fM$ admits an $\fS_A$-basis $\beta$. Combining \eqref{eq-mono in Smax} with part (1) of Lemma~\ref{lem-Smax containment} ensures we can choose a tuple $N(\beta)$ in $\fM$ so that $N(\beta) \equiv N_{\nabla}(\beta)$ modulo $\fM \otimes\frac{u^{e+1}}{p}S_{\operatorname{max}}$. Extending $N$ to $\fM$ via the Leibniz rule in (1) then produces an operator $N$ as claimed. To construct $N^\varphi$ note that, due to the choice of $N$, we have
    $$
    \varphi^*N(\varphi^*\beta) \equiv \varphi^*N_\nabla(\varphi^*\beta) \mod \tfrac{u^{ep+p}}{p} \varphi^*\fM \otimes_{\fS} S_{\operatorname{max}}
    $$
    Using part (2) of Lemma~\ref{lem-Smax containment}, we can actually find tuples $A,B,C$ so that 
    $$
    \varphi^*N(\varphi^*\beta) = \varphi^*N_\nabla(\varphi^*\beta) + u^p\tfrac{E(u)^p}{p} A + u^{e(p-1) + p}B + u^p p C
    $$
    with $A \in \varphi^*\fM \otimes_{\fS} S_{\operatorname{max}}$ and $B,C \in \varphi^*\fM$. We then consider the tuple in $\varphi^*\fM$ given by $
    N^\varphi(\varphi^*\beta) := \varphi^*N(\varphi^*\beta) - \big(  u^{e(p-1) + p}B + u^p p C \big)$. Extending $N^\varphi$ to an operator on $\varphi^*\fM$ using the Leibniz rule in (2) yields a derivation with $N^\varphi \equiv \varphi^*N_\nabla$ modulo $u\frac{E(u)^p}{p}\varphi^*\fM \otimes_{\fS} S_{\operatorname{max}}$; here we are using that $uc^*(u) \equiv u\frac{\varphi(E(u))}{p}$ modulo $u\frac{E(u)^p}{p}S_{\operatorname{max}}$. Since $uc^*(u) - u\frac{\varphi(E(u))}{p} \in \frac{u^{ep+1}}{p}\fS + pu\fS$ we also obtain $N^\varphi \equiv \varphi^*N$ modulo $\frac{u^{ep+1}}{p}\varphi^*\fM + u^{e(p-1)+p}\varphi^*\fM + pu \varphi^*\fM$.
    \end{proof}

The next proposition describes relations between $N$ and $N^\varphi$ which, after reducing modulo $p$, match the conditions imposed in Definition~\ref{def-BK+monodromy stack}.

\begin{prop}\label{prop-respecting conv and equality at last step}
	Let $N$ and $N^\varphi$ be derivations as in Proposition~\ref{prop-create monodromies}. If $\fM$ has height $h \leq p$ then 
    $$
    E(u)N^\varphi(\fM_e) \subset u\fM_e
    $$
    for $\fM_e = \varphi_{\fM}^{-1}(E(u)^h \fM)$ as defined above. If $\fM$ has height $h \leq p-1$ then, as operators on $\fM_e$, we furthermore have
	$$
	E(u)N^\varphi \equiv \varphi_{\fM}^{-1} \circ c(u)N \circ \varphi_\fM 
	$$
    modulo $u^{e+1}\fM_e + pu\fM_e$.
\end{prop}

\begin{proof}
		As operators on $\varphi^*\fM$ we have $E(u)N^\varphi \equiv E(u)\varphi^* N_\nabla$ modulo 
		$$
		\varphi^*\fM \otimes_{\fS} u\tfrac{E(u)^{p+1}}{p}S_{\operatorname{max}} \subset \varphi^*\fM \otimes_{\fS} uE(u)^p S_{\operatorname{max}} \subset  \begin{cases}
		    \fM_e \otimes_{\fS} uE(u)S_{\operatorname{max}} & \text{ if $h \leq p-1$} \\
            \fM_e \otimes_{\fS} uS_{\operatorname{max}} & \text{ if $h \leq p$}
		\end{cases}
		$$
		(the first inclusion following since $E(u) \in pS_{\operatorname{max}}$ and the second since $E(u)^h \varphi^*\fM \subset \fM_e$). This combined with \eqref{eq-Griffiths trans}  shows $E(u)N^\varphi(\fM_e) \subset u\fM_e$ whenever $h\leq p$. Also, since $	N \equiv N_\nabla$ modulo $\fM \otimes_{\fS}  \tfrac{u^{e+1}}{p}S_{\operatorname{max}}$, it follows that, as operators on $E(u)^h\fM$, we have
	$$
	\tfrac{\varphi(E(u))}{p} N_\nabla \equiv c(u) N_\nabla \equiv c(u)N 
    \mod E(u)^h \fM \otimes_{\fS} \tfrac{u^{e+1}}{p} S_{\operatorname{max}}
	$$
	If $h \leq p-1$ then we just saw that $E(u)N^\varphi \equiv E(u) \varphi^* N_\nabla$ modulo $\fM_e\otimes_{\fS} uE(u)S_{\operatorname{max}}$ as operators on $\varphi^*\fM$. Therefore, \eqref{eq-Nphi = p phi N} implies that $E(u)N^\varphi \equiv \varphi_{\fM}^{-1} \circ c(u)N \circ \varphi_{\fM} $ modulo 
    $$
    \left( \fM_e \otimes_{\fS} \tfrac{u^{e+1}}{p} S_{\operatorname{max}} + \fM_e \otimes_{\fS} u E(u)S_{\operatorname{max}} \right) 
    $$
    Since $\fM_e$ is $\fS$-free,  part (3) of Lemma~\ref{lem-Smax containment} implies this congruence holds modulo $u^{e+1}\fM_e + pu\fM_e$, which finishes the proof.
\end{proof}

\begin{cor}\label{cor-congruences mod p}
    Assume $\fM$ has height $ h \leq p-1$ and set $\overline{\fM} = \fM \otimes_{\bZ_p} \bF_p$. For a derivation $N$ on $\fM$ as in Proposition~\ref{prop-create monodromies}, set $\overline{N} = N \otimes_{\bZ_p} \bF_p$ and $\overline{N}_0 = \overline{N}$  modulo $u^{e+1}\overline{\fM}$. Then:
    \begin{itemize}
        \item If $\overline{N}_0^\varphi$ is the derivation on $\varphi^*\overline{\fM}/ u^{ep+1}\varphi^*\overline{\fM}$ formed as in Definition~\ref{def-BK+monodromy stack} then $u^{e-1}\overline{N}_0^\varphi$ stabilises $\overline{\fM}_e/ u^{ep+e}\varphi^*\overline{\fM}$ and 
        $$
        u^e \overline{N}_0^\varphi \equiv \varphi^{-1}_{\overline{\fM}} \circ c(u) \overline{N} \circ \varphi_{\overline{\fM}} \mod u^{e+1}\overline{\fM}_e
        $$
         In particular, $\overline{\fM}$ and $\overline{N}_0$ define an $A \otimes_{\bZ_p} \bF_p$-valued point of $Y^\nabla_{\leq h}$.
       
    \end{itemize}
    Furthermore, if $N^\varphi$ is a derivation on $\varphi^*\fM$ as in Proposition~\ref{prop-create monodromies} then 
    $$
    N^\varphi \equiv \overline{N}_0^\varphi \mod pu\varphi^*\fM + u^{e(p-1)+p}\varphi^*\fM
    $$
    as derivations on $\varphi^*\overline{\fM}/ u^{ep+1}\varphi^*\overline{\fM}$.
\end{cor}
\begin{proof}
    We begin with the last part, which follows from the last part of Proposition~\ref{prop-create monodromies}. More precisely, it follows from the proof, which asserts that for a choice of basis $\beta$ one has $N^\varphi(\varphi^*\beta) \equiv \varphi^*N(\varphi^*\beta)$ modulo $u^{e(p-1)+p} \varphi^*\fM + pu^p\varphi^*\fM$. This gives the claim since $uc^*(u) \equiv uc(u)$ modulo $pu\fS$. To prove the first part, we combine what we've just shown with Proposition~\ref{prop-respecting conv and equality at last step}.  Since $h \leq p-1$ this gives 
    $$
    E(u)\overline{N}_0^\varphi(\overline{\fM}_e) \subset u\overline{\fM}_e + u^{e(p-1)+p+e}\varphi^*\overline{\fM}
    $$
    (this containment really occurring in $\varphi^*\overline{\fM} / u^{ep+1+e}\varphi^*\overline{\fM}$). As $h \leq p-1$ we have $u^{e(p-1)+p+e}\varphi^*\overline{\fM} \subset u^{e+1}\overline{\fM}_e$; thus dividing the displayed containment by $u$ gives the required stability. Likewise, the second part of Proposition~\ref{prop-respecting conv and equality at last step} ensures 
    $$
    u^e\overline{N}_0^\varphi \equiv \varphi_{\overline{\fM}}^{-1} \circ c(u) \overline{N} \circ \varphi_{\overline{\fM}}
    $$
    modulo $u^{e+1}\overline{\fM}_e + u^{e(p-1)+p+e}\varphi^*\overline{\fM} + up\varphi^*\overline{\fM}$. Again using that  $u^{e(p-1)+p+e}\varphi^*\overline{\fM} \subset u^{e+1}\overline{\fM}_e$ gives the claimed relation.
\end{proof}

\subsection{Galois and monodromy}
Here we show how the operator $N$ from Proposition~\ref{prop-create monodromies} can be interpreted in terms of the crystalline $G_K$-action on $\fM$. This is really just an unravelling of the proof of \ref{eq-mono in Smax} from \cite[\S20]{B23}.  Reducing this description modulo $p$ will give the formula for $N_0$ in Theorem~\ref{thm-embedding}. The key property we need is that, after base-changing to $B_{\operatorname{dR}}$, \eqref{eq-M and D comparison} becomes $G_K$-equivariant  for the trivial $G_K$-action on $D$. From this one can deduce that
\begin{equation}\label{eq-match Galois and monodromy}
	N_\nabla(m) = \frac{E(u)}{\operatorname{log}([\epsilon(\sigma)])}\sum_{n \geq 1} (-1)^{n+1} \frac{(\sigma-1)^n(m)}{n}   , \qquad (\sigma -1)(m) = \sum_{n \geq 1} \left( \frac{N_\nabla}{E(u)} \right) ^n(m) \otimes \frac{\operatorname{log}([\epsilon(\sigma)])^n}{n!}
\end{equation}
whenever $m \in \fM$ and $\sigma \in G_K$. See \cite[Lemma 20.9]{B23}, but note again that  $N_\nabla$ defined here is $E(u)N_\nabla$ for $N_\nabla$ from  loc.\ cit.

\begin{lemma}\label{lem-monodromy via Galois}
	Let $N$ be a derivation as in Proposition~\ref{prop-create monodromies}.
	\begin{enumerate}
		\item Fix $\tau \in G_K$ as in Theorem~\ref{thm-embedding}. Then there are $c_n \in A_{\operatorname{inf}}$ converging to zero for the $p$-adic topology on $A_{\operatorname{inf}}$ and not depending on $\fM$,  so that 
		$$
		N(m) \equiv \sum_{n \geq 1} \left( \frac{(\tau-1)^n(m)}{\varphi^{-1}(\mu)^n} \right) c_n \mod  \mathfrak{M} \otimes_{\fS_A} u^{e+1}A_{\operatorname{inf},A}
		$$
		for $m \in \fM$.
		\item For each $\sigma \in G_K$ there exists a sequence $d_n(\sigma) \in \varphi^{-1}(\mu) A_{\operatorname{inf}}$ converging to zero for the $p$-adic topology on $A_{\operatorname{inf}}$ and not depending on $\fM$, so that 
        $$
        (\sigma-1)(m) \equiv \sum_{n \geq 1} N^n(m)d_n(\sigma) \mod \fM \otimes_{\fS_A}u^{e+1}\varphi^{-1}(\mu) A_{\operatorname{inf},A}
        $$
        for $m \in \fM$.
	\end{enumerate}
\end{lemma}
\begin{proof}
	From \eqref{eq-match Galois and monodromy} we find
	$$
	N_\nabla(m) =  \sum_{n \geq 1} \left( \frac{(\tau-1)^n(m)}{\varphi^{-1}(\mu)^n} \right) C_n
	$$
	for $C_n = (-1)^{n+1}\frac{E(u)\varphi^{-1}(\mu)^n}{\operatorname{log}([\epsilon(\tau)])n}$. The proof of \cite[Proposition 20.4]{B23} shows that $C_n$ lies in the subring $A_{\operatorname{max}} \subset B_{\operatorname{dR}}$ consisting of sums $\sum_{n \geq 0} a_n \left( \frac{E(u)}{p} \right)^n$ with $a_n \rightarrow 0$ for the $(p,u)$-adic topology on $A_{\operatorname{inf}}$, and converges to zero $p$-adically. Now, $A_{\operatorname{inf}}$ surjects onto $A_{\operatorname{max}}/\tfrac{u^e}{p}$ and so we can find $c_n \in A_{\operatorname{inf}}$ converging to zero $p$-adically and with $c_n \equiv C_n$ modulo $\frac{u^e}{p}A_{\operatorname{max}}$. We know, from \eqref{eq-refined Galois}, that $\frac{(\tau-1)^n}{\varphi^{-1}(\mu)^n}(m) \in  \fM \otimes_{\fS} uA_{\operatorname{inf}}$. Therefore,  $N(m) - \sum_{n \geq 1} \left( \frac{(\tau-1)^n(m)}{\varphi^{-1}(\mu)^n} \right) c_n \in \fM \otimes_{\fS} \frac{u^{e+1}}{p}A_{\operatorname{max}}$ whenever $m \in \fM$. Part (1) then follows since $\frac{u^{e+1}}{p}A_{\operatorname{max}} \cap A_{\operatorname{inf}} = u^{e+1}A_{\operatorname{inf}}$ (see the last two sentences in the proof of \cite[Proposition 20.10]{B23}).
	
	For (2), an induction shows there are $d_{n,i} \in W(k)[u]$ so that $E(u)^n( \frac{N_\nabla}{E(u)})^n(m) = N_\nabla^n(m) + \sum_{i <n} d_{n,i}N_\nabla^i(m)$. We can therefore express the second identity from \eqref{eq-match Galois and monodromy} as 
	$$
	(\sigma-1)(m) = \sum_{n \geq 1} N_\nabla^n(m) D_n(\sigma)
	$$
	where $D_n(\sigma)$ is a $W(k)[u]$-linear combination of $\frac{1}{i!}\left( \frac{\operatorname{log}([\epsilon(\sigma)])}{E(u)} \right)^i$ for $i \geq 1$. Observe $E(u) \varphi^{-1}(\mu)$ divides $\operatorname{log}([\epsilon(\sigma)])$ in $A_{\operatorname{max}}$. Also, the argument from the proof of \cite[Proposition 20.4]{B23} showing $C_n \in A_{\operatorname{max}}$ implies more generally that $\frac{\varphi^{-1}(\mu)^{i-1}}{i!} \in A_{\operatorname{max}}$ for any $i \geq 1$. Thus, $D_n(\sigma) \in \varphi^{-1}(\mu)A_{\operatorname{max}}$ for each $n \geq 1$. Since $N^l(m) \equiv N_\nabla^l(m)$ modulo $\fM \otimes_{\fS} \frac{u^{e+1}}{p}A_{\operatorname{max}}$ for each $l \geq 1$ and $N \equiv 0$ modulo $u\fM$ it follows that
	$$
	(\sigma-1)(m) -  \sum_{n \geq 1} N^n(m)  d_n(\sigma) \in \fM \otimes_{\fS} \tfrac{u^{e+1}}{p}\varphi^{-1}(\mu)A_{\operatorname{max}}
	$$
	where $d_n(\sigma) \in A_{\operatorname{inf}}$ is $\varphi^{-1}(\mu)$ multiplied by a lift of $\frac{D_n(\sigma)}{\varphi^{-1}(\mu)}$ to $A_{\operatorname{inf}}$ modulo $\frac{u^e}{p}A_{\operatorname{max}}$. As $A_{\operatorname{inf}} \cap \frac{u^{e+1}}{p}A_{\operatorname{max}} = u^{e+1} A_{\operatorname{inf}}$ the left hand side actually lies in $\fM \otimes_{\fS} u^{e+1}\varphi^{-1}(\mu)A_{\operatorname{inf}}$ which proves (2).
\end{proof}
\subsection{Proof of Theorem~\ref{thm-embedding}}\label{sec-finish proof}

Fix a sequence $c_n \rightarrow 0$ in $A_{\operatorname{inf}}$ as in part (1) of Lemma~\ref{lem-monodromy via Galois}. The starting point is to show that the formula
\begin{equation}\label{eq-identity for monodromy}
    \sum_{n \geq 1} \Bigg( \frac{(\tau-1)^n(m)}{\varphi^{-1}(\mu)^n} \Bigg) c_n \mod u^{e+1} \fM \otimes_{\fS_A} A_{\operatorname{inf},A}
\end{equation}
for $m \in \fM$ defines an $\fS_A$-linear endomorphism of $\fM / u^{e+1}\fM$ whenever $\fM \in \cY^{\operatorname{cr}}_{\leq h} \otimes_{\bZ_p} \bF_p$. When $\fM$ is defined over a finite $\bF_p$-algebra, this follows by combining Lemma~\ref{lem-lifting points} and Lemma~\ref{lem-monodromy via Galois}. However, an additional argument is required to handle general points.

\begin{lemma}\label{lem-isom}
    Assume $h \leq p$ and $\tau \in G_K$ is as in Theorem~\ref{thm-embedding}. Then $\cY^{\operatorname{cr}}_{\leq h} \otimes_{\bZ_p} \bF_p$ is isomorphic to the algebraic stack over $\operatorname{Spec}\bF_p$ whose $A$-points, for any finite type $\bF_p$-algebra $A$, classifies pairs $(\fM,N_0)$ with $\fM \in \cY^{\operatorname{cr}}_{\leq h}(A)$ and $N_0$ an $\fS_A$-linear endomorphism of $\fM/u^{e+1}\fM$ so that \eqref{eq-identity for monodromy} holds. Specifically, this isomorphism is given by forgetting the endomorphism $N_0$.
\end{lemma}
\begin{proof}
    Let $Z^{\operatorname{cr}}$ temporarily denote the stack of pairs $(\fM, N_0)$ described in the lemma.  We claim that $Z^{\operatorname{cr}}$ is a closed substack of the finite type algebraic stack over $\operatorname{Spec}\bF_p$ whose points valued in a finite type $\bF_p$-algebra $A$ classify $\fM \in \cY^{\operatorname{cr}}_{\leq h}(A)$ together with an $\fS_A$-linear endomorphism $N_0$ of $\fM/u^{e+1}\fM$.
 To see this recall, e.g.\ from \cite[\S B.24]{EGstack}, that if $A$ is any finite type $\bF_p$-algebra and $x \in A_{\operatorname{inf},A}$ then the condition that $x \in u^{e+1} A_{\operatorname{inf},A}$ is closed in $\operatorname{Spec}A$. 

Forgetting the endomorphism $N_0$ therefore gives a morphism $Z^{\operatorname{cr}} \rightarrow \cY^{\operatorname{cr}}_{\leq h} \otimes_{\bZ_p} \bF_p$ of algebraic stacks which are finite type over $\operatorname{Spec}\bF_p$. To show this is an isomorphism it suffices, by \cite[Lemma 15.1]{B24}, to show this morphism induces an equivalence on points valued in a finite local $\bF_p$-algebra $A$. For essential surjectivity, take $\fM \in \cY^{\operatorname{cr}}_{\leq h}(A)$ and apply Lemma~\ref{lem-lifting points} to produce $\fM^\circ \in \cY^{\operatorname{cr}}_{\leq h}(A^\circ)$ with $A^\circ$ finite flat over $\bZ_p$. Applying Proposition~\ref{prop-create monodromies} to $\fM^\circ$ produces an operator $N^\circ$ on $\fM^\circ$ whose reduction modulo $p\fS_A + u^{e+1}\fS_A$ is an $\fS_A$-linear endomorphism of $\fM/u^{e+1}\fM$. Reducing the formula for $N^\circ$ given in Lemma~\ref{lem-monodromy via Galois} shows $(\fM,N_0)$ lies in the image of $Z^{\operatorname{cr}}(A) \rightarrow \cY^{\operatorname{cr}}_{\leq h}(A)$. For full-faithfulness, take $(\fM,N_0) \in Z^{\operatorname{cr}}(A)$. The formula defining $N_0$ in the definition of $Z^{\operatorname{cr}}$ shows $N_0$ is completely determined by the $G_K$-action on $\fM \otimes_{\fS_{A}} A_{\operatorname{inf},A}/u^{e+1}\varphi^{-1}(\mu)A_{\operatorname{inf},A}$. This ensures faithfullness. That it is full is a consequence of the uniqueness in Lemma~\ref{lem-Galois mod u^e extends}, which finishes the proof.
\end{proof}

\begin{lemma}\label{lem-Galois mod u^e extends}
	Let $A$ be any finite type $\bF_p$-algebra and $\fM$ a Breuil--Kisin module over $A$ of height $\leq p$. Suppose $\sigma_0$ denotes an $A_{\operatorname{inf},A}$-semilinear, continuous action of $G_K$ on $\fM \otimes_{\fS_A} A_{\operatorname{inf},A} / u^{e+1}\varphi^{-1}(\mu)A_{\operatorname{inf},A}$ satisfying 
        $$
        \varphi_{\fM} \circ (\widetilde{\sigma}_0 -1) \circ \varphi^{-1}_{\fM} \equiv \widetilde{\sigma}_0 -1 \mod \fM \otimes_{\fS_A} u^{e+1}\varphi^{-1}(\mu)A_{\operatorname{inf},A}
        $$
        for any lift $\widetilde{\sigma}_0$ of $\sigma_0$ to an operator on $\fM \otimes_{\fS_A} A_{\operatorname{inf},A}$. Here $\varphi_{\fM}$ is interpreted as a semilinear operator on $\fM$.
    Then there exists a unique $A_{\operatorname{inf},A}$-semilinear, continuous action $\sigma$ of $G_K$ on $\fM \otimes_{\fS_A} A_{\operatorname{inf},A}$ which is $\varphi$-equivariant and satisfies $\sigma \equiv \sigma_0$ modulo $\fM \otimes_{\fS_A} u^{e+1}\varphi^{-1}(\mu)A_{\operatorname{inf},A}$.
\end{lemma}
\begin{proof}
	Let $Q$ be any $A$-linear endomorphism of $\fM \otimes_{\fS_A} A_{\operatorname{inf},A}$ with image contained in $\fM \otimes_{\fS_A} u^{e+1} \varphi^{-1}(\mu)A_{\operatorname{inf},A}$. Since $A$ is an $\bF_p$-algebra, $\varphi^{-1}(\mu)$ and $u^{e/(p-1)}$ generate the same ideal of $A_{\operatorname{inf},A}$, and so this is the same as $Q$ having image in $\fM \otimes_{\fS_A} u^{e+1 + e/(p-1)}A_{\operatorname{inf},A}$. Thus, $\varphi \circ Q \circ \varphi^{-1}$ maps $u^{eh} \fM$ into $\fM \otimes_{\fS_A} u^{pe +p + pe/(p-1)} A_{\operatorname{inf}}$. If $h \leq p$ then $\varphi \circ Q \circ \varphi^{-1}$ is an $A$-linear endomorphism of $\fM \otimes_{\fS_A} A_{\operatorname{inf},A}$ with image contained in $\fM \otimes_{\fS_A} u^{e + e/(p-1) + p}A_{\operatorname{inf},A}$. It follows that the action of $Q \mapsto \varphi \circ Q \circ \varphi^{-1}$ on such endomorphisms is topologically nilpotent.

    Applying this with $Q = \varphi_{\fM}\circ (\widetilde{\sigma}_0 -1) \circ \varphi_{\fM}^{-1} - (\widetilde{\sigma}_0 -1)$ shows that, for each $\sigma \in G_K$,
    $$
    \sigma := \widetilde{\sigma}_0 + \sum_{n \geq 0} \varphi_{\fM}^n \circ Q \circ \varphi^{-n}_{\fM} 
    $$
    defines a $\varphi$-equivariant $A_{\operatorname{inf},A}$-semilinear operator on $\fM \otimes_{\fS_A} A_{\operatorname{inf},A}$ varying continuously with $\sigma \in G_K$ and lifting $\sigma_0$. This operator is unique since if $\sigma_1,\sigma_2$ are two such lifts of $\sigma_0$ then taking $Q = \sigma_1 -\sigma_2$ shows $\sigma_1-\sigma_2 = \varphi^n_{\fM} \circ (\sigma_1-\sigma_2) \circ \varphi^{-n}_{\fM}$ goes to zero as $n \rightarrow \infty$. This uniqueness also shows that the operators $\sigma$ define an action of $G_K$ on $\fM \otimes_{\fS_A} A_{\operatorname{inf},A}$, which finishes the proof.
\end{proof}

\begin{proof}[Proof of Theorem~\ref{thm-embedding}]
    
Using Lemma~\ref{lem-isom} we obtain a morphism 
\begin{equation}\label{eq-temp map}
\cY^{\operatorname{cr}}_{\leq h} \otimes_{\bZ_p} \bF_p \rightarrow Z_{\leq h}, \qquad \fM \mapsto (\fM,N_0)
\end{equation}
where $Z_{\leq h}$ denotes the finite type algebraic stack over $\operatorname{Spec}\bF_p$ whose $A$-points classify Breuil--Kisin modules over $A$ of height $\leq h$ together with an $\fS_A$-linear endomorphism of $\fM/ u^{e+1}\fM$. This is a monomorphism, as follows by combining part (1) of \cite[Lemma 15.1]{B24} with the full faithfulness argument in the last part of Lemma~\ref{lem-isom}. So far we have only used the bound $h \leq p$. 

Clearly, $Y^{\nabla}_{\leq h}$ appears as a closed substack of $Z_{\leq h}$ and so the existence of a monomorphism as in Theorem~\ref{thm-embedding} can be deduced by showing \eqref{eq-temp map} factors through this closed substack. By \cite[Corollary 15.2]{B24} it suffices to do this for points valued in a finite local $\bF_p$-algebra. This factorisation therefore follows from the lifting process employed in the proof of Lemma~\ref{lem-isom} and applying Corollary~\ref{cor-congruences mod p}. This is where we need $h\leq p-1$.

In order to complete the proof of Theorem~\ref{thm-embedding} we need to show this morphism is not just a monomorphism but a closed immersion. For this choose, for each $\sigma \in G_K$, sequences $d_n(\sigma) \in \varphi^{-1}(\mu)A_{\operatorname{inf}}$ as in Lemma~\ref{lem-monodromy via Galois} and, for each $\fM \in Y^{\nabla}_{\leq h}(A)$, consider the $A_{\operatorname{inf},A}$-semilinear operators on $\fM \otimes_{\fS_A} A_{\operatorname{inf},A} / u^{e+1}\varphi^{-1}(\mu)A_{\operatorname{inf},A}$ defined by
$$
        (\sigma_0-1)(m) = \sum_{n \geq 1} N^n(m)d_n(\sigma) \mod \fM \otimes_{\fS_A}u^{e+1}\varphi^{-1}(\mu) A_{\operatorname{inf},A}
        $$
        for a lift $N$ of $N_0$ to a derivation on $\fM$ (note $\sigma_0$ does not depend on this lift since $d_n(\sigma) \in \varphi^{-1}(\mu)A_{\operatorname{inf}}$).
        We can then consider the locus $Q \subset Y^{\nabla}_{\leq h}$ where this $\sigma_0$ defines an action of the group $G_K$ and satisfies the hypothesis in Lemma~\ref{lem-Galois mod u^e extends}. This locus is closed by the same fact used in the proof of Lemma~\ref{lem-isom}: namely, if $x \in A_{\operatorname{inf},A}$ then $x \in u^{e+1}A_{\operatorname{inf},A}$ defines a closed condition on $\operatorname{Spec}A$ whenever $A$ is a finite type $\bF_p$-algebra.
        \begin{rmk}
    It seems likely that the inclusion $Q \subset Y^{\nabla}_{\leq h}$ is an equality, but we have not tried to prove this.
\end{rmk}

Using Lemma~\ref{lem-Galois mod u^e extends} we obtain a monomorphism $Q \rightarrow \cY_{\leq h} \otimes_{\bZ_p} \bF_p$ (recall from Definition~\ref{def-BKcryststack} that $\cY_{\leq h}$ parametrises Breuil--Kisin modules of height $\leq h$ together with a crystalline $G_K$-action). On the other hand, the  argument from the previous paragraph shows that $\cY^{\operatorname{cr}}_{\leq h} \otimes_{\bZ_p} \bF_p \rightarrow Y^{\nabla}_{\leq h}$ factors through $Q$. The uniqueness in Lemma~\ref{lem-Galois mod u^e extends} then shows that the composite
        $$
        \cY^{\operatorname{cr}}_{\leq h} \otimes_{\bZ_p} \bF_p \rightarrow Q \rightarrow \cY_{\leq h} \otimes_{\bZ_p} \bF_p
        $$
        is just the base change of the closed immersion $\cY_{\leq h}^{\operatorname{cr}} \rightarrow \cY_{\leq h}$. In particular, this composite is proper. Since $Q \rightarrow \cY_{\leq h} \otimes_{\bZ_p} \bF_p$ is a monomorphism it is, by \cite[Tag 01L4]{stacks-project}, separated. Thus, $\cY^{\operatorname{cr}}_{\leq h} \otimes_{\bZ_p} \bF_p \rightarrow Q$ is proper too. Since proper monomorphisms are closed immersions  it follows that $\cY^{\operatorname{cr}}_{\leq h} \otimes_{\bZ_p} \bF_p \rightarrow Q$ is a closed immersion, and hence $\cY^{\operatorname{cr}}_{\leq h} \otimes_{\bZ_p} \bF_p \rightarrow Y^{\nabla}_{\leq h}$ is a closed immersion also.
\end{proof}

\section{Pl\"ucker coordinates for the Hodge type}

\subsection{Setup continued}\label{sec-setup cont}

In this section we further specialise the setup from Section~\ref{sec-setup} by introducing coefficients: 
\begin{itemize}
    \item Fix a finite extension of $\bQ_p$ containing a Galois closure of $K$, and write $\cO$ and $\bF$ respectively for its ring of integers and residue field. Consequently, if $\cJ := \operatorname{Hom}_{\bQ_p}(K,\cO[\frac{1}{p}])$ then $a\otimes b \mapsto (\kappa(a)b)_{\kappa \in \cJ}$ defines an isomorphism $K \otimes_{\bZ_p} \cO \cong \prod_{\kappa\in \cJ} \cO[\frac{1}{p}]$. Thus, any $K \otimes_{\bZ_p} \cO$-module $D$ decomposes as 
    \begin{equation}\label{eq-kappa decomp}
    D = \prod_{\kappa \in \cJ} D_\kappa 
    \end{equation}
    for $D_\kappa \subset D$ the $\cO[\frac{1}{p}]$-submodule on which $K$ acts through $\kappa$. We call $D_\kappa$ the $\kappa$-th part of $D$ and refer to the projection of $d \in D$ onto $D_\kappa$ as the $\kappa$-th part of $d$.
    \item Similarly, if $\cJ_0 := \operatorname{Hom}_{\bZ_p}(W(k),\cO)$ then $a \otimes b \mapsto (\tau(a)b)_{\tau\in \cJ_0}$ gives an isomorphism $W(k) \otimes_{\bZ_p} \cO \cong \prod_{\tau \in \cJ_0} \cO$. Thus, any $W(k) \otimes_{\bZ_p} A$-module $M$ decomposes as 
    \begin{equation}\label{eq-tau decomp}
        M = \prod_{\tau \in \cJ_0} M_\tau
    \end{equation}
    where $M_\tau \subset M$ identifies as the $\cO$-submodule on which $W(k)$ acts through $\tau$.\footnote{Note that \eqref{eq-kappa decomp} does not descend to a decomposition of $\cO_K \otimes_{\bZ_p} \cO$-modules since the idempotents in $K \otimes_{\bZ_p} \cO$ are not integral, unless $K = K_0$.} As above, we call $M_\tau$ the $\tau$-th part of $M$ and refer to the projection of $m \in M$ onto $M_\tau$ as its $\tau$-th part.
    \item For each embedding $\tau \in \cJ_0$ we choose an indexing $\kappa = \kappa(i,\tau)$ for $1 \leq i \leq e$ of the embeddings $\kappa \in \cJ$ extending $\tau$. We can therefore write
    $$
    E(u) \otimes 1 = \prod_{i=1}^e (u-\pi_i)
    $$
    inside $(W(k)\otimes_{\bZ_p} \cO)[u]$ where 
    $\pi_i = (\kappa(\pi))_{\kappa = \kappa(i,\tau)}$
    under the above isomorphism $W(k) \otimes_{\bZ_p} \cO \cong \prod_{\tau \in \cJ_0} \cO$. For $1 \leq i \leq e$, we also set $E_i(u) = \prod_{j=1}^i (u-\pi_j)$. 
    \end{itemize}

\subsection{Fixing the Hodge type}\label{sec-introducing hodge types}

\begin{defn}\label{def-Hodge types}
	\begin{itemize}
		\item A Hodge type is an isomorphism class $\lambda$ of $\bZ$-gradings on $(K \otimes_{\bZ_p} \cO)^d$ by $K \otimes_{\bZ_p} \cO$-submodules. We will often express such $\lambda$ as $(\lambda_\kappa)_{\kappa \in \cJ}$ with $\lambda_\kappa = (\lambda_{\kappa,1} \geq \ldots \geq \lambda_{\kappa,d})$ a tuple of integers containing $\ell$ with multiplicity equal to the $\cO[\frac{1}{p}]$-dimension of the $\kappa$-th part (in the sense of \eqref{eq-kappa decomp}) of $\operatorname{gr}^{\ell}(\lambda)$.
		\item If $\fM \in \cY^{\operatorname{cr}}_{\leq h}(A)$ for a finite flat $\cO$-algebra $A$ then we say $\fM$ has Hodge type $\lambda$ if the filtration
        $$
        \operatorname{Fil}^{\ell}(\varphi^*\fM / E(u) \varphi^*\fM) := \operatorname{Im}\left( \varphi^*\fM \cap \varphi_{\fM}^{-1}(E(u)^{\ell} \fM) \rightarrow \varphi^*\fM / E(u)\varphi^*\fM \right)
        $$
        on $\varphi^*\fM /E(u)\varphi^*\fM$ has associated graded of type $\lambda \otimes_{\cO[\frac{1}{p}]} A[\frac{1}{p}]$ after inverting $p$.
	\end{itemize}
\end{defn}

\begin{exam}
    Consider the rank one Breuil--Kisin module over $\fS$ with $\fS$-generator $e \in \fM$ and $\varphi_{\fM}(e\otimes 1) = E(u)e$. This corresponds, using the conventions from Section~\ref{sec-creating integral monodromy}, to the one dimensional representation of $G_K$ given by the inverse of the cyclotomic character. Then 
    $$
    \operatorname{Fil}^{\ell}(\varphi^*\fM /E(u) \varphi^*\fM) = \begin{cases}
        \varphi^*\fM / E(u) \varphi^*\fM & \text{ if $\ell 
        \leq 1$} \\
        0 & \text{ if $\ell \geq 2$}
    \end{cases}
    $$    
    and so $\fM$ has Hodge type $1$.
\end{exam}

	It is known, see \cite[Proposition 4.7.2]{EGstack}, that the graded pieces of $\left( \varphi^*\fM / E(u)\varphi^*\fM\right)[\frac{1}{p}]$  
    are always $A[\tfrac{1}{p}]$-projective and so each $\fM \in \cY^{\operatorname{cr}}_{\leq h}(A)$ has a well defined Hodge type whenever $\operatorname{Spec}A$ is connected. Consequently, the locus in $\cY^{\operatorname{cr}}_{\leq h}$ of a fixed Hodge type is closed. More precisely, 
\begin{prop}\label{prop-create locus of hodge type mu BKmods}
	For any Hodge type $\lambda$ concentrated in degrees $[0,h]$ there exists an $\cO$-flat substack $\cY^{\operatorname{cr}}_{\lambda} \subset \cY^{\operatorname{cr}}_{\leq h}$ characterised by the fact that if $A$ is finite flat over $\cO$ then $\fM \in \cY^{\operatorname{cr}}_{\leq h}(A)$ factors through $\cY^{\operatorname{cr}}_{\lambda}$ if and only if $\fM$ has Hodge type $\lambda$. Furthermore,
	\begin{enumerate}
		\item If $A$ is a finite $\bF$-algebra and $\fM \in \cY^{\operatorname{cr}}_{\lambda}(A)$ then there exists a finite flat $\cO$-algebra $A^\circ$ with $A = A^\circ\otimes_{\cO}\bF$ and $\fM^\circ \in \cY^{\operatorname{cr}}_\lambda(A^\circ)$ with $\fM^\circ \otimes_{\cO} \bF = \fM$.
		\item  $\operatorname{dim} \cY^{\operatorname{cr}}_{\lambda}\otimes_{\cO} \bF =  \operatorname{dim} \operatorname{FL}_{\lambda}	$
		where $\operatorname{FL}_{\lambda}$ denotes the flag variety over  $\cO[\tfrac{1}{p}]$ classifying decreasing\footnote{i.e.\ with $\ldots \subset \operatorname{Fil}^{i+1} \subset \operatorname{Fil}^i \subset \operatorname{Fil}^{i-1} \subset \ldots$} filtrations on $(K \otimes_{\bZ_p} \cO)^d$ with associated graded of type $\lambda$.
	\end{enumerate}
\end{prop}
\begin{proof}
See \cite[\S4.8]{EGstack} or \cite[\S10]{B24}.
\end{proof}

The difficulty with these constructions is that they only make sense after inverting $p$. This makes the integral and mod $p$ behaviour of $\cY^{\operatorname{cr}}_{\lambda}$ difficult to analyse. In this section we address this by giving integral conditions cutting out these loci.  

\subsection{Convolution}

   One immediate problem that arises when attempting to impose the Hodge type integrally in the presence of ramification is that, while Hodge types decompose according to embeddings $\kappa \in \cJ$, the integral Breuil--Kisin modules only decompose according to the embeddings $\tau \in \cJ_0$. To accommodate this disparity we introduce the following notion:

   \begin{defn}\label{def-convdef}
       Let $\fM$ be a Breuil--Kisin module over a $p$-adically complete $\cO$-algebra $A$ of height $\leq h$. Then a convolution structure $\fM_\bullet$ on $\fM$ is a filtration
       $$
        \varphi^{-1}_{\fM}( E(u)^h \fM) = \fM_e \subset \fM_{e-1} \subset \ldots \subset \fM_1 \subset \fM_0 = \varphi^*\fM
       $$
        by finite projective $\fS_A$-submodules such that $(u-\pi_i)^h \fM_{i-1} \subset \fM_i \subset \fM_{i-1}$ for each $1 \leq i \leq e$. If $A$ is $\bZ_p$-flat then, by \cite[Lemma 3.3.3]{B23b}, such a sequence is uniquely determined by the formula $\fM_i = \varphi^*\fM \cap \varphi^{-1}_{\fM}\left( E_i(u)^h \fM\right)$.
   \end{defn}   

   We can then, after base-changing to $\operatorname{Spf}\cO$, consider variants of $\cY_{\leq h}$ and $\cY^{\operatorname{cr}}_{\leq h}$ from Definition~\ref{def-BKcryststack}, in which we additionally parametrise a choice of convolution structure:
   \begin{defn}
       For any $p$-adically complete $\cO$-algebra $A$ define $\cY_{\leq h}^{\operatorname{conv}}(A)$ to be the groupoid consisting of $(\fM,\fM_\bullet)$ with $\fM \in \cY_{\leq h}(A)$ and $\fM_\bullet$ a convolution structure on $\fM$. The resulting limit preserving category fibred over $\operatorname{Spf}\cO$ is clearly then projective over $\cY_{\leq h} \times_{\operatorname{Spf}\bZ_p} \operatorname{Spf}\cO$ and so itself a  $p$-adic formal stack of finite type over $\operatorname{Spf}\cO$. Write $\cY^{\operatorname{cr},\operatorname{conv}}_{\leq h}$ for the $\cO$-flat substack of $\cY^{\operatorname{conv}}_{\leq h}$.
   \end{defn}

   Likewise, we can add convolution structures to the $\cY^{\operatorname{cr}}_\lambda$ from Proposition~\ref{prop-create locus of hodge type mu BKmods}:

   \begin{prop}\label{prop-create locus of hodge type mu conv BKmods}
	For any Hodge type $\lambda$ concentrated in degrees $[0,h]$ there exists an $\cO$-flat substack $\cY^{\operatorname{cr,conv}}_{\lambda} \subset \cY^{\operatorname{cr,conv}}_{\leq h}$ characterised by the fact that if $A$ is finite flat over $\cO$ then $\fM \in \cY^{\operatorname{conv}}_{\leq h}(A)$ factors through $\cY^{\operatorname{cr,conv}}_{\lambda}$ if and only if $\fM$ has Hodge type $\lambda$. Furthermore,
	\begin{enumerate}
		\item If $A$ is a finite $\bF$-algebra and $\fM \in \cY^{\operatorname{cr,conv}}_{\lambda}(A)$ then there exists a finite flat $\cO$-algebra $A^\circ$ with $A = A^\circ\otimes_{\cO}\bF$ and $\fM^\circ \in \cY^{\operatorname{cr,conv}}_\lambda(A^\circ)$ with $\fM^\circ \otimes_{\cO} \bF = \fM$.
    		\item  $\operatorname{dim} \cY^{\operatorname{cr,conv}}_{\lambda}\otimes_{\cO} \bF =  \operatorname{dim} \operatorname{FL}_{\lambda} = \sum_{\kappa \in \cJ} \operatorname{dim}_{\cO} \operatorname{FL}_{\lambda_\kappa}	$
		where $\operatorname{FL}_{\lambda}$ and $\operatorname{FL}_{\lambda_\kappa}$ are the flag varieties from Proposition~\ref{prop-create locus of hodge type mu BKmods}.
	\end{enumerate}
\end{prop}
\begin{proof}
The proof is identical to that of Proposition~\ref{prop-create locus of hodge type mu BKmods}. For (2) one uses that convolution structures are uniquely determined on Breuil--Kisin modules with $\bZ_p$-flat coefficients.
\end{proof}
\subsection{Affine Grassmannians}\label{sec-affinegrassmannian}

We will see that if $\fM \in \cY^{\operatorname{cr,conv}}_{\leq h}(A)$ for a finite flat $\cO$-algebra $A$ then the relative positions of the lattices $\fM_{i,\tau} \subset \fM_{i-1,\tau}$ encode the Hodge type $\kappa(i,\tau)$-th part of the Hodge type of $\fM$. The most convenient way to articulate this is via the affine Grassmannian:
\begin{const}\label{const-definition of Psi}
	We fit $\cY^{\operatorname{cr,conv}}_{\leq h}$ into the following diagram 
	$$
	\begin{tikzcd}
		& \ar[dr,"\Psi"] \ar[dl]\widetilde{\cY^{\operatorname{cr,conv}}_{\leq h}} & \\
		\cY^{\operatorname{cr,conv}}_{\leq h} & & \prod_{\kappa \in \cJ} \operatorname{Gr}_{\leq h}^{(\kappa)}
	\end{tikzcd}
	$$
	where:
	\begin{itemize}
		\item $\widetilde{\cY^{\operatorname{cr,conv}}_\lambda}$ denotes the formal $p$-adic stack over $\operatorname{Spf}\cO$ classifying pairs $\fM \in \cY^{\operatorname{cr,conv}}_{\lambda}(A)$ together with a choice of $\fS_A$-basis $\beta_i$ of $\fM_{i}$. By convention we set $\beta^0 := \varphi_{\fM}(E(u)^{-h}\beta_e)$, which is an $\fS_A$-basis of $\fM$ and $\beta_0 := \beta^0 \otimes 1$, which is an $\fS_A$-basis of $\varphi^*\fM$. 
		\item $\operatorname{Gr}_{\leq h}^{(\kappa)}$ denotes the finite type $\cO$-scheme with $A$-points classifying $A[[u-\kappa(\pi)]]$-submodules $(u-\kappa(\pi))^hA[[u-\kappa(\pi)]]^d \subset \cE \subset A[[u-\kappa(\pi)]]^d$. This is equipped with an action of the group scheme $L^{+,(\kappa)}G \rtimes \operatorname{Aut}^+$, where $L^{+,(\kappa)}G$ represents the functor $A \mapsto \operatorname{GL}_d(A[[u-\kappa(\pi)]])$ and acts via the standard action on $A[[u-\kappa(\pi)]]^d$ while $\operatorname{Aut}^+$ represents $A \mapsto A[[u-\kappa(\pi)]]^\times$ and acts by scaling the parameter $u-\kappa(\pi)$.
		\item The $\kappa = \kappa(i, \tau)$-th  factor of $\Psi$ sends $(\fM,\beta_\bullet)$ onto the element $\Psi(\fM_{i,\tau},\beta_{i-1}) \in \operatorname{Gr}_{\leq h}^{(\kappa)}(A)$ corresponding to the submodule
		$$
	   \fM_{i,\tau} \subset \fM_{i-1,\tau} \cong A[[u-\kappa(\pi)]]^d  
		$$
		where the isomorphism is obtained by considering the $\tau$-th part of the identification $\fM_{i-1} \cong \fS_A^d$ induced by $\beta_{i-1}$ using that $\fS_{A,\tau} \cong A[[u-\kappa(\pi)]]$ whenever $A$ is $p$-adically complete.
	\end{itemize}
\end{const}

Construction~\ref{const-definition of Psi} allows us to pull closed conditions on $\operatorname{Gr}^{(\kappa)}_{\leq h}$ back to $\widetilde{\cY^{\operatorname{cr,conv}}_{\leq h}}$. The most basic example of such closed conditions are given by Schubert varieties:

\begin{defn}\label{defn-schubert vars} 
	If $\eta = (\eta_1 \geq \ldots \geq \eta_{d})$ with each $\eta_{\ell} \in [0,h]$ then we write $\operatorname{Gr}_{ \eta}^{(\kappa)}$ for the $L^{+,(\kappa)}G$-orbit through the $\cO$-point $\cE_{\eta} \in \operatorname{Gr}_{\leq h}^{(\kappa)}$ corresponding to the submodule in $\cO[[u-\kappa(\pi)]]^d$ generated by $(u-\kappa(\pi))^{\eta_{\ell}} e_\ell$ for $e_1,\ldots,e_d$ the standard basis of $A[[u-\kappa(\pi)]]^d$. Set $\operatorname{Gr}_{\leq \eta}^{(\kappa)}$ equal to the closure of $\operatorname{Gr}_{ \eta}^{(\kappa)}$ in $\operatorname{Gr}_{\leq h}^{(\kappa)}$. \end{defn}

The geometry of these orbit closures is very well understood. In particular:

\begin{prop}\label{prop-schubert vars geom}
    $\operatorname{Gr}_{\leq \eta}^{(\kappa)} \otimes_{\cO} \bF$ is Cohen--Macaulay and integral of dimension
    $$
    \sum_{1 \leq j < \ell \leq d } \left( \eta_{j} - \eta_{\ell} \right)
    $$ 
    Furthermore, the $\cO$-valued point of $\operatorname{Gr}_{\leq h}^{(\kappa)}$ corresponding to $\cE_\nu \subset \cO[[u-\kappa(\pi)]]^d$ factors through $\operatorname{Gr}^{(\kappa)}_{\leq \eta}$ if and only if 
    $\nu_{d}+\ldots +\ldots+\nu_{d-j} \geq\eta_{d}+\ldots+\eta_{d-j}$
    for each $0\leq j \leq d-1$, with equality when $j=d-1$.
\end{prop}
\begin{proof}
    See, for example, \cite[Theorem 1.4]{L23}.
\end{proof}

In what follows we will write $\nu \leq \eta$ if $\nu = (\nu_1,\ldots,\nu_d)$ satisfies the last condition from Proposition~\ref{prop-schubert vars geom}. If $\eta$ and $\nu$ are interpreted as cocharacters of the diagonal torus in $G = \operatorname{GL}_d$ this is just the usual dominance ordering with respect to the upper triangular Borel.

\subsection{Main results}\label{sec-mainresults}

We now state the main results of this section. Our goal is to isolate the Hodge type on $\cY^{\operatorname{cr},\operatorname{conv}}_{\leq h}$ without inverting $p$ using the constructions from Construction~\ref{const-definition of Psi}. The following most basic constraint is well-known (but see Section~\ref{sec-proofs} for a precise proof): 

\begin{prop}\label{prop-in schubert variety}
	Let $\lambda$ be a Hodge type concentrated in degree $[0,h]$. Suppose $(\fM,\beta_\bullet) \in \widetilde{\cY^{\operatorname{cr,conv}}_\lambda}(A)$ for an $\cO$-algebra $A$ of topologically finite type. Then, for each $\kappa = \kappa(i,\tau)$, one has 
    $$\Psi(\fM_{i,\tau},\beta_{i-1}) \in \operatorname{Gr}_{\leq \lambda_{\kappa}^*}^{(\kappa)}(A)
    $$
    where $\lambda_{\kappa}^*= (\lambda_{\kappa,1}^* \geq \ldots \geq \lambda_{\kappa,d}^*)$ with $\lambda_{\kappa,i}^* := h - \lambda_{\kappa,d-i}$ for all $1 \leq i \leq d$. Here, no assumptions on $h$ are necessary.
 \end{prop}
 \begin{rmk}
     The twisted $\lambda_\kappa^*$ appears in Proposition~\ref{prop-in schubert variety} due to the definition of $\Psi(\fM_{i,\tau},\beta_{i-1})$ from Construction~\ref{const-definition of Psi}. Specifically, $\Psi(\fM_{i,\tau},\beta_{i-1})$ is built from convolution structures between $\fM_e = \varphi_{\fM}^{-1}(E(u)^h\fM)$ and $\varphi^*\fM$, and so must depend on $h$. 
 \end{rmk}

As observed in the introduction, this is not sufficient to isolate $\cY^{\operatorname{cr},\operatorname{conv}}_\lambda$ as the same conditions hold for any Hodge type $\nu$ with $\nu_\kappa \leq \lambda_\kappa$ for each $\kappa$. The main results of this section (see Theorems~\ref{thm-integral equations1} and~\ref{thm-integral equations2} below) describe a refinement of this condition.  

In order to articulate these refinements we use an embedding of $\operatorname{Gr}_{\leq \lambda_\kappa^*}^{(\kappa)}$ into projective space. More precisely, set $\cV_h := \cO[[u-\kappa(\pi)]]^d / (u-\kappa(\pi))^h \cO[[u-\kappa(\pi)]]^d$ and consider the Pl\"ucker embedding
\begin{equation}\label{eq-pluckerembedding}
\Theta_\kappa: \operatorname{Gr}_{\leq \lambda_\kappa^*}^{(\kappa)} \rightarrow \mathbb{P}\left( \bigwedge^{dh-r} \cV_h \right),\qquad r := \lambda_{\kappa,1}^*+\ldots+\lambda_{\kappa,d}^* = dh - (\lambda_{\kappa,1} + \ldots + \lambda_{\kappa,d}) 
\end{equation}
sending an $A$-valued point $\cE$ onto the line in $\bigwedge^{dh-r} (\cV_h \otimes_{\cO} A)$ given by the determinant of the image of $\cE$ in $(\cV_h \otimes_{\cO} A)$. This is well-defined since $A[[u-\kappa(\pi)]]^d /\cE$ is $A$-projective of rank $r$ whenever $\cE \in \operatorname{Gr}_{\leq \lambda_\kappa^*}^{(\kappa)}(A)$. The map is $L^{+,(\kappa)}G \rtimes \operatorname{Aut}^+$-equivariant, for the standard action on $\cV_h$.

\begin{thm}\label{thm-integral equations1}
    For any finite flat $\cO$-algebra $A$, consider $(\fM,\beta_\bullet) \in \widetilde{\cY^{\operatorname{cr},\operatorname{conv}}_{\lambda}}(A)$ with $h \leq p$ and fix a derivation $N^\varphi$ on $\varphi^*\fM$ as in Proposition~\ref{prop-create monodromies}. Then for each $1 \leq i \leq e$,
    \begin{equation}\label{eq-containment}
        \tfrac{E_{i-1}(u)}{u}N^\varphi(\fM_{i-1}) \subset \fM_{i-1} 
    \end{equation}
    for $E_i(u) = \prod_{j=1}^i (u-\pi_j)$ defined as in Section~\ref{sec-setup cont}. Moreover, for each $\kappa = \kappa(i,\tau)$, we have the following identity
    \begin{equation}\label{eq-key identity}
        \tfrac{E_i(u)}{u}N^\varphi \cdot v_{\fM,\kappa} = \Bigg( \Big( c^*(u)E_{i-1}(u)\Big)\Big|_{u=\kappa(\pi)} \sum_{l=1}^d  \sum_{j =\lambda_{\kappa,l}^*}^{h-1} j \Bigg) v_{\fM,\kappa}
    \end{equation}
    where $c^*(u)$ is as defined in Section~\ref{sec-setup}, and
    \begin{itemize}
        \item  $\frac{E_i(u)}{u}N^\varphi$ is interpreted as a derivation on $\cV_h \otimes_{\cO} A \cong \fM_{i-1,\tau} / (u-\kappa(\pi))^h\fM_{i-1,\tau}$ with the isomorphism obtained from the basis $\beta_{i-1}$.
        \item $\frac{E_i(u)}{u}N^\varphi$ acts on exterior powers of $\cV_h \otimes_{\cO} A$ by sending $v_1\wedge \ldots \wedge v_{dh-r}$ to  
        $\sum_{i=1}^{dh-r} v_1 \wedge \ldots \wedge v_{i-1} \wedge  \tfrac{E_i(u)}{u}N^\varphi(v_i) \wedge v_{i+1} \wedge \ldots \wedge v_{dh-r}$.
        \item $v_{\fM,\kappa} \in \bigwedge^{dh-r} (\cV_h \otimes_{\cO} A)$ is a vector spanning the line $\Theta_\kappa ( \Psi(\fM_{i,\tau},\beta_{i-1}))$ $($such $v_{\fM,\kappa}$ exist Zariski locally on $\operatorname{Spec}A)$.
            \end{itemize} 
\end{thm}

The proof will be given in Section~\ref{sec-proofs} following some general setup in Section~\ref{sec-preparations}. For later applications it will be convenient to work with the following dual version of Theorem~\ref{thm-integral equations1}. Here we impose the condition in \eqref{eq-key identity} through the vanishing of specific global sections on $\operatorname{Gr}_{\leq \lambda_\kappa^*}^{(\kappa)}$. More precisely, let $\cL$ denote the pullback of $\cO(1)$ along \eqref{eq-pluckerembedding} and view $\bigwedge^r(\cV_h)$ as (possibly zero) global sections on $\cL$ via the non-degenerate pairing
\begin{equation}\label{eq-duality}
\bigwedge^r \cV_h \times \bigwedge^{dh-r} \cV_h \rightarrow \det(\cV_h)
\end{equation}
given by $(v,w) \mapsto v \wedge w$. Concretely, the pullback of $\cL$ along an $A$-valued point corresponding to $\cE \in \operatorname{Gr}_{\leq \lambda_\kappa^*}^{(\kappa)}$ is given by the $A$-module $\bigwedge^r A[[u-\kappa(\pi)]]^d / \cE$.

\begin{thm}\label{thm-integral equations2}
 	Maintain the notation from Theorem~\ref{thm-integral equations1}  so that $\tfrac{E_i(u)}{u}N^\varphi$ can be viewed as a derivation on $\cV_h \otimes_{\cO} A$. Then $(\fM,\beta_\bullet) \in 
    \widetilde{\cY^{\operatorname{cr,conv}}_\lambda}(A)$ implies $\Psi(\fM_{i,\tau},\beta_{i-1})$ lies in the vanishing locus of 
 	\begin{equation}\label{eq-section which will vanish}
 	\tfrac{E_i(u)}{u}N^\varphi \cdot v - \Bigg( \Big( c^*(u)E_{i-1}(u) \Big)\Big|_{u=\kappa(\pi)} \sum_{l=1}^d  \sum_{j = 1}^{\lambda_{\kappa,l}^*-1} j \Bigg) v \in \bigwedge^r \left( \cV_h \otimes_{\cO} A \right)
 	\end{equation}
 	for all $v \in \bigwedge^r \cV_h$ interpreted as global sections on $\cL$. 
 \end{thm} 

We will see that the majority of the equations imposed in \eqref{eq-section which will vanish} are vacuous in the sense that they vanish on the whole of $\operatorname{Gr}_{\leq \lambda_\kappa^*}^{(\kappa)}$. Specifically, this occurs for any $v \in \bigwedge^r \cV_h$ on which the constant subgroup $\bG_m \subset \operatorname{Aut}^+$ acts with weight $\geq \sum_{l=1}^d \sum_{j=1}^{\lambda_{\kappa,l}^*-1} j$ (as in Construction~\ref{const-definition of Psi}, $\operatorname{Aut}^+$ acts on $\cV_h$ by scaling the parameter $u-\kappa(\pi))$.

 \subsection{Pl\"ucker coordinates and constant flag varieties}\label{sec-preparations}

Here we explain the conditions from the two theorems in the previous section in a general context. For this fix $\kappa$ and consider $\nu = (\nu_1 \geq  \ldots  \geq \nu_d) \leq \lambda_\kappa$ (note, this is not $\lambda_\kappa^*$). There is then a closed immersion
\begin{equation}\label{eq-twisted embedding of FL}
\operatorname{FL}_{\nu} \hookrightarrow \operatorname{Gr}_{\leq \lambda_\kappa^*}^{(\kappa)}, \qquad \operatorname{Fil}^\bullet \mapsto \sum_{j \in \bZ} \operatorname{Fil}^j \otimes_A(u-\kappa(\pi))^{h-j} A[[u-\kappa(\pi)]]
\end{equation}
where, as in Proposition~\ref{prop-create locus of hodge type mu conv BKmods}, $\operatorname{FL}_{\nu}$ denotes the projective $\cO$-scheme with $A$-points classifying decreasing filtrations on $A^d$ whose $i$-th graded piece is $A$-projective of constant rank equal to the multiplicity of $i$ in $\nu$. Equivalently, this closed immersion identifies $\operatorname{FL}_\nu$ with the (closed) $G$-orbit in $\operatorname{Gr}_{\leq h}^{(\kappa)}$ through the $\cO$-valued point $\cE_{\nu^*} \in \operatorname{Gr}_{\leq h}^{(\kappa)}(\cO)$ from Definition~\ref{defn-schubert vars}, for $\nu^* = (\nu^*_1 \geq \ldots \geq \nu^*_d)$ with $\nu^*_\ell := h - \nu_{d-\ell}$ for $1 \leq \ell \leq d$.

We  write $\partial := (u-\kappa(\pi))\frac{d}{du}$ which we view as a derivation on $A[[u-\kappa(\pi)]]^d$ acting coordinate-wise. Also, if $X \in L^{+,(\kappa)}G(A)$ then we set $\operatorname{dlog}(X) = X^{-1} \partial(X)$, interpreted as an $A[[u-\kappa(\pi)]]$-linear endomorphism of $A[[u-\kappa(\pi)]]^d$. Consider the following assertions:
\begin{prop}\label{prop-loop fixed and loop derivative fixed}
    Let $A$ be any $\cO$-algebra with $\operatorname{Spec}A$ connected. Suppose $\cE \in \operatorname{Gr}_{\leq \lambda_\kappa^*}^{(\kappa)}(A)$ and that $v_{\cE} \in \bigwedge^{dh-r} (\cV_h \otimes_{\cO} A)$ spans the line corresponding to $\Psi_\kappa(\cE) \in \mathbb{P}(\bigwedge^{dh-r} \cV_h)$. Consider the following conditions:
    \begin{enumerate}
        \item There is an $X \in L^{+,(\kappa)}G(A)$ and $\nu = (\nu_1 \geq \ldots \geq \nu_d) \leq \lambda_\kappa$ so that $X \cdot \cE \in \operatorname{FL}_{\nu}$.
        \item There is an $\alpha(u) \in A[[u-\kappa(\pi)]]$ and a derivation $\cN$ on $A[[u-\kappa(\pi)]]^d$ over $\alpha(u)\partial$ so that 
        $$
        \cN \cdot v_{\cE} = \left( \alpha(u)|_{u=\kappa(\pi)} \sum_{i=1}^d \sum_{j=\nu_i^*}^{h-1} j \right) v_{\cE}
        $$ 
        for some $\nu = (\nu_1 \geq \ldots \geq \nu_d) \leq \lambda_\kappa$.
        \item There is an $\alpha(u) \in A[[u-\kappa(\pi)]]$ and a derivation $\cN$ on $A[[u-\kappa(\pi)]]^d$ over $\alpha(u)\partial$ so that $\cN(\cE) \subset \cE$. 
    \end{enumerate}
    Then $(1) \Rightarrow (2)$ for a fixed $\alpha(u)$ and $\cN := \alpha(u)\operatorname{dlog}(X) + \alpha(u)\partial$. Also, $(2) \Rightarrow (3)$ for fixed choices of $\alpha(u)$ and $\cN$.
    
     Conversely, suppose in (3) that  $\alpha(u) \in A[[u-\kappa(\pi)]]^\times$ and $\cN = \cN_0 + \alpha(u)\partial$ for an endomorphism $\cN_0 \equiv 0$ modulo $u-\kappa(\pi)$. If $h\leq p$ then (1) holds for $X \in L^{+,(\kappa)}G(A)$ satisfying $\alpha(u)\operatorname{dlog}(X) \equiv \cN_0$ modulo $(u-\kappa(\pi))^h$.
\end{prop}

In the proof it will be convenient to interpret derivations $\cN$ as in the proposition as elements of $\operatorname{Lie}\left( L^{+,(\kappa)}G \rtimes \operatorname{Aut}^+ \right) \otimes_{\cO} A$. More precisely, any derivation $\cN = \cN_0 + \alpha(u)\partial$ identifies with the $A[\epsilon]/(\epsilon^2)$-valued point of $L^{+,(\kappa)}G \rtimes \operatorname{Aut}^+$ given by $(1+ \epsilon \cN_0, 1+ \epsilon\alpha(u))$. In particular, under these identifications the adjoint action of $X^{-1} \in L^{+,(\kappa)}G(A)$ on $\operatorname{Lie}\left( L^{+,(\kappa)}G \rtimes \operatorname{Aut}^+ \right) \otimes_{\cO} A$ sends $1+ \epsilon \alpha(u)$ onto $\alpha(u) \operatorname{dlog}(X) + \alpha(u)\partial$.

\begin{proof}
    To prove $(1) \Rightarrow (2)$ we observe that $\operatorname{FL}_\nu$ is fixed by the $\operatorname{Aut}^+$-action on $\operatorname{Gr}_{\leq \lambda_\kappa^*}^{(\kappa)}$. Thus, if $X \cdot \cE \in \operatorname{FL}_\nu$ then $\operatorname{Aut}^+$ acts on the line spanned by $X \cdot v_{\cE}$ in $\bigwedge^{dh-r} \left( \cV_h \otimes_{\cO} A\right)$ through some character of $\operatorname{Aut}^+$. Any such character must factor through the maximal reductive quotient of $\operatorname{Aut}^+$, i.e.\ through the constant term map $\operatorname{Aut}^+ \rightarrow \bG_m$. Consequently, there is an $N \in \bZ$ such that 
    \begin{equation}\label{eq-eigenvector for loop}
    f \cdot (X \cdot v_{\cE}) = f(\kappa(\pi))^N X \cdot v_{\cE}
    \end{equation}
    whenever $f(u) \in \operatorname{Aut}^+$. We claim that $X \cdot \cE \in \operatorname{FL}_{\nu}$ implies $N = \sum_{i=1}^d \sum_{j=\nu^*}^{h-1} j$.
    This can be checked directly on any point of $\operatorname{FL}_\nu$ because the constant subgroup $G \subset L^{+,(\kappa)}G$ acts transitively on $\operatorname{FL}_\nu$ and commutes with $\operatorname{Aut}^+$. In particular, we can consider the $\cO$-valued point $\cE_{\nu^*}$  whose corresponding line $\Psi_\kappa(\cE_\nu)$ is spanned by the vector 
    \begin{equation}\label{eq-vector for E_lambda}
        \bigwedge_{j=1}^d \left( (u-\kappa(\pi))^{\nu^*_i} e_i \wedge \ldots \wedge (u-\kappa(\pi))^{h-1}e_i\right)
    \end{equation}
    on which $\operatorname{Aut}^+$ clearly acts as claimed with $N = \sum_{i=1}^d \sum_{j=\nu_i^*}^{h-1} j$.
    Applying \eqref{eq-eigenvector for loop} with $f = 1 +\epsilon \alpha(u)$ therefore gives
    $$
     X^{-1}(1+\epsilon \alpha(u)) \cdot (X \cdot v_{\cE}) =  \left( 1 + \epsilon \left( \alpha(u)|_{u=\kappa(\pi)} \sum_{i=1}^d \sum_{j=\nu_i^*}^{h-1} j \right)\right)   \cdot v_{\cE}
    $$
    This combined with the observation that the adjoint action of $X^{-1}$ sends $1+ \epsilon \alpha(u)$ onto $\alpha(u) \operatorname{dlog}(X) + \alpha(u) \partial$ implies (2). 

    To prove $(2) \Rightarrow (3)$ we write, possibly after localising $A$, $v_{\cE} = v_1 \wedge \ldots \wedge v_{dh-r}$ for an $A$-basis of the image of $\cE$ inside $\cV_h \otimes_{\cO} A$. Then $\cN(\cE) \subset \cE$ if and only if 
    $$
    \cN(v_j) \wedge v_1 \wedge \ldots \wedge v_{dh-r} = \pm (\cN \cdot v_{\cE}) \wedge v_j
    $$
    vanishes inside $\bigwedge^{dh-r+1} \left( \cV_h \otimes_{\cO} A \right)$ for each $j=1,\ldots,dh-r$. But (2) implies $(\cN \cdot v_{\cE}) \wedge v_j = 0$ so we are done.

    Finally, we turn to the converse statement. Since $\alpha(u)$ is now a unit we can assume $\alpha(u)=1$ by rescaling $\cN$. Since $\cN_0 \equiv 0$ modulo $u-\kappa(\pi)$,  a straightforward calculation shows the existence of $X \in L^{+,(\kappa)}G$ with $\cN_0 \equiv \operatorname{dlog}(X)$ modulo $(u-\kappa(\pi))^h$ whenever $(h-1)!$ is invertible in $A$. As $h \leq p$ this is automatic. For such $X$ we have $X^{-1}\partial(X \cdot \cE) = \cN(\cE)$ and $\cN \cdot v_{\cE} = X^{-1} \partial( X\cdot v_{\cE})$. This allows us to reduce the claim to the case $\cN_0= 0$. The condition $\cN(\cE) \subset \cE$ is then equivalent to the assertion that, for $\bG_m \subset \operatorname{Aut}^+$ the constant subgroup, the action map $\bG_m \rightarrow \operatorname{Gr}_{\leq \lambda_\kappa^*}^{(\kappa)}$ given by $t \mapsto t \cdot \cE$ has vanishing derivative. We claim this implies $\cE$ is fixed by this $\bG_m$-action. This would follow if an affine neighbourhood of $ \cE \in \operatorname{Gr}_{\leq \lambda_\kappa^*}^{(\kappa)}$ could be embedded $\bG_m$-equivariantly into a smooth affine scheme with $\bG_m$-weights $< p$. But this is clear since $\operatorname{Gr}_{\leq \lambda_\kappa^*}^{(\kappa)}$ embeds into the Grassmannian of $dh-r$-planes inside $\cV_h$ and $h \leq p$ ensures $\cV_h$ has $\bG_m$-weights $<p$. It is well-known that the $\bG_m$-fixed points in $\operatorname{Gr}_{\leq \lambda_\kappa^*}^{(\kappa)}$ are the disjoint union of the $\operatorname{FL}_{\nu}$ for $\nu\leq \lambda_\kappa$. Since each $\operatorname{FL}_\nu$ is connected we have $\cE \in \operatorname{FL}_\nu$ for some $\nu$, which finishes the proof.
\end{proof}

The following shows that the condition in (2) of the proposition is sufficient to isolate $\operatorname{FL}_{\lambda_\kappa} \subset \operatorname{Gr}_{\leq \lambda_\kappa}^{(\kappa)}$ in characteristic zero. However, this need not be the case integrally, even when $h \leq p$:
\begin{cor}\label{cor-equation for flag variety}
    The condition $\partial \cdot v_{\cE} = \left( \sum_{\ell=1}^d \sum_{j = \lambda_{\kappa,\ell}^*}^{h-1} j  \right) v_{\cE}$
    cuts out the closed subscheme $\operatorname{FL}_{\lambda_\kappa}[\frac{1}{p}] \subset \operatorname{Gr}_{\leq \lambda_\kappa}^{(\kappa)}[\frac{1}{p}]$. The same is also true without inverting $p$ if $\sum_{\ell=1}^d \sum_{j = \nu_\ell^*}^{h-1} j < p$ for all $\nu\leq \lambda_\kappa$.
\end{cor}
\begin{proof} 
    Proposition~\ref{prop-loop fixed and loop derivative fixed} shows the stated condition cuts out a union of $\operatorname{FL}_\nu$ for $\nu \leq \lambda_\kappa$. If $\nu = (\nu_1,\ldots,\nu_d) < \lambda_\kappa$ then $\sum_{\ell=1}^d \sum_{j = \lambda_{\kappa,\ell}^*}^{h-1}j < \sum_{\ell=1}^d \sum^{h-1}_{j=\nu_{\ell}^*} j$ so if $p$ is invertible this condition isolates $\operatorname{FL}_{\lambda_\kappa}$. If $p$ is not invertible then this is only the case if $\sum_{\ell=1}^d \sum_{j = \lambda_{\kappa,\ell}^*}^{h-1}j \not\equiv \sum_{\ell=1}^d \sum_{j=\nu_\ell^*}^{h-1} j$ modulo $p$ for any $\nu <\lambda_\kappa$.
\end{proof}

Finally, we explain the relationship between the conditions from Theorem~\ref{thm-integral equations1} and Theorem~\ref{thm-integral equations2}.
\begin{lemma}\label{lem-eigenvector vs equations}
    Suppose $v \in \bigwedge^{dh-r} (\cV_h \otimes_{\cO} A)$ spans a line contained in the image of $\Psi_\kappa$ and that $\cN = \cN_0 + \alpha(u) \partial$ for an endomorphism $\cN_0 \equiv 0$ modulo $u-\kappa(\pi)$ and $\alpha(u) \in A[[u-\kappa(\pi)]]$. The following are then equivalent:
    \begin{enumerate}
        \item $\cN \cdot v = \left( \sum_{i=\ell}^d \sum_{j = \lambda_{\kappa,\ell}^*}^{h-1} j  \right) v$
        \item The linear functional $\cN \cdot w -\left( \alpha(u)|_{u = \kappa(\pi)} \sum_{\ell=1}^d \sum_{j=1}^{\lambda_{\kappa,\ell}^*-1} j \right)w$ vanishes on $v$ for all $w \in \bigwedge^r \cV_h$
        \item The same as in (2) but only for those $w \in \bigwedge^r \cV_h$ with weight $< \sum_{\ell=1}^d \sum_{j=1}^{\lambda_{\kappa,\ell}-1} j$ for the action of the constant subgroup $\bG_m \subset \operatorname{Aut}^+$. 
    \end{enumerate}
    Here $w \in \bigwedge^r \cV_h$ are viewed as linear functionals on $\bigwedge^{dh-r} \cV_h$ via \eqref{eq-duality} and a chosen trivialisation of $\operatorname{det}\cV_h$.
\end{lemma}
\begin{proof}
    For $(1) \Leftrightarrow (2)$, note that the $L^{+,(\kappa)}G \rtimes \operatorname{Aut}^+$-equivariance of the pairing \eqref{eq-duality} implies that if $n \in \operatorname{Lie}\left(  L^{+,(\kappa)}G \rtimes \operatorname{Aut}^+\right)$ then  $v \wedge (n \cdot w) = \operatorname{det}(n) (v \wedge w )- (n \cdot v) \wedge w$ where $\operatorname{det}(n)$ denotes the scalar through which $n$ acts on $\operatorname{det}(\cV_h)$. The assumption that $\cN_0 \equiv 0$ modulo $u - \kappa(\pi)$ implies $\cN$ acts on $\operatorname{det}(\cV_h)$ via $\alpha(u)|_{u = \kappa(\pi)} \sum_{\ell=1}^d \sum_{j=1}^{h-1} j$ and so
    $$
     v \wedge (\cN \cdot w) =  \left( \alpha(u)|_{u = \kappa(\pi)} \sum_{i=\ell}^d \sum_{j=1}^{h-1} j \right) v \wedge w - (\cN \cdot v) \wedge w     $$
    for any $w \in \bigwedge^r \cV_h$. This immediately shows $(1) \Rightarrow (2)$ and, when combined with the non-degeneracy of \eqref{eq-duality}, gives $(2) \Rightarrow (1)$.

    It remains to show $(3) \Rightarrow (2)$. We do this by showing that if $w$ has $\bG_m$-weight $\geq  \sum_{\ell=1}^d \sum_{j=1}^{\lambda_{\kappa,\ell}^* -1} j$ then $\cN \cdot w +(x- \alpha(u)|_{u = \kappa(\pi)} \sum_{\ell=1}^d \sum_{j=1}^{h-1} j )w$ vanishes when viewed as an element of $H^0(\operatorname{Gr}_{\leq \lambda_\kappa^*}^{(\kappa)},\cL)$. First, we claim that $w$ itself vanishes inside $H^0(\operatorname{Gr}_{\leq \lambda_\kappa^*}^{(\kappa)},\cL)$ if it has $\bG_m$-weight $> \sum_{\ell=1}^d \sum_{j=1}^{\lambda_{\kappa,\ell}^* -1} j$. Let us see why this claim finishes the proof. Since $\cN_0 \equiv 0$ modulo $u-\kappa(\pi)$ we can express $\cN_0 \cdot w $ as a sum of vectors with strictly greater $\bG_m$-weights. The claim therefore lets us assume $\cN = \alpha(u)\partial$ and that $w$ has exact $\bG_m$-weight $\sum_{i=\ell}^d \sum_{j=1}^{\lambda_{\kappa,\ell}^* -1} j$. But then $\alpha(u)\partial \cdot w = \alpha(u)|_{u =\kappa(\pi)}\sum_{\ell=1}^d \sum_{j=1}^{\lambda_{\kappa,\ell}^* -1} w$ is automatic.
    
    It only remains to check the claim. From the construction of  $\operatorname{Gr}_{\leq \lambda_\kappa^*}^{(\kappa)}$ in Definition~\ref{defn-schubert vars} the given vanishing can be checked on any $\cE = X \cdot \cE_{\lambda_\kappa^*}$ with $ X\in L^{+,(\kappa)}G(A)$. Thus, it suffices to show
    $$
    w \wedge (X \cdot v_{\lambda_\kappa^*}) = \operatorname{det}(X) (X^{-1} \cdot w) \wedge v_{\lambda_\kappa^*} 
    $$
    vanishes for $v_{\lambda_\kappa^*} \in \bigwedge^{dh-r}\cV$ spanning the line $\Psi_\kappa(\cE_{\lambda_\kappa^*})$. Invertibility of $X$ implies $\operatorname{det}(X) \in A^\times$ while $X^{-1} \cdot w$ is an $A$-linear sum of elements in $\bigwedge^r \cV_h$ with $\bG_m$-weight $\geq$ that of $w$. We are therefore reduced to showing $w \wedge v_{\lambda_\kappa^*} =0$ whenever $w$ has $\bG_m$-weight $>\sum_{\ell=1}^d \sum_{j=1}^{\lambda_{\kappa,\ell}^*-1} j$. We can assume $w$ is of the form $\bigwedge_{s=1}^r (u-\kappa(\pi))^{j_s}e_{\ell_s}$ for some $1\leq \ell_s \leq d$ and $0 \leq j_s \leq h-1$ with $\sum_{s=1}^r j_s > \sum_{\ell=1}^d \sum_{j=1}^{\lambda_{\kappa,\ell}^*-1}j$. But $v_{\lambda_\kappa^*}$ can be expressed, up to a scalar, as in \eqref{eq-vector for E_lambda}, and so the vanishing of $w \wedge v_{\lambda_\kappa^*}$ is clear. 
\end{proof}

 \subsection{Proofs of Proposition~\ref{prop-in schubert variety}, Theorem~\ref{thm-integral equations1}, and Theorem~\ref{thm-integral equations2}}\label{sec-proofs}

We now return to the notation from Section~\ref{sec-mainresults}. Thus, we fix a finite flat $\cO$-algebra $A$ and $(\fM,\beta_\bullet) \in \widetilde{\cY_\lambda^{\operatorname{cr},\operatorname{conv}}}(A)$.

Recall the convention that $\beta_0 := \varphi_{\fM}(E(u)^{-h}\beta_e) \otimes 1$ is an $\fS_A$-basis of $\varphi^*\fM$. Recall also that $\fS_{A[\frac{1}{p}]}$ is defined as the $E(u)$-adic completion of $(W(k) \otimes_{\bZ_p} A[\frac{1}{p}])[u]$ and so there is an isomorphism $\fS_{A[\frac{1}{p}]} \cong \prod_{\kappa \in \cJ} A[\frac{1}{p}][[u-\kappa(\pi)]]$ arising from the factorisation $E(u) = \prod_{\kappa \in \cJ} (u-\kappa(\pi))$. As a consequence, we can, for each $\kappa$, consider the $A[\frac{1}{p}]$-valued point
    $
    \Psi(\fM,\beta_0)_\kappa \in \operatorname{Gr}_{\leq h}^{(\kappa)}(A[\tfrac{1}{p}])$
    corresponding to the $\kappa$-th part of the submodule
    $$
    \fM_e \otimes_{\fS_A} \fS_{A[\frac{1}{p}]} \subset \varphi^*\fM \otimes_{\fS_A} \fS_{A[\frac{1}{p}]} \cong \fS_{A[\frac{1}{p}]}^d
    $$
    with the right-hand isomorphism induced by $\beta_0$. 
    
    The following proposition refines the notion from Definition~\ref{def-Hodge types} that $\fM$ has Hodge type $\lambda$, and is the fundamental property on which all the results in Section~\ref{sec-mainresults} depend.

\begin{prop}\label{prop-translate in flag}
    For each $\kappa(i,\tau)$ there exists $X_\kappa \in L^{+,(\kappa)}G(A[\frac{1}{p}])$ such that 
    \begin{itemize}
        \item $X_\kappa \cdot \Psi(\fM,\beta_0)_\kappa \in \operatorname{FL}_{\lambda_\kappa}[\frac{1}{p}] \subset \operatorname{Gr}_{\leq h}^{(\kappa)}$.
        \item If $X = (X_\kappa)_{\kappa \in \cJ} \in \operatorname{GL}_d(\fS_{A[\frac{1}{p}]})$ then the derivation $\varphi^* N_\nabla$ from Section~\ref{sec-creating integral monodromy} can be written, with respect to the basis $\beta_0$, as $\varphi^*N_\nabla = \frac{\varphi(E(u))}{p} X^{-1} u\frac{d}{du}(X) +  \frac{\varphi(E(u))}{p}u\frac{d}{du}$.
    \end{itemize}
\end{prop}
\begin{proof}
    As explained in \cite[\S6]{B25} this follows from Kisin's initial construction in \cite{Kis05} of the Breuil--Kisin modules associated to crystalline representations. Note, however, that the point $\Psi(\fM,\beta_0)_\kappa$ defined here is a twist of that in \cite{B25} since we use $\fM_e$ in place of $\varphi_{\fM}^{-1}(\fM)$. This mirrors the twist by $h$ appearing in the embedding  $\operatorname{FL}_{\lambda_\kappa} \hookrightarrow \operatorname{Gr}_{\leq \lambda_\kappa^*}^{(\kappa)}$ from \eqref{eq-twisted embedding of FL}. This discussion shows additionally that $X$ is determined, up to right translation by a constant matrix in $\operatorname{GL}_d(A[\frac{1}{p}])$, by the identity $\beta_0 = \overline{\beta} X$ for some choice of $A[\frac{1}{p}]$-basis $\overline{\beta}$ of $\varphi^*D$ under the Frobenius twist of \eqref{eq-M and D comparison}.

    The statement can also be deduced as follows. The assertion in \eqref{eq-Griffiths trans} that $E(u)\varphi^*N_\nabla$ stabilises $\fM_e \otimes_{\fS} \mathcal{O}^{\operatorname{rig}}[\frac{1}{\varphi(\lambda)}]$ combined with Proposition~\ref{prop-loop fixed and loop derivative fixed} implies the existence of $X_\kappa$ satisfying the second bullet point, and with $X_\kappa \cdot \Psi(\fM,\beta_0)_\kappa \in \operatorname{FL}_\nu[\frac{1}{p}]$ for some $\nu$. The description of what it means for $\fM$ to have Hodge type $\nu$ in Definition~\ref{def-Hodge types} ensures $\nu = \lambda_\kappa$. 
\end{proof}

The next lemma, which follows directly from the definitions, shows how to translate between the $A[\frac{1}{p}]$-valued point $\Psi(\fM,\beta_0)_\kappa$ and the $A$-valued point $\Psi(\fM_{i,\tau},\beta_{i-1})$.

\begin{lemma}\label{lem-translate between points in the grassmannian}
    Suppose that $g_{i-1} \in \operatorname{GL}_d(\fS_A[\frac{1}{(u-\pi_1)\ldots(u-\pi_{i-1})}])$ is such that $\beta_{i-1} = \beta_0 g_{i-1}$. If $\kappa = \kappa(i,\tau)$ then
    $$
    \Psi(\fM,\beta_0)_{\kappa} = g_{i-1,\kappa} \cdot \Psi(\fM_i,\beta_{i-1})[\tfrac{1}{p}]
    $$
    for $g_{i-1,\kappa} \in L^{+,(\kappa)}G(A[\frac{1}{p}])$ the $\kappa$-th part of $g_{i-1}$.
\end{lemma}

\begin{proof}[Proof of Proposition~\ref{prop-in schubert variety}]
 	Proposition~\ref{prop-translate in flag} and Lemma~\ref{lem-translate between points in the grassmannian} give
    \begin{equation}\label{eq-containment in flag after 1/p}
    \left( X_\kappa g_{i-1,\kappa} \right) \cdot \Psi(\fM_{i,\tau},\beta_{i-1})[\tfrac{1}{p}] \in \operatorname{FL}_{\lambda_\kappa}[\tfrac{1}{p}] 
    \end{equation}
   Thus, $\Psi(\fM_{i,\tau},\beta_{i-1})[\tfrac{1}{p}] \in \operatorname{Gr}_{\leq \lambda_\kappa^*}^{(\kappa)}(A[\frac{1}{p}])$ and so $\Psi(\fM_{i,\tau},\beta_{i-1}) \in \operatorname{Gr}_{\leq \lambda_\kappa^*}^{(\kappa)}(A)$ as claimed. \end{proof}

 \begin{proof}[Proof of Theorem~\ref{thm-integral equations1}]  
 	The first thing we check is that if \eqref{eq-key identity} holds for all $\tau$ then 
    $$
    \frac{E_i(u)}{u}N^\varphi(\fM_{i}) \subset \fM_i
    $$
    In particular, this means \eqref{eq-containment} holds with $i-1$ replaced by $i$. For this notice that, using the basis $\beta_{i-1}$, we can express the derivation $\frac{E_{i}(u)}{u}N^\varphi$ on $\fM_{i-1}$ as
    $$
    \cN^{(i)} + \Bigg( E_{i-1}(u) c^*(u) \Bigg) (u-\pi_i)\frac{d}{du}
    $$
    for an endomorphism $\cN^{(i)}$ of $\fM_{i-1}$. Note $\cN^{(i)} \equiv 0$ modulo $(u-\pi_i)$ since $\frac{E_{i-1}(u)}{u}N^\varphi$ already stabilises $\fM_{i-1}$. For each $\tau \in \cJ_0$ we can therefore apply Proposition~\ref{prop-loop fixed and loop derivative fixed} with $\kappa = \kappa(i,\tau)$ and $\alpha(u)$ equal the $\tau$-th part of $E_{i-1}(u)c^*(u)$. In particular, the implication $(2) \Rightarrow (3)$ of Proposition~\ref{prop-loop fixed and loop derivative fixed} ensures the $\tau$-th part of $\frac{E_i(u)}{u}(\fM_{i})$ lies inside the $\tau$-th part of $\fM_i$, as required.

    Since the containment in \eqref{eq-containment} is automatic when $i=1$  (as $N^\varphi \equiv 0$ modulo $u$ on $\varphi^*\fM$ from the definition in Proposition~\ref{prop-create monodromies}) an inductive argument on $1 \leq i \leq e$ therefore allows us to assume \eqref{eq-containment}. It just remains to prove \eqref{eq-key identity} for $i$.

    Combining \eqref{eq-containment in flag after 1/p} with the implication $(1) \Rightarrow (2)$ from Proposition~\ref{prop-loop fixed and loop derivative fixed} shows \begin{equation}\label{eq-section but Nnabla}
 		\left( \alpha(u)  \operatorname{dlog}(X_\kappa g_{i-1,\kappa}) +  \alpha(u) \partial\right)  \cdot v_{\fM,\kappa} =  \left( \alpha(u)|_{u = \kappa(\pi)}\sum_{i=1}^d  \sum_{j = \lambda_{\kappa,i}^*}^{h-1} j \right) v_{\fM,\kappa}
\end{equation}
    for $\alpha(u)$ equal the $\tau$-th part of $c^*(u) E_{i-1}(u)$ and $\partial = (u-\kappa(\pi))\frac{d}{du}$. Recall that here we are using the basis $\beta_{i-1}$ to view $\alpha(u)  \operatorname{dlog}(X_\kappa g_{i-1,\kappa}) +  \alpha(u) \partial$ as a derivation on $\fM_{i-1,\tau} \cong A[[u-\kappa(\pi)]]^d$. To prove the theorem we need to identify this derivation with $\frac{E_{i}}{u}N^\varphi$. For this we compute
 	\begin{equation*}
    \begin{aligned}
        	\varphi^* N_\nabla (\beta_{i-1}) &= \varphi^*N_\nabla(\beta_0) g_{i-1} +  \tfrac{\varphi(E(u))}{p} \beta_0 \left(u\tfrac{d}{du}\right)(g_{i-1}) \\
            & =  \beta_0 \tfrac{\varphi(E(u))}{p}X^{-1} \left(u\tfrac{d}{du} \right)(X) g_{i-1} +  \tfrac{\varphi(E(u))}{p} \beta_0 \left(u\tfrac{d}{du}\right)(g_{i-1}) \\ 
            &= \beta_{i-1} (Xg_{i-1})^{-1} \left( \tfrac{\varphi(E(u))}{p}u\tfrac{d}{du} \right)(Xg_{i-1})
    \end{aligned}
 	\end{equation*}
    where the second equality is using point two from Proposition~\ref{prop-translate in flag}. Multiplying this identity by $\frac{E_i(u)}{u}$ and using that $N^\varphi \equiv \varphi^*N_{\nabla}$ modulo $E(u)^h \varphi^*\fM \otimes_{\fS_A} \fS_{A[\frac{1}{p}]}$ and $c^*(u) \equiv \frac{\varphi(E(u))}{u}$ modulo $E(u)^h$ gives 
    $$
    \frac{E_i(u)}{u} N^\varphi(\beta_{i-1}) \equiv \beta_{i-1} (Xg_{i-1})^{-1} \left( c^*(u) E_{i-1}(u) (u-\pi_i) \tfrac{d}{du} \right)(X g_{i-1})    $$
    modulo $(u-\pi_i)^h \fM_{i-1} \otimes_{\fS_A} \fS_{A[\frac{1}{p}]}$.
    Passing to the $\kappa$-th part of this identity shows $\operatorname{dlog}(X_\kappa g_{i-1,\kappa}) +  \alpha(u) \partial \equiv \frac{E_{i}(u)}{u}N^\varphi$ modulo $(u-\kappa(\pi))^h$ as derivations on $\fM_{i-1,\tau}$. We can therefore substitute $\frac{E_i(u)}{u}N^\varphi$ into \eqref{eq-section but Nnabla} which proves the theorem.
 \end{proof}

\begin{proof}[Proof of the equivalence of Theorem~\ref{thm-integral equations1} and Theorem~\ref{thm-integral equations2}]
    This follows from Lemma~\ref{lem-eigenvector vs equations}, applied with $\alpha(u) = c^*(u) E_{i-1}(u)$ and $\cN=  \frac{E_i(u)}{u} N^\varphi$. Note that this is valid since $\cN \equiv 0$ modulo $u-\kappa(\pi)$.
\end{proof}

\section{Equations for quasi-minuscule Hodge types}

 Here we maintain the setup from Section~\ref{sec-setup cont} and fix a Hodge type $\lambda$ concentrated in degrees $[0,h]$ with $h\leq p-1$. Theorem~\ref{thm-embedding} therefore yields a closed immersion $\cY^{\operatorname{cr}}_{\leq h} \otimes_{\bZ_p} \bF \rightarrow Y^\nabla_{\leq h} \otimes_{\bF_p} \bF$. In this section we add convolutions structures to this morphism and, by imposing mod~$p$ versions of the equations from Theorem~\ref{thm-integral equations2}, give a more refined description of the image of $\cY^{\operatorname{cr},\operatorname{conv}}_{\lambda}$.

\subsection{New loci in the affine Grassmannian}\label{sec-loci in aff grass}

First, we consider an arbitrary tuple of integers $\mu = (\mu_{1} \geq \ldots \geq \mu_{d})$ contained in $[0,h]$. To fit the conventions from Section set $\mu^* = (\mu_1^* \geq \ldots \geq \mu_d^*)$ with $
\mu_i^* = h - \mu_{d-i}$ and consider  $\operatorname{Gr}_{\leq \mu^*} := \operatorname{Gr}_{\leq \mu^*}^{(\kappa)} \otimes_{\cO} \bF$
(which is independent of the choice of $\kappa$). We also set $\mathfrak{g} = \operatorname{Lie}G$ and write $\mathfrak{g}[u]_{< h-1} \subset \mathfrak{g}[[u]]$ for the affine space of polynomials in $\mathfrak{g}$ of degree $< h-1$. Then, for each $1 \leq i \leq e$ and each unit $c(u) \in \bF[[u]]^\times$ we use the equations in Theorem~\ref{thm-integral equations2} to define closed subschemes
$$
\mathcal{S}_{\mu,i,c(u)} \subset \operatorname{Gr}_{\leq \mu^*} \times u\mathfrak{g}[u]_{<h-1}
$$
Specifically, $\mathcal{S}_{\mu,i,c(u)}$ is defined as the locus of $(\cE,\cN)$ on which the the linear functionals
$$
\Big( \cN + c(u) u^i \tfrac{d}{du} \Big) \cdot w - \begin{cases}
     \left( c(u)\Big|_{u =0} \sum_{l=1}^d  \sum_{j = 1}^{\mu_{l}^*-1} j \right) w & \text{ if $i =1$} \\
     0 & \text{ if $i >1$}
\end{cases}
$$
vanish, for all $w \in \bigwedge^r \cV_h$ on which the constant subgroup $\bG_m \subset \operatorname{Aut}^+$ acts with weight $ < \sum_{l=1}^d \sum_{j=1}^{\mu_{l}^*-1} j $ and where, as in \eqref{eq-pluckerembedding}, $r = \mu_1^*+\ldots+\mu_d^*$. As the following lemma makes precise, these closed subschemes describe a family discretely interpolating between constant flag varieties and loci in the affine Springer space.

\begin{lemma}\label{lem-flag to springer}
    \begin{enumerate}
        \item If $i =1$ and $\sum_{\ell=1}^d \sum_{j = \nu_\ell^*}^{h-1} j < p$ for all $\nu \leq \mu$ then $\mathcal{S}_{\mu,i,c(u)} = \operatorname{FL}_{\mu}$ via the embedding in \eqref{eq-twisted embedding of FL}.
        \item If $i \geq h$ then $\cS_{\mu,i,c(u)}$ is contained in the affine Springer locus consisting of $(\cE,\cN)$ with $\cN(\cE) \subset \cE$.
    \end{enumerate} 
\end{lemma}
\begin{proof}
    The first part follows from Corollary~\ref{cor-equation for flag variety} and the second from the implication $(2) \Rightarrow (3)$ in Proposition~\ref{prop-loop fixed and loop derivative fixed}.
\end{proof}

\begin{exam}\label{exam}
    Here we describe the equations cutting out $S_{\mu,i,c(u)}$ explicitly in the case  $d=3$, $h=2$, and $\mu = (2,1,0)$. In this case $\mu^*=\mu$ and $r=3$. Write $e_1,e_2,e_3$ for the standard basis of $\bF[[u]]^3$ and  index the standard $\bF$-basis of $\cV_h = A[[u]]^3/ u^2A[[u]]^3$ as 
    $$
    (y_0,y_1,y_2,y_3,y_4,y_5) = (e_1,e_2,e_3,ue_1,ue_2,ue_3) 
    $$
    and write $y_{i,j,k} := y_i \wedge y_j \wedge y_k \in \bigwedge^r \cV_h$. Notice $y_{0,1,2} \in \bigwedge^3 \cV_h$ spans the unique line on which $\bG_m \subset \operatorname{Aut}^+$ acts with weight $< 1 = \sum_{l=1}^d \sum_{j=1}^{\mu_{l}^*-1} j$. Thus, $\cS_{\mu,i,c(u)}$ is the vanishing of the a single equation. Concretely, these are equations can be computed as:
    \begin{itemize}
        \item For each $i \geq 1$,  $\cS_{\mu,i,c(u)}$ is the vanishing locus of the section
    $$    
    \begin{aligned}
     y_{1,2,3}n_{1,1}+y_{1,2,4}n_{2,1} + y_{1,2,5} n_{3,1}-y_{0,2,3}n_{1,2} -y_{0,2,4}n_{2,2} \\-y_{0,2,5}n_{3,2}
     +y_{0,1,3}n_{1,3}+y_{0,1,4}n_{2,3}+y_{0,1,5}n_{3,3} \\
     +\begin{cases}
         y_{0,1,2}c(u)\Big|_{u=0}
 & \text{ if $i=1$} \\
 0 & \text{ if $i\geq 2$}
     \end{cases}\end{aligned}  
    $$
    where $n_{ij}$ denotes the function on $(\cE,\cN)$ with $\cN = u\big( \begin{smallmatrix}
        n_{1,1} & n_{1,2} & n_{1,3} \\ n_{2,1} & n_{2,2} & n_{2,3} \\
        n_{3,1} & n_{3,2} & n_{3,3}
    \end{smallmatrix}\big)$ modulo $u^2$.
    \end{itemize}
    For this specific $\mu$ one can check that the $\cS_{\mu,i,c(u)}$ is exactly the underlying reduced locus in $\operatorname{Gr}_{\leq \mu^*} \times u\mathfrak{g}[u]_{\leq h-1}$ of the affine Springer locus defined by $\cN(\cE) \subset \cE$.
\end{exam}

 \subsection{Equations on the special fibre}
 
 It is convenient to work inside the following ambient space: Let $Z^{\operatorname{conv}}_{\leq \lambda}$  denote the algebraic stack over $\operatorname{Spec} \bF$ whose $A$-points classify tuples $(\fM,N_0)$ as follows where
 \begin{itemize}
     \item $\fM$ is a Breuil--Kisin module over $A$ of height $\leq h$ with convolution structure $\fM_\bullet$ such that, Zariski locally on $\operatorname{Spec}A$, $\fM_i$ admits $\fS_A$-bases $\beta_i$ with
 $$
	\Psi(\fM_{i,\tau},\beta_{i-1}) \in \operatorname{Gr}_{\leq \lambda_\kappa^*}^{(\kappa)}(A)
$$
for all $\kappa = \kappa(i,\tau)$. Here $\Psi(\fM_{i,\tau},\beta_{i-1})$ is defined exactly as in Construction~\ref{const-definition of Psi}, with the usual convention that $\beta_0 := \varphi_{\fM}(E(u)^{-h}\beta_e) \otimes 1$, while $\lambda_\kappa^*$ is as in Proposition~\ref{prop-in schubert variety}.
\item $N_0$ an $\fS_A$-linear endomorphism of $\fM/u^{e+1}\fM$ with $N_0\equiv 0$ modulo $u\fM$
 \end{itemize}
 Let $\widetilde{Z^{\operatorname{conv}}_{\leq \lambda}}$ denote the $\bF$-scheme over $Z^{\operatorname{conv}}_{\leq \lambda}$ with $A$-points classifying tuples $(\fM,N_0,\beta_\bullet)$ with $(\fM,N_0) \in Z^{\operatorname{conv}}_{\leq \lambda}$ and $\beta_i$ an $\fS_A$-basis of $\fM_i$ for each $1 \leq i \leq e$. Then $\widetilde{Z^{\operatorname{conv}}_{\leq \lambda}}$ is a $\prod_{\kappa:K \hookrightarrow\cO[\frac{1}{p}]} L^{+,(\kappa)}G$-torsor over $Z^{\operatorname{conv}}_{\leq \lambda}$,  with the group operating on the bases elements $\beta_\bullet$.

\begin{prop}\label{prop-cm and dim}
    $Z^{\operatorname{conv}}_{\leq \lambda}$ is Cohen--Macaulay of dimension $( \sum_{\kappa \in \cJ} \left(   \operatorname{dim} \operatorname{Gr}_{\leq \lambda_\kappa^*}^{(\kappa)} \otimes_{\cO} \bF \right)) + \sum_{\tau \in \cJ_0} e d^2$. 
\end{prop}
\begin{proof}
   This is standard, so we just sketch the argument. Consider the map
   $$
\widetilde{Z^{\operatorname{conv}}_{\leq \lambda}} \rightarrow \prod_{\kappa \in \cJ} \left( \operatorname{Gr}_{\leq \lambda_\kappa^*}^{(\kappa)} \otimes_{\cO} \bF \right) \times \operatorname{Mat}_d(\fS_{\bF}) / u^e \operatorname{Mat}_{d}(\fS_{\bF})
   $$
   given by  $(\fM,N_0,\beta_i) \mapsto ((\Psi(\fM_{i,\tau},\beta_{i-1}))_{\kappa(i, \tau)},u^{-1}N_0)$. 
   For $N>>0$ relative to $\lambda$ this map factors through the quotient $\widetilde{Z^{\operatorname{conv}}_{\leq \lambda}} / \left( \prod_{\kappa \in \cJ} \cK_N^{(\kappa)} \right)$ where $\cK_N^{(\kappa)} \subset L^{+,(\kappa)}G$ is the congruence subgroup consisting of $g \equiv 1$ modulo $(u-\kappa(\pi))^N$. The basic observation, first made in \cite{PR09}, is that, possibly after further increasing $N$, the resulting factorisation is an $\prod_{\kappa \in \cJ} \left( L^{+,(\kappa)}G / \cK_N^{(\kappa)} \right)$-torsor for a new action of this group on $\widetilde{Z^{\operatorname{conv}}_{\leq \lambda}} /\cK_N$. See also \cite[9.7]{B24} for an account with similar notation to that considered here. It follows that $\widetilde{Z^{\operatorname{conv}}_{\leq \lambda}} /\left( \prod_{\kappa\in \cJ} \cK_N^{(\kappa)} \right)$ is a smooth cover of $\left(\prod_{\kappa \in \cJ} \operatorname{Gr}_{\leq \lambda_\kappa^*}^{(\kappa)} \otimes_{\cO} \bF \right) \times \operatorname{Mat}_d(\fS_{\bF}) / u^e \operatorname{Mat}_d(\fS_{\bF})$ with relative dimension $\operatorname{dim}_{\cO} \prod_{\kappa \in \cJ} \left( L^{+,(\kappa)}G / \cK_N^{(\kappa)} \right)$. This gives the claimed dimension of $Z^{\operatorname{conv}}_{\leq \lambda}$, and the Cohen--Macaulayness follows from that of $\operatorname{Gr}_{\leq \lambda_\kappa^*}^{(\kappa)} \otimes_{\cO} \bF$ in Proposition~\ref{prop-schubert vars geom}.
\end{proof}


\begin{const}\label{con-impose mod p equations}
      Just as in Definition~\ref{def-BK+monodromy stack}, if $(\fM,N_0) \in Z^{\operatorname{conv}}_{\leq \lambda}$ then we write $N_0^{\varphi}$ for the derivation on $\varphi^*\fM/u^{ep+1}\varphi^*\fM$ over $c(u)u\frac{d}{du}$ given by $m \otimes f \mapsto N_0(m) \otimes f + m \otimes c(u)u\frac{d}{du}(f)  $. We can then impose the following closed conditions on $\widetilde{Z_{\leq \lambda}^{\operatorname{conv}}}$ for $1 \leq i \leq e$:
	 \begin{itemize}
        \item[$(\mathbf{A}_i)$]The derivation $u^{i-1} N_0^{\varphi}$ on $\varphi^*\fM/ u^{ep+i}\varphi^*\fM$ stabilises $\fM_i/u^{ep+i}\varphi^*\fM$. 
	 	\item[$(\mathbf{B}_i)$] Assume $(\mathbf{A}_{i-1})$, so that $u^{i-1}N_0^{\varphi}$ is a derivation on $\fM_{i-1} / u^{ep+i} \varphi^*\fM$ which is $\equiv 0$ modulo $u\fM_{i-1} / u^{ep+i}\varphi^*\fM$  (note this is automatic if $i=1$). Thus, for each $\kappa = \kappa(i,\tau)$, there is a morphism 
        $$
        \widetilde{Z^{\operatorname{conv}}_{\leq \lambda}} \rightarrow \operatorname{Gr}_{\leq \lambda_\kappa^*} \times u \mathfrak{g}[u]_{< h-1}
        $$
        given by $(\fM,N_0,\beta_{\bullet}) \mapsto (\Psi(\fM_{i,\tau},\beta_{i-1}), \cN_{i-1,\tau})$ where $\cN_{i-1,\tau} \in u\mathfrak{g}[u]_{\leq h-1}$ is the reduction modulo $u^h$ of the matrix expressing the action of $u^{i-1}N_0^{\varphi}$ on the $\tau$-th part of $\beta_{i-1}$. Then condition $(\mathbf{B}_i)$ asks that $(\fM,N_0,\beta_\bullet)$ lies in the pullback of $S_{\lambda_\kappa,i,c(u)}$.
	 	\item[$(\mathbf{C})$] Assume $(\mathbf{A}_e)$. For any (equivalently, one) derivation $N$ on $\fM$ over $u^{e+1}\frac{d}{du}$ lifting $N_0$ one has a congruence
	 	$$
	 	u^e N_0^{\varphi} \equiv \varphi_{\fM}^{-1} \circ c(u)N \circ \varphi_{\fM} \text{ modulo } u^{e+1}\fM_e
	 	$$
	 	of operators on $\fM_e$.
	 \end{itemize}
\end{const}

     \begin{lemma}\label{lem-props of closed conditions}
         \begin{enumerate}
             \item $(\mathbf{B}_i) \Rightarrow (\mathbf{A}_i)$ for each $1 \leq i \leq e$.
             \item Assume that $\sum_{i=1}^d \sum_{j = \nu_i}^{h-1} j < p$ for all $\nu\leq \lambda_\kappa$ for $\kappa = \kappa(1,\tau)$. Then $(\mathbf{B}_1)$ implies $\Psi(\fM_{1,\tau},\beta_0)$ lies in the closed subscheme $\operatorname{FL}_{\lambda_\kappa} \subset \operatorname{Gr}_{\leq\lambda_\kappa^*}^{(\kappa)}$.
             \item Each of $(\mathbf{A}_i)$ and $(\mathbf{B}_i)$ are stable under the $\prod_{\kappa \in \cJ} L^{+,(\kappa)}G$-action on $\widetilde{Z^{\operatorname{conv}}_{\leq \lambda}}$ acting on the bases $\beta_\bullet$.
         \end{enumerate}
     \end{lemma}
     \begin{proof}
         Applying Lemma~\ref{lem-eigenvector vs equations} shows that $(\mathbf{B}_i)$ is equivalent to asking that, Zariski locally on $\operatorname{Spec}A$,
         $$
         \cN_{i-1,\tau} \cdot v_{\cE} = \left( \left( c(u)u^{i-1} \right)|_{u=\kappa(\pi)} \sum_{l=1}^d  \sum_{j = \lambda_{\kappa,l}^*}^{h-1} j \right) v_{\cE}
         $$
         for $v_{\cE}$ spanning the line $\Theta_\kappa(\Psi(\fM_{i,\tau},\beta_{i-1}))$. Then (1) follows from the implication $(2) \Rightarrow (3)$ in Proposition~\ref{prop-loop fixed and loop derivative fixed}, while (2) follows from Corollary~\ref{cor-equation for flag variety}. It only remains to check the stability in (3). For $(\mathbf{A}_i)$ this is clear. For $(\mathbf{B}_i)$ take $g = (g_{i,\tau})_{\kappa=\kappa(i,\tau)} \in \prod_{\kappa \in \cJ} L^{+,(\kappa)}G$ and note that $\cN_{i-1,\tau} \cdot v_{\cE} = C v_{\cE}$ for a constant $C$ implies 
         $$
         \left( g_{i-1,\tau} \cN_{i-1,\tau} g_{i-1,\tau}^{-1} \right) \cdot v_{g_{i-1,\tau} \cdot \cE} = C v_{g_{i-1,\tau} \cdot \cE}
         $$
         This proves the desired stability since, under the action of $g$ on $(\fM,N_0,\beta_\bullet)$, the data of $\Psi(\fM_{i-1,\tau},\beta_{i-1})$ and $\cN_{i-1,\tau}$ transforms to $g_{i-1,\tau} \cdot \Psi(\fM_{i-1,\tau},\beta_{i-1})$ and $g_{i-1,\tau} \cN_{i-1,\tau} g_{i-1,\tau}^{-1}$.
     \end{proof}
    
\begin{defn}\label{def-imposing mod p conditions}
    Consider the closed subschemes
    $$
    \widetilde{Y^{\nabla,\operatorname{conv}}_{\lambda}} \subset \widetilde{Y^{\nabla,\operatorname{conv}}_{\leq \lambda}} \subset \widetilde{Z^{\operatorname{conv}}_{\leq \lambda}}
    $$
    with the former defined using part (1) in Lemma~\ref{lem-props of closed conditions} to inductively impose $(\mathbf{B}_i)$ for $1 \leq i \leq e$ and $(\mathbf{C})$, and the latter defined by imposing $(\mathbf{B}_1)$, $(\mathbf{A}_i)$ for $2 \leq i \leq e$, and $(\mathbf{C})$. By part (3) of Lemma~\ref{lem-props of closed conditions} each of these subschemes is $\prod_{\kappa \in \cJ} L^{+,(\kappa)}G$-stable and hence descends to closed substacks
    $$
    Y^{\nabla,\operatorname{conv}}_{\lambda} \subset Y^{\nabla,\operatorname{conv}}_{\leq \lambda} \subset Z^{\operatorname{conv}}_{\leq \lambda}
    $$
    
\end{defn}
\begin{rmk}
    The stacks $Y_{\leq \lambda}^{\nabla,\operatorname{conv}}$ are introduced in Definition~\ref{def-imposing mod p conditions} primarily for convenience. As we will see later when $d=3$, we expect that the underlying reduced substacks of $Y^{\nabla,\operatorname{conv}}_{\leq \lambda}$ and $Y^{\nabla,\operatorname{conv}}_\lambda$ coincide when $\lambda_\kappa = (d-1,d-2,\ldots,1,0)$ (though one should not expect $Y^{\nabla,\operatorname{conv}}_{\leq \lambda}$ itself to be reduced unless $\lambda$ is minuscule). Thus, in these cases $Y^{\nabla,\operatorname{conv}}_{\leq \lambda}$ can be used to control topological aspects of $Y^{\nabla,\operatorname{conv}}_\lambda$. Note, one does not expect this to be the case for general $\lambda$ since it contradicts the Breuil--M\'ezard conjecture.
\end{rmk}
    
 \begin{prop}\label{prop-factoristion modulo p}
	Suppose $h \leq p-1$. Then, for each Hodge type $\lambda$ concentrated in degree $[0,h]$, there is a monomorphism
    $
\cY^{\operatorname{cr},\operatorname{conv}}_{\lambda} \otimes_{\cO} \bF \rightarrow Y^{\nabla,\operatorname{conv}}_{\lambda}$
fitting into the commutative diagram
$$
\begin{tikzcd}
    \cY^{\operatorname{cr},\operatorname{conv}}_{\lambda} \otimes_{\cO} \bF \ar[r] \ar[d] & Y^{\nabla,\operatorname{conv}}_{\lambda} \ar[d] \\
    \cY^{\operatorname{cr}}_{\leq h} \otimes_{\bZ_p} \bF \ar[r] & Y^{\nabla}_{\leq h} \otimes_{\bF_p} \bF
\end{tikzcd}
$$
whose vertical arrows forget convolution structures, and whose bottom horizontal arrow is the base change of the monomorphism from Theorem~\ref{thm-embedding}.
\end{prop}
\begin{proof}
	We immediately get a commutative diagram as claimed but with the top arrow replaced by the morphism $\cY^{\operatorname{cr},\operatorname{conv}}_\lambda \otimes_{\cO} \bF \rightarrow Y^{\nabla,\operatorname{conv}}_{\leq h}$ induced from Theorem~\ref{thm-embedding}, with $Y^{\nabla,\operatorname{conv}}_{\leq h} \subset Z_{\leq h}^{\operatorname{conv}}$ the locus defined by the conditions $(\mathbf{A}_e)$ and $(\mathbf{C})$. Indeed, $(\mathbf{A}_e)$ and $(\mathbf{C})$ are exactly the conditions defining $Y^{\nabla}_{\leq h}$ in Definition~\ref{def-BK+monodromy stack}.
    
    It therefore suffices to show that, on the level of points valued in a finite $\bF$-algebra $A$, this morphism factors through $Y^{\nabla,\operatorname{conv}}_\lambda$. 
    Fix $\fM \in \cY^{\operatorname{cr},\operatorname{conv}}_\lambda(A)$ which is mapped onto $(\fM,N_0) \in Y^{\nabla,\operatorname{conv}}_{\leq h}$. The claimed factorisation can be checked Zariski locally, and so we can assume the $\fM_\bullet$ admits $\fS_A$-bases $\beta_\bullet$. We saw in Section~\ref{sec-proofs} that $N_0^{\varphi}$  (as in Definition~\ref{def-BK+monodromy stack}) can be obtained by lifting $\fM$ to $\fM^\circ \in\cY^{\operatorname{cr},\operatorname{conv}}_\lambda(A^\circ)$ for $A^\circ$ finite flat over $\cO$ and setting  $N_0^{\varphi} = N^{\circ,\varphi} \otimes_{\cO} \bF$ modulo $u^{ep+1}\varphi^*\fM$ with $N^{\circ,\varphi}$ a derivation as in part (2) of Proposition~\ref{prop-create monodromies}. 
    On the other hand, applying $\otimes_{\cO} \bF$ to Theorem~\ref{thm-integral equations2} shows $\Psi(\fM_{i,\tau},\beta_{i-1})$ lies in the zero locus of
    $$
	 	u^{i-1} N^{\circ,\varphi} \otimes_{\cO} \bF \cdot v -\left(  \left(  c(u)u^{i-1} \right)|_{u=\kappa(\pi)} \sum_{l=1}^d  \sum_{j = 1}^{\lambda_{\kappa,l}^*-1} j \right) v 	
    $$
	 	for all $v \in \bigwedge^r \cV_h$ on which the constant subgroup $\bG_m \subset \operatorname{Aut}^+$ acts with weight $< \sum_{l=1}^d \sum_{j=1}^{\lambda_{\kappa,l}^*-1} j$. Since $u^{i-1} N^{\circ,\varphi} \otimes_{\cO} \bF$ lifts $u^{i-1}N_0^{\varphi}$ we conclude that $(\fM,N_0,\beta_\bullet)$ satisfies $(\mathbf{B}_i)$ and hence that $(\fM,N_0) \in Y^{\nabla}_\lambda(A)$. 
        \end{proof}

 \subsection{Quasi-minuscule coweights}

	We do not know whether the morphism in Proposition~\ref{prop-factoristion modulo p} is generally an isomorphism. However, for sufficiently constrained $\lambda$ we can show this is the case under sufficient topological control on $Y^{\nabla,\operatorname{conv}}_{\leq \lambda}$. Specifically we consider $\lambda$ which are minuscule/quasi-minuscule at each $\kappa$, so that the only $\nu \leq  \lambda_\kappa$ in the Bruhat order is a constant tuple of integers. Concretely, this means that, up to a twist, we can take
	$$
	\lambda_\kappa^* = (1,\ldots,1,0,\ldots,0) \qquad \text{or} \qquad \lambda_\kappa^* = (2,1,\ldots,1,0)
	$$
	for each $\kappa \in \cJ$. The significance of this restriction is that 
	$$
	D_\kappa := \codim(\operatorname{FL}_{\lambda_\kappa},\operatorname{Gr}_{\leq \lambda_\kappa^*}^{(\kappa)}) = \begin{cases}
		1 & \text{ if $\lambda_\kappa^* = (2,1,\ldots,1,0)$} \\
		0 & \text{ if $\lambda_\kappa^* = (1,\ldots,1,0,\ldots,0)$}
	\end{cases}
	$$
    This follows from a comparison of well-known dimension formulas for $\operatorname{FL}_{\lambda_\kappa}$ and $\operatorname{Gr}_{\leq \lambda_\kappa^*}^{(\kappa)}$. It also follows from a combination of Corollary~\ref{cor-equation for flag variety} and Lemma~\ref{lem-eigenvector vs equations}, which assert that $\operatorname{FL}_{\lambda_\kappa}[\frac{1}{p}] \subset \operatorname{Gr}_{\leq \lambda_\kappa^*}^{(\kappa)}[\frac{1}{p}]$ is cut out by the vanishing of equations indexed by $v \in \bigwedge^r\cV_h$ with $\bG_m$-weight $< \sum_{i=1}^d \sum_{j=1}^{\lambda_{\kappa,i}^*-1} j$. Indeed, if $\lambda_\kappa^* = (1,\ldots,1,0,\ldots,0)$ then $\sum_{l=1}^d \sum_{j=1}^{\lambda_{\kappa,l}^*-1} j = 0$, while if $\lambda_\kappa^* = (2,1,\ldots,1,0)$ then $\sum_{l=1}^d \sum_{j=1}^{\lambda_{\kappa,l}^*-1} j = 1$ and, since $r =d$, $\bigwedge^r \cV_h$ contains a unique line with $\bG_m$-weight $0$, namely that spanned by $e_1\wedge\ldots \wedge e_d$.
 	
\begin{thm}\label{thm-lci in schubert}
	Suppose $\lambda$ is minuscule/quasi-minuscule and $p \geq 3$. If $\operatorname{dim}Y^{\nabla,\operatorname{conv}}_{\lambda} = \operatorname{dim}\cY^{\operatorname{cr,conv}}_{\lambda} \otimes_{\cO} \bF$ then $Y^{\nabla,\operatorname{conv}}_{\lambda}$ is Cohen--Macaulay. If, in addition, $Y^{\nabla,\operatorname{conv}}_{\lambda}$ is irreducible and generically reduced then $Y^{\nabla,\operatorname{conv}}_{\lambda}$ is reduced and Proposition~\ref{prop-factoristion modulo p} induces an isomorphism
	$\cY^{\operatorname{cr,conv}}_{\lambda} \otimes_{\cO} \bF \cong Y^{\nabla,\operatorname{conv}}_{\lambda}$.
\end{thm}
\begin{proof}
 Since $\lambda$ is minuscule/quasi-minuscule, each $\lambda_\kappa$ is bounded in the interval $[0,2]$ and so all our results thus far apply whenever $2 \leq p-1$. Furthermore, Proposition~\ref{prop-cm and dim} implies
$$
\operatorname{dim} Y^{\nabla,\operatorname{conv}}_{\lambda} = \operatorname{dim} \cY^{\operatorname{cr},\operatorname{conv}}_\lambda \otimes_{\cO} \bF = \operatorname{dim} Z^{\nabla,\operatorname{conv}}_{
\lambda} -  \sum_{\kappa \in \cJ}  D_\kappa  - \sum_{\tau \in \cJ_0} d^2
$$ 
with $D_\kappa$ as defined above. The conditions $(\mathbf{B}_i)$ from Construction~\ref{con-impose mod p equations} for $1\leq i \leq e$ impose the vanishing  of $\sum_{\kappa \in \cJ} D_\kappa$ locally defined equations, while condition $(\mathbf{C})$ specifies the equality of two derivations on $\fM_e / u^{e+1} \fM_e$ which are $\equiv 0$ modulo $u$, and hence is equivalent to the vanishing of $\sum_{\tau \in cJ_0} d^2$ locally defined equations. Thus, $Y^{\nabla,\operatorname{conv}}_\lambda$ is a local complete intersection inside the Cohen--Macaulay algebraic stack $Z^{\operatorname{conv}}_{\leq \lambda}$. It is therefore itself Cohen--Macaulay. 

For the last assertion recall that generically reduced and Cohen--Macaulay algebraic stacks are reduced \cite[Tag 0344]{stacks-project}, while any closed immersion between two finite type algebraic stacks of the same dimension is necessarily an isomorphism if the target is irreducible and reduced.
\end{proof}

\section{Dimension bounds via an explicit model}

\subsection{Main results}

The purpose of this section is to prove the following:
\begin{thm}\label{thm-dimension bounds}
    Suppose $d =3$ and $\eta_\kappa = (2,1,0)$ for each $\kappa \in \cJ$. If $p\geq 5$ and $e \equiv 0$ modulo $3$ then $Y^{\nabla,\operatorname{conv}}_{\leq \eta}$ from Definition~\ref{def-imposing mod p conditions} has dimension $3e$ with a single top-dimensional irreducible component which is generically reduced.
\end{thm}

The proof of this theorem takes the entire section, but see Section~\ref{sec-strategy} for an overview of the argument. The assumption that $p \geq 5$ and $e \equiv 0$ modulo $3$ appear as artifacts of certain dimension estimates, and it seems plausible that the theorem remains true without these hypotheses (though recall we only establish a link between $Y^{\nabla,\operatorname{conv}}_{\leq \eta}$ and $\cY^{\operatorname{cr},\operatorname{conv}}_{\eta}$ when $p \geq 3$).  

\begin{cor}\label{cor-true = explicit}
    Under the assumptions in Theorem~\ref{thm-dimension bounds} the morphism in Proposition~\ref{prop-factoristion modulo p} induces an isomorphism $Y^{\nabla,\operatorname{conv}}_{\eta} \cong \cY^{\operatorname{cr,conv}}_\eta \otimes_{\cO} \bF$ of algebraic stacks. Furthermore, both are irreducible and reduced. 
\end{cor}
\begin{proof}
    Since $Y^{\nabla,\operatorname{conv}}_{\eta} \subset Y^{\nabla,\operatorname{conv}}_{\leq \eta}$ we have $\operatorname{dim}Y^{\nabla,\operatorname{conv}}_\eta \leq 3e$. By Theorem~\ref{thm-lci in schubert}, $Y^{\nabla,\operatorname{conv}}_\eta$ is Cohen--Macaulay of dimension $3e$; hence equidimensional. Therefore $Y^{\nabla,\operatorname{conv}}_\eta \subset Y^{\nabla,\operatorname{conv}}_{\leq \eta}$ is supported on the top-dimensional component of $Y^{\nabla,\operatorname{conv}}_{\leq \eta}$. In particular,  $Y^{\nabla,\operatorname{conv}}_{ \eta}$ is irreducible and generically reduced. Applying Theorem~\ref{thm-lci in schubert} again shows $Y^{\nabla,\operatorname{conv}}_{ \eta}$ is reduced and that Proposition~\ref{prop-factoristion modulo p} induces an isomorphism $Y^{\nabla,\operatorname{conv}}_{\eta} \cong \cY^{\operatorname{cr,conv}}_\eta \otimes_{\cO} \bF$.
\end{proof}

For the whole of this section we set $G = \operatorname{GL}_3$. We point out, however, that various elements of our analysis go through without this restriction. Since we also work entirely over $\bF$ we use similar notation to that in Section~\ref{sec-loci in aff grass}. Specifically, we write $L^{+}G := L^{+,(\kappa)}G \otimes_{\cO} \bF$ and, for a coweight $\mu = (\mu_1\geq \ldots \geq \mu_d)$, we set $\operatorname{Gr}_{\leq \mu} := \operatorname{Gr}_{\leq \mu}^{(\kappa)} \otimes_{\cO} \bF$ and $\operatorname{Gr}_\mu := \operatorname{Gr}_{\mu}^{(\kappa)} \otimes_{\cO} \bF$. For $N \geq 1$ set $\cK_N \subset L^+G$ equal to the congruence subgroups of matrices $\equiv 1$ modulo $u^N$. Set $\mathfrak{g} := \operatorname{Lie} G$ and $\mathfrak{g}[[u]] := \operatorname{Lie} L^+G$. Finally, we continue to write $\mathfrak{g}[u]_{< e} \subset \mathfrak{g}[[u]]$ for the affine space of polynomials in $\mathfrak{g}$ of degree $< e$.

\subsection{Strategy}\label{sec-strategy}
To prove Theorem~\ref{thm-dimension bounds} set $h =2$. Note that, in this case, $\eta_\kappa^* = \eta_\kappa = (2,1,0)$ for each $\kappa \in \cJ$. Consider the morphism  
\begin{equation}\label{eq-morphism forgetting convolution}
Y^{\nabla,\operatorname{conv}}_{\leq \eta} \rightarrow Y^{\nabla}_{\leq h}
\end{equation}
forgetting the convolution structure. The (reduced) image of this morphism admits a natural stratification indexed by tuples $\mu = (\mu_\tau)_{\tau \in \cJ_0}$ with
$$
\mu_\tau = (\mu_{\tau,1} \geq \mu_{\tau,2} \geq \mu_{\tau,3}) \leq \sum_{\kappa= \kappa(i,\tau)} \eta_\kappa = (2e,e,0)
$$
For such $\mu$ let $Y^\nabla(\mu) \subset Y^{\nabla}_{\leq h}$ denote the locally closed substack whose closed points consist of $\fM \in Y^{\nabla}_{\leq h}$ admitting a basis $\beta^0$ so that $\varphi_{\fM}( \beta^0 \otimes 1) = \beta^0 X$ for a matrix $X$ with $\tau$-th part $X_\tau \in L^+G u^{\mu_\tau} L^+G $. Then the image of \eqref{eq-morphism forgetting convolution} has underlying reduced substack equal to the union of all such $Y^\nabla(\mu)$. Our first step is to bound the dimension of these strata. For this the following notation is useful:

\begin{notation}\label{not-roots}
    Viewing $\mu_\tau$ as a cocharacter of the diagonal torus in $G = \operatorname{GL}_3$ and $\alpha,\beta$ as the simple positive roots  (relative to the upper triangular Borel) with $\gamma = \alpha+\beta$, allows us to write
$$
\mu_{\tau, *} = \langle *,\mu_\tau \rangle
$$
for $* \in \lbrace \alpha,\beta,\gamma \rbrace$ and $\langle-,-\rangle$ the standard evaluation pairing of characters and cocharacters.
\end{notation}

\begin{defn}   For each $\tau \in \cJ_0$ write $\mu_\tau = e(2,1,0) + n_\tau(-1,1,0) + m_\tau(0,-1,1)$ with $m_{\tau},n_\tau \geq 0$. We say $\mu_{\tau}$ is \emph{unbalanced} if one of the following holds:
    \begin{itemize}
        \item $\mu_{\tau,\alpha} = e - 2n_\tau +m_\tau > e$,
        \item $\mu_{\tau,\beta} = e + n_\tau -2m_\tau > e$,
        \item $\mu_{\tau,\gamma} = 2e - n_\tau -m_\tau < e$.
    \end{itemize}
Otherwise, we say that $\mu_{\tau}$ is \emph{balanced}.  
\end{defn}

In Section~\ref{sec-dim without conv} we prove the following by an explicit computation on an open cover.

\begin{thm}\label{thm-bounding dimension without convolution}
    For each $\tau \in \cJ_0$ write $\mu_\tau = e(2,1,0) + n_\tau(-1,1,0) + m_\tau(0,-1,1)$ with $m_{\tau},n_\tau \geq 0$ as above, and recall $h=2$. Then 
    $$
    \operatorname{dim} Y^\nabla(\mu) \leq \sum_{\tau \in \cJ_0} (3e - n_\tau - m_\tau) 
    $$
    with the inequality strict if there exists $\tau \in \cJ_0$ so that $\mu_\tau$ is unbalanced.

    Furthermore, if $\mu_{\tau} = (2e,e,0)$ for all $\tau \in \cJ_0$ then $Y^\nabla(\mu)$ contains a unique irreducible component of dimension $\sum_{\tau \in \cJ_0} 3e$, and this component is generically reduced.
\end{thm}

 Granting Theorem~\ref{thm-bounding dimension without convolution} we can prove Theorem~\ref{thm-dimension bounds} by bounding the fibres of \eqref{eq-morphism forgetting convolution} over each $Y^\nabla(\mu)$. Specifically, it suffices to show that
\begin{equation}\label{eq-basic convolution bound}
\sum_{\tau \in \cJ_0} n_\tau + m_\tau
\end{equation}
gives an upper bound for the dimension of the fibres of \eqref{eq-morphism forgetting convolution} over $Y^\nabla(\mu)$, with the bound strict on a dense open subset of $Y^\nabla(\mu)$ whenever each $\mu_\tau$ is balanced and $\mu_\tau \neq (2e,e,0)$ for each $\tau \in \cJ_0$. In fact, the bound in \eqref{eq-basic convolution bound} arises automatically because $\sum_{\tau \in \cJ_0} n_\tau + m_\tau$ is precisely the dimension of the locus
\begin{equation}\label{eq-convolution fibres}
m^{-1}_{\leq \eta_\bullet}(\fM) := 
\left\{ \fM_\bullet \middle|\ 
\begin{array}{l}
 \Psi(\fM_{i,\tau},\beta_{i-1}) \in \operatorname{Gr}_{\leq (2,1,0)} \text{ for any choice of }\\
\text{$\fS_A$-bases $\beta_\bullet$, and any $\tau\in \cJ_0$ and $1 \leq i \leq e$.} 
\end{array}
\right\}
\end{equation}
whenever $(\fM,N_0) \in Y^\nabla(\mu)$. Indeed, $m^{-1}_{\leq \eta_\bullet}(\fM)$ is a product over $\tau \in \cJ_0$ of  fibres of the convolution morphism in the affine Grassmannians, and the dimension of the latter are described in e.g.\ \cite[\S 2.1]{Haines06}. Thus, Theorem~\ref{thm-dimension bounds} follows from:

\begin{thm}\label{thm-dim bounds for fibres of convolution}
    Assume that $\mu_{\tau}$ is balanced for all $\tau \in \cJ_0$ and there exists $\tau \in \cJ_0$ with $\mu_\tau \neq (2e,e,0)$ and set $h=2$. Assume additionally that $e \equiv 0$ modulo $3$. Then there exists an open substack $U \subset Y^\nabla(\mu)$ such that
    \begin{itemize}
        \item The complement $Y_{\leq h}^\nabla(\mu) \setminus U$ has dimension $< \sum_{\tau \in \cJ_0} (3e - n_\tau - m_\tau)$.
        \item For each closed point $(\fM,N_0) \in U$ and each top-dimensional irreducible component $C \subset m^{-1}_{\leq \eta_\bullet}(\fM)$, there exists $\fM_\bullet \in C$ with $(\fM,N_0,\fM_\bullet) \not\in Y^{\nabla,\operatorname{conv}}_{\leq \eta}$.
    \end{itemize}
\end{thm}

This is proven in Section~\ref{sec-conv bounds}. The central ingredient is that the top-dimensional irreducible components of $m^{-1}_{\leq \eta_\bullet}(\fM)$ have an explicit (and, in our specific setting, very simple) description in terms of Mirkovic--Villonen cycles. The assumption $e \equiv 0$ modulo $3$ is likely unnecessary, but allows some difficult edge cases to be avoided.

\begin{proof}[Proof of Theorem~\ref{thm-dimension bounds} granting Theorem~\ref{thm-bounding dimension without convolution} and Theorem~\ref{thm-dim bounds for fibres of convolution}]
    Theorem~\ref{thm-bounding dimension without convolution} and Theorem~\ref{thm-dim bounds for fibres of convolution} together imply that the preimage of $Y^\nabla(\mu)$ under \eqref{eq-morphism forgetting convolution} has dimension $< 3e$ except when $\mu_\tau = (2e,e,0)$ for each $\tau$. Since \eqref{eq-morphism forgetting convolution} is an isomorphism over $Y^\nabla(\mu)$ when $\mu_\tau = (2e,e,0)$ for each $\tau$ it follows from the last part of Theorem~\ref{thm-bounding dimension without convolution} that $Y^\nabla_{\leq \eta}$ has a unique top-dimensional irreducible component of dimension $3e$, which is furthermore generically reduced.
\end{proof}

\subsection{Dimension bounds without convolution}\label{sec-dim without conv}

\subsubsection{Explicit coordinates}\label{sec-explicit equations}

We begin by temporarily dropping the assumptions $d=3$ and $h =2$ and giving an explicit description of $Y^{\nabla}_{\leq h}$ in terms of matrix equations. Suppose $\fM$ is a Breuil--Kisin module of height $\leq h$ over $A$ with a choice of $\fS_A$-basis $\beta^0$ and an endomorphism $N_0$ of $\fM/ u^{e+1}\fM$. Then there are matrices $X,\cN$ such that 
$$
\varphi_\fM(\beta_0) = \beta^0 X,\qquad N_0(\beta^0) \equiv \beta^0 \cN \mod u^{e+1}\fM
$$
for $\beta_0 := \beta^0 \otimes 1$, which is an $\fS_A$-basis of $\varphi^*\fM$. If $\beta_e := \varphi_{\fM}^{-1}(E(u)^h\beta^0)$, which is an $\fS_A$-basis of $\fM_e := \varphi_{\fM}^{-1}(E(u)^h\fM)$ then $\beta_e X_e = \beta_0  $ for $X_e = X E(u)^{-h}$. Thus
\begin{equation}\label{eq-phi^*N_0 of beta_e}
    u^eN_0^{\varphi}(\beta_e) \equiv \beta_e \bigg[ u^e X_e \varphi(\cN) X_e^{-1} - c(u) u^e\partial(X_e) X_e^{-1} \bigg] \mod u^{ep+1}\varphi^*\fM
\end{equation}
where $\partial = u\frac{d}{du}$. On the other hand, if $N$ is the derivation of $\fM$ over $u^e\partial$ with $N(\beta^0) = \beta^0 \cN$ then
\begin{equation}\label{eq-phi o N o phi^{-1}}
\varphi_{\fM} \circ c(u)N \circ \varphi_{\fM}^{-1}(\beta_e) = \beta_e \bigg[ c(u)\cN + c(u)u^e\partial(E(u)^h) E(u)^{-h} \bigg]
\end{equation}
Substituting $X = X_e E(u)^h$ into \eqref{eq-phi^*N_0 of beta_e} and then taking $\tau$-th parts therefore shows that $(\fM,N_0) \in Y^{\nabla}_{\leq h}(A)$ if and only if, 
\begin{equation}\label{eq-second explicit equation}
	u^e \bigg[ X_\tau \phi(\cN_{\tau \circ \varphi}) X_\tau^{-1} - c_\tau(u)\partial(X_\tau) X_\tau^{-1}  \bigg] \equiv  c_\tau(u)\cN_\tau  \mod u^{e+1} \mathfrak{g}[[u]]
\end{equation}
for each $\tau \in \cJ_0$, and for $\phi$ the $\bF$-linear endomorphism of $\mathfrak{g}[[u]]$ given by $u \mapsto u^p$. Here we use that the $\tau$-th part of $\varphi(\cN)$ equals $\phi(\cN_{\tau \circ \varphi})$. In particular, this gives

\begin{lemma}\label{lem-equations}
    Let $\widetilde{Y^\nabla(\mu)}$ denote the $\bF$-scheme with $A$-points classifying $(\fM,N_0) \in Y^{\nabla}(\mu)(A)$ together with an $\fS_A$-basis $\beta^0$ of $\fM$. Then there is a closed immersion
    $$
    \widetilde{Y^\nabla(\mu)}\rightarrow \prod_{\tau \in \cJ_0} \left( L^+G u^{\mu_\tau} L^+G \times u \mathfrak{g}[u]_{< e} \right)
    $$
    with image cut out by \eqref{eq-second explicit equation}.
\end{lemma}

    For later analysis it will be useful to simplify the equations in \eqref{eq-second explicit equation}. For convenience we return to the case $d=3$ (though the same can clearly be done for any $d$) and consider affine spaces $J_{\tau} \subset L^+G$ and $\mathfrak{n} \subset u\mathfrak{g}[u]_{\leq e-1}$ respectively consisting of matrices
$$
B_\tau = \begin{pmatrix}
    1 & 0 & 0 \\ 
    x_\tau & 1 & 0 \\ 
    z_\tau & y_\tau & 1
\end{pmatrix}, \qquad Y_\tau = \begin{pmatrix}
    0 & 0 & 0 \\ Y_{\tau,\alpha} & 0 & 0 \\ Y_{\tau,\gamma} & Y_{\tau,\beta} & 0 
\end{pmatrix}
$$
where
\begin{itemize}
    \item $x_\tau,y_\tau$, and $z_\tau$ are polynomials with degree $< \mu_{\tau,\alpha}$, $\mu_{\tau,\beta}$, and $\mu_{\tau,\gamma}$ respectively.
    \item $Y_{\tau,\alpha}$, $Y_{\tau,\beta}$, and $Y_{\tau,\gamma}$ are each polynomials of degree $\leq e$ with vanishing constant term.
\end{itemize}
We write $x_{\tau,i}$ for the coefficient of $u^i$ in $x_{\tau}$, and likewise with the other entries of these matrices, and interpret these coefficients as coordinates on these affine spaces. More generally, for any $3 \times 3$ matrix over $R(\!(u)\!)$, let $M_{\alpha}$ denote the $(21)$-entry and $M_{\alpha, i}$ denote the coefficient of $u^i$ in $M_{\alpha}$.  Similarly, let $M_{\beta}$ (resp. $M_{\gamma}$) denote the $(32)$ (resp. $(31)$)-entry.

We then consider the locus 
$$
\bN_\mu^\nabla \subset \prod_{\tau \in \cJ_0} \left( L^+G  \times J_\tau \times \mathfrak{n}  \right)
$$
consisting of $(h_\tau,B_\tau,Y_\tau)_{\tau \in \cJ_0}$ satisfying the closed condition
\begin{equation}\label{eq-defining Nlambda}
\operatorname{Ad}( u^{\mu_\tau}) u^e\bigg[\partial(B_\tau) B^{-1}_\tau+  B_\tau \phi\bigg( h_{\tau \circ \varphi}^{-1} d_{\tau \circ \varphi}(u) Y_{\tau \circ \varphi} h_{\tau \circ \varphi}\bigg) B_\tau^{-1}\bigg] = Y_\tau \mod u^{e+1}\mathfrak{g}[[u]]
\end{equation} 
for each $\tau \in \cJ_0$, where $d_\tau(u) \in \bF[[u]]^\times$ denotes the $\tau$-th part of $d(u) = \varphi^{-1}(c(u))^{-1} \in \fS_{\bF}$ for $c(u)$ defined in Section~\ref{sec-setup}. Note that $d_{\tau}(0) \neq 0$. 

\begin{lemma}\label{lem-open covers}
    Fix $g_\tau \in G$ for each $\tau \in \cJ_0$. Then there is an immersion 
    $$
    \bN_\mu^\nabla \rightarrow \prod_{\tau \in \cJ_0} \left( L^+G u^{\mu_\tau} L^+G \times u\mathfrak{g}[u]_{< e} \right)
    $$
    given by $(h_\tau,B_\tau,Y_\tau)_{\tau \in \cJ_0} \mapsto (h_\tau u^{\mu_\tau} B_\tau g_\tau, -h_\tau^{-1} Y_\tau h_\tau)$ which
    identifies $\bN_\mu^\nabla$ with an open subscheme of $\widetilde{Y^\nabla(\mu)}$. Furthermore, these images form an open cover of $\widetilde{Y^\nabla(\mu)}$ as $(g_\tau)_{\tau}$ varies in $G^{\cJ_0}$.
\end{lemma}
\begin{proof}
    Recall that $L^+Gu^{\mu_\tau} J_\tau g_\tau \subset L^+Gu^{\mu_\tau} L^+G$ is open and forms an open cover as $g_\tau$ runs over the Weyl group of $G$. In particular, the same holds for $g_\tau$ running over $G$. Consequently, it suffices to show that \eqref{eq-second explicit equation} and \eqref{eq-defining Nlambda} are equivalent after setting $X_\tau = h_\tau u^{\mu_\tau} B_\tau g_\tau$ and $\cN_\tau = - h_\tau^{-1} Y_\tau h_\tau$, which is a straightforward manipulation using that $\partial(g_\tau) =0$ whenever $g_\tau \in G$. Lastly, one observes that the dominance of $\mu_\tau$ forces any solution to \eqref{eq-defining Nlambda} with $Y_\tau \in u\mathfrak{g}[u]_{\leq e}$ to be nilpotent lower triangular.  
\end{proof}

Notice that if $N>>0$ then there is a left action of $\prod_{\tau \in \cJ_0} \cK_N$ on $\bN_\mu^\nabla$ given by $(k_\tau) \cdot (h_\tau, B_\tau,Y_\tau) = (k_\tau h_\tau, B_\tau,Y_\tau)$.   
\begin{cor}\label{cor-open cover}
    There is a cover $\lbrace U_{g} \rbrace$ indexed by $g = (g_\tau)_{\tau \in \cJ_0} \in \prod_{\tau \in \cJ_0} G(\overline{\bF})$ of $Y^\nabla(\mu)$ by open substacks such that, for each $g = (g_\tau)_{\tau \in \cJ_0}$ and $N>>0$,
    $$
    \left(\prod_{\tau \in \cJ_0} \cK_N \right) \backslash \bN_\mu^\nabla
    $$
    can be realised as an $\prod_{\tau \in \cJ_0} L^+G /\cK_N$-torsor over $U_g$.
\end{cor}
\begin{proof}
    This follows from the same standard argument employed in Proposition~\ref{prop-cm and dim}. Specifically, as in e.g.\ \cite[9.7]{B24} one shows that, for  sufficiently $N$, there is an isomorphism of quotient stacks
    $$
      \left( \prod_{\tau \in \cJ_0} \cK_N \right) \backslash  \widetilde{Y_{\leq h}^\nabla(\mu) } \cong \widetilde{Y_{\leq h}^\nabla(\mu)}/ \left( \prod_{\tau \in \cJ_0} \cK_N \right)
    $$
    where the right hand action is given by $ \prod_{\tau \in \cJ_0} \cK_N$ operating on the choice of basis $\beta^0$ and the left hand action identifies, under Lemma~\ref{lem-open covers}, with the left multiplication action on the $L^+G u^{\mu_\tau} L^+G$-th factor. Tracing through the substitutions just made shows this coincides with the left multiplication action on $\bN_\mu^\nabla$. Clearly, $\left( \prod_{\tau \in \cJ_0} \cK_N \right) \backslash  \widetilde{Y_{\leq h}^\nabla(\mu) }$ is a $\prod_{\tau \in c\J_0} L^+G / 
    \cK_N$-torsor over $Y_{\leq h}^\nabla(\mu)$ so this finishes the proof.
\end{proof}

\subsubsection{Some tools for bounding dimension}

Let $\F[\un{X}]$ be a polynomial ring with a $\Gm$-action scaling the variables.  For $f \in \F[\un{X}]$, the leading term $\mathrm{lead}(f)$ of $f$  is the sum of terms with highest weight for the $\Gm$-action.   

\begin{lemma} \label{lemma:leading}
Let $I = (f_1, f_2, \ldots, f_r) \subset \F[\un{X}]$.   If $J = (\mathrm{lead}(f_1), \mathrm{lead}(f_2), \ldots, \mathrm{lead}(f_r))$,  then 
\[
\dim \F[\un{X}]/J \geq \dim \F[\un{X}]/I.
\]
\end{lemma}
\begin{proof}
Let $k_i$ be the weight of the leading term of $f_i$. 
Consider the family defined by the ideal
\[
J_t = ( t^{k_i} f_i(t^{-1} \un{X})) \subset \F[\un{X}, t] 
\]
This defines a family $X_t = V(J_t) \rightarrow \mathbb{A}^1_{\F}$.   Note that the $\Gm$-action identifies the fibers over $t\neq 0$.  Thus, $\dim X_1 = \dim \F[\un{X}]/I$ is the same as the generic fiber dimension.   Since $X_0 = \Spec  \F[\un{X}]/J$, we have $\dim X_0 \geq \dim X_1$.  
\end{proof}

Let $\mathbb{P}_n$ denote the affine space of polynomials of degree $\leq n$ over $\F$.  

\begin{lemma} \label{lem:dimPoly} Let $r,s$ be positive integers. Let $t \leq r +s +1$.  Consider the space $M$ of polynomials $A, C \in \mathbb{P}_r$ and $B, D \in \mathbb{P}_s$ in variable $Z$ such that $AB = CD \mod Z^t$ in $\mathbb{P}_{r+s}$.   Then the codimension of $M$  in $\mathbb{P}_r^2 \times \mathbb{P}_s^2$ is greater than or equal to $\min\{2r +2, 2s +2, t\}$ 
\end{lemma}

\begin{proof}
First, consider the case where $r + s + 1= t$, where we have equality $AB = CD$.  We can stratify the space by the degree of the polynomials so let $M_{k, \ell, m, n} \subset M$ be the locally closed subspace where $\deg A = k, \deg B = \ell, \deg C = m$ and $\deg D = n$.   The condition requires that $k + \ell = m +n$ for the strata to be non-empty.  Consider the map 
\[
M_{k, \ell, m, n} \ra \bA^{k} \times \bA^{\ell} \times \Gm
\]
which sends $(A, B, C, D)$ to $(A, B, c_m)$ where $c_m$ is highest degree coefficient of $C$. The map is finite because with $AB$ fixed, there are only finitely many choices for $C$ which then determines $D$ uniquely.   Thus, $\dim M_{k, \ell, m, n} \leq (k+1) + (\ell + 1)  + 1 \leq r  + s + 3$ and so 
\[
\codim  M_{k, \ell, m, n} \geq (2r + 2) + (2s + 2) - (r+s +3) = r+ s +1 = t
\]
The complement of the $M_{k, \ell, m, n}$ is the locus where one of the pairs $(A, C), (A, D), (B, C)$ or $(B, D)$ are identically zero. These are easily seen to have codimension $\geq \min\{2 r+2, 2s+2\}$.  

For the general case $t \leq r +s +1$, fix $t$ and let $M^{(r,s)}$ be the corresponding space.   Applying Lemma \ref{lemma:leading} giving the leading coefficient of $B, D$ weight 0 and the remaining variables weight $>1$, we see that 
\[
\dim M^{(r,s)} \leq \dim M^{(r, s-1)} + 2.
\]
Similarly, $  \dim M^{(r,s)} \leq \dim M^{(r-1, s)} + 2$.  Applying this procedure inductively (always to the larger of $r$ and $s$ so that the degrees are positive), we eventually arrive at the case where $r +s +1 = t$.  

\end{proof}

\subsubsection{Enumerating the equations defining $\bN_\mu^\nabla$}

In view of Corollary~\ref{cor-open cover} it suffices to control the dimension of $\bN_\mu^\nabla$ by analysing the fibres of the projection 
$$
\bN_\mu^\nabla \rightarrow \prod_{\tau \in \cJ_0} L^+G, \qquad (h_\tau,B_\tau,Y_\tau)_{\tau\in J_0}\mapsto (h_\tau)_{\tau\in J_0}
$$
Fix \((h_\tau)_{\tau\in \cJ_0}\). The corresponding fibre is a closed subscheme of
\(\prod_{\tau\in \cJ_0}(J_\tau\times\mathfrak n)\), an affine space of dimension
\(\sum_{\tau\in \cJ_0}(\mu_{\tau,\alpha}+\mu_{\tau,\beta}+\mu_{\tau,\gamma}+3e)\), and is cut out by the
equations in \eqref{eq-defining Nlambda} for each \(\tau\in \cJ_0\). To make these equations explicit, fix \(\tau\in \cJ_0\) and \(\delta\in\{\alpha,\beta,\gamma\}\), and take the
\(\delta\)-entry of~(5.7). After multiplying by \(u^{\mu_{\tau,\delta}-e}\), the condition becomes
$$
u^{\mu_{\tau,\delta}-e} Y_{\tau,\delta} \equiv \bigg[  B_\tau \phi\bigg( h_{\tau \circ \varphi}^{-1} d_{\tau \circ \varphi}(u) Y_{\tau \circ \varphi} h_{\tau \circ \varphi}\bigg) B_\tau^{-1} + \partial(B_\tau) B_\tau^{-1} \bigg]_{\delta} \mod u^{\mu_{\tau,\delta}+1}\bF[[u]]
$$
Equivalently, comparing coefficients of \(u^i\) yields a family of equations
\begin{equation}\label{eq-main delta and u^i equation}
Y_{\tau,\delta,i - (\mu_{\tau,\delta} -e)} = \bigg[  B_\tau \phi\bigg( h_{\tau \circ \varphi}^{-1} d_{\tau \circ \varphi}(u) Y_{\tau \circ \varphi} h_{\tau \circ \varphi}\bigg) B_\tau^{-1} \bigg]_{\delta, i} + \begin{cases}
    ix_{\tau,i} & \text{ if $\delta = \alpha$} \\
    iy_{\tau,i} & \text{ if $\delta = \beta$} \\
    iz_{\tau,i} + \sum_{k+\ell =i} \ell x_{\tau,k}y_{\tau,\ell}& \text{ if $\delta = \gamma$}
\end{cases}
\end{equation}
indexed by $\min\lbrace 1,1 + \mu_{\tau,\delta}-e \rbrace \leq i \leq \mu_{\tau,\delta}$ (as usual, coordinates of $u$-adic degree below those allowed in the definition of $\bN_\mu^\nabla$ are
understood to be zero).  In particular, the number of equations in~\eqref{eq-main delta and u^i equation} immediately gives the lower bound
$$
 \sum_{\tau \in \cJ_0}  \sum_{\delta \in \lbrace\alpha,\beta,\gamma \rbrace}\operatorname{min}\lbrace e,\mu_{\tau,\delta}\rbrace 
$$
on the dimension of any non-empty fibre. We expect this bound to be sharp (and prove it in the balanced case in Proposition~\ref{prop:bounded}). When at least one
\(\mu_\tau\) is unbalanced we will instead establish an intermediate upper bound, which suffices for our purposes. We begin with the following observations:
\begin{itemize}
    \item \textbf{Type I substitutions (of index $i$, $\delta$, and $\tau$):} If $i > \frac{p(\mu_{\tau,\delta}-e)}{p-1}$ then \eqref{eq-main delta and u^i equation} expresses $Y_{\tau,\delta,i-(\mu_{\tau,\delta}-e)}$ in terms of $x_{\tau,j},y_{\tau,j},z_{\tau,j}$ with $j \leq i$ and $Y_{\tau \circ \varphi,\delta',j}$ with $j + (\mu_{\tau,\delta}-e) < i$. Indeed, if $Y_{\tau \circ \varphi,\delta',j}$ appears on the right hand side of \eqref{eq-main delta and u^i equation} then $j \leq i/p < i - (\mu_{\tau,\delta} -e)$.
    \item \textbf{Type II substitutions (of index $i$, $\delta$, and $\tau$):} When $1 \leq i \leq \mu_{\tau,\delta}$ and is prime to $p$ then \eqref{eq-main delta and u^i equation} with $\delta= \alpha,\beta$ allows $x_{\tau,i}$ and $y_{\tau,i}$ to be expressed in terms of $Y_{\tau,\delta,i - (\mu_{\tau,\delta} -e)}$, $Y_{\tau \circ \varphi,\delta',j}$ with $j \leq i/p$, and $x_{\tau,j},y_{\tau,j},z_{\tau,j}$ with $j < i$. Similarly, if $\delta = \gamma$ then $z_{\tau,i}$ can be expressed in terms of $x_{\tau,i},y_{\tau,i}$, $Y_{\tau,\delta,i - (\mu_{\tau,\delta} -e)}$, $Y_{\tau \circ \varphi,\delta',j}$ with $j \leq i/p$, and $x_{\tau,j},y_{\tau,j},z_{\tau,j}$ with $j < i$.
    \end{itemize}
    In particular, if for each $\delta$ and $\tau$, we are given subsets
    $$
    S_{\tau,\delta,\operatorname{I}} \subset \lbrace i \in \bZ \mid \tfrac{p(\mu_{\tau,\delta} -e)}{p-1} < i \leq \mu_{\tau,\delta} \rbrace,\qquad S_{\tau,\delta,\operatorname{II}} \subset \lbrace i \in \bZ \mid 1 \leq i \leq \mu_{\tau,\delta}, i \not\equiv 0 \mod p \rbrace 
    $$
    with $S_{\tau,\delta,\operatorname{I}}\cap S_{\tau,\delta,\operatorname{II}} = \emptyset$ then repeatedly applying Type I substitutions of index $i,\delta$, and $\tau$ whenever $i \in S_{\tau,\delta,\operatorname{I}}$ and Type II substitutions of index $i,\delta$, and $\tau$ whenever $i \in S_{\tau,\delta,\operatorname{II}}$ allow the complete elimination of the appropriate variables from the equations in \eqref{eq-main delta and u^i equation}. The following lemma provides the key control we need over this process:

\begin{lemma}\label{lem-sub and grade}
    For any fixed $S_{\tau,\delta,\operatorname{I}}$ and $S_{\tau,\delta,\operatorname{II}}$ the resulting expressions for any substituted $Y_{\tau,\delta, i-(\mu{\tau,\delta}-e)}$ (if $i \in S_{\tau,\delta,\operatorname{I}}$) and any substituted  $x_{\tau,i}$, $y_{\tau,i}$ or $z_{\tau,i}$ (if $i \in S_{\tau,\delta,\operatorname{II}}$) have total $u$-adic degree $\leq i$ in the remaining variables (the $u$-adic grading being that placing $x_{\tau,i},y_{\tau,i},z_{\tau,i}$ and $Y_{\tau,\delta,i}$ in degree $i$).
\end{lemma}
In particular, note that Type I substitutions can increase the $u$-adic grading if $\mu_{\tau,\delta}-e >0$.
\begin{proof}
We argue by induction on $i$ and the length of the root $\delta$. Thus, we assume that if $j \in S_{\tau',\delta',\operatorname{I}}$ then $Y_{\tau',\delta',j - (\mu_{\tau',\delta'}-e)}$ has an expression in the remaining variables of $u$-adic degree $\leq j$ whenever $j <i$ or whenever $j =i$, $\delta = \gamma$, and $\delta' \in \lbrace \alpha,\beta \rbrace$. Similarly, for $j \in S_{\tau',\delta',\operatorname{II}}$.

Now suppose $i \in S_{\tau,\delta,\operatorname{I}}$ and consider the expression for $Y_{\tau,\delta,i-(\mu_{\tau,\delta}-e)}$ given by \eqref{eq-main delta and u^i equation}. If $\delta = \alpha$ or $\beta$ then the rightmost term of \eqref{eq-main delta and u^i equation} has an expression of degree $\leq i$ in terms of the remaining variables because $i \not\in S_{\tau,\delta,\operatorname{II}}$. The same is true if $\delta =\gamma$ by the inductive hypothesis. It therefore suffices to prove the same is true of any monomial appearing in the square bracketed term of \eqref{eq-main delta and u^i equation}. Such a monomial is a product of $Y_{\tau \circ \varphi,\delta',j}$ with $1\leq j \leq i/p$ and a second monomial in the $x_{\tau,k},y_{\tau,k},z_{\tau,k}$ of total $u$-adic degree $i - jp< i$. In particular, each $k <i$ and so the inductive hypothesis ensures this second monomial has an expression of total degree $\leq i-jp$ in the remaining variables. To handle the term $Y_{\tau \circ \varphi,\delta',j}$ note that if $j +\mu_{\tau \circ \varphi,\delta'} -e \geq i$ then $j +\mu_{\tau \circ \varphi,\delta'} -e \geq pj$. Thus, $j \leq \frac{\mu_{\tau\circ \varphi,\delta'} -e}{p-1}$ and so $j + (\mu_{\tau\circ \varphi,\delta'} -e) \not\in S_{\tau \circ \varphi,\delta',\operatorname{I}}$. We conclude, in this case, that the entire monomial has an expression of degree $\leq i$ in terms of the remaining variables as required. If $j + (\mu_{\tau\circ \varphi,\delta'} -e) \in S_{\tau \circ \varphi,\delta',\operatorname{I}}$ then we just showed $j +\mu_{\tau \circ \varphi,\delta'} -e < i$. The inductive hypothesis therefore gives an expression of degree $\leq j + \mu_{\tau\circ \varphi,\delta'} -e$ for $Y_{\tau \circ \varphi,\delta',j}$ in the remaining variables. Consequently, the whole monomial has such an expression of total degree 
$$
\leq i-pj + j + \mu_{\tau\circ \varphi,\delta'} -e < i
$$
where the right inequality uses that $j + (\mu_{\tau\circ \varphi,\delta'} -e) \in S_{\tau \circ \varphi,\delta',\operatorname{I}}$ which implies $j > \frac{\mu_{\tau\circ \varphi,\delta'}-e}{p-1}$. Completely identical calculations handle the case $i \in S_{\tau,\delta,\operatorname{II}}$.
\end{proof}

\subsubsection{Dimension bounds via degeneration}

Recall that $\mu_\tau$ is balanced if $\mu_{\tau, \alpha}, \mu_{\tau, \beta} \leq e$, and $\mu_{\tau,\gamma} \geq e$ for all $\tau \in \cJ_0$.

\begin{prop} \label{prop:bounded}  The fibres of the natural projection $\mathbb{N}^{\nabla}_{\mu} \rightarrow \prod_{\tau \in \cJ_0} L^+G$ have dimension $\leq \sum_{\tau \in \cJ_0} (\mu_{\tau, \alpha} + \mu_{\tau, \beta} + e) = \sum_{\tau \in \cJ_0} ( 3e -(n_\tau+m_\tau))$ with equality occurring if and only if $\mu_\tau$ is balanced for each $\tau \in \cJ_0$.
\end{prop} 
\begin{proof}
    We apply the Type I and II substitutions from the previous section as follows: If $\mu_{\tau,\delta} \leq e$ we take 
        $$
    S_{\tau,\delta,\operatorname{I}} = \lbrace 1 + \mu_{\tau,\delta}-e,\ldots, \mu_{\tau,\delta} \rbrace, \qquad S_{\tau,\delta,\operatorname{II}} =  \emptyset
    $$
    In other words, we eliminate all of the $Y_{\tau,\delta,i}$ coordinates.
    If $\mu_{\tau,\gamma} > e$ we take 
       $$
    S_{\tau,\gamma,\operatorname{I}} = \lbrace pi \in \bZ \mid \tfrac{p(\mu_{\tau,\gamma} -e)}{p-1} < pi \leq \mu_{\tau,\gamma} \rbrace,\qquad S_{\tau,\gamma,\operatorname{II}} = \lbrace i \in \bZ \mid 1 \leq i \leq \mu_{\tau,\gamma}, i \not\equiv 0 \mod p \rbrace 
    $$
    Thus, we eliminate each of the $z_{\tau,i}$ with $i$ prime to $p$, as well as some of the $Y_{\tau,\gamma,i}$. If $\mu_{\tau,\alpha} > e$ we take
    $$
    S_{\tau,\alpha,\operatorname{I}} = \lbrace pi \in \bZ \mid \tfrac{p(\mu_{\tau,\alpha} -e)}{p-1} < pi \leq \mu_{\tau,\alpha} \rbrace,\qquad S_{\tau,\alpha,\operatorname{II}} = \lbrace i \in \bZ \mid 1 \leq i \leq \mu_{\tau,\alpha}-e, i \not\equiv 0 \mod p \rbrace 
    $$
    and similarly if $\alpha$ and $\beta$ are interchanged. These substitutions are similar to those made when $\mu_{\tau,\gamma} >e$, but we are careful not to eliminate all the $x_{\tau,i}$. As we will see below, maintaining some of the $x_{\tau,i}$ as free variables will help make use of the unused $\delta = \gamma$ equations for $i = np$ with $n \leq \frac{\mu_{\tau,\gamma}-e}{p-1}$.
    
    Making these substitutions produces an affine space, and we write $D_\tau$ for the number of variables over $\tau$ (we will enumerate this number precisely momentarily). To prove the proposition we are going to show that the equations on this affine space arising from \eqref{eq-main delta and u^i equation} when $\delta = \gamma$ and $i = np$ for $1 \leq n \leq t_\tau := \lfloor \frac{\mu_{\tau,\gamma}-e}{p-1} \rfloor$ (which are unused in all the above substitutions) produce a high enough codimension. 

    To achieve this we simplify these equations using a degeneration. Specifically, we equip our affine space with a ``new-grading'' which places the unsubstituted $x_{\tau,i},y_{\tau,i}$ and $z_{\tau,i}$ in degree $i$ and the unsubstituted $Y_{\tau,\gamma,i}$ in degree $0$. Then:

    \begin{claim}
        Assume $\mu_{\tau,\alpha} \geq \mu_{\tau,\beta}$. With the above new grading the leading terms, in the sense of Lemma~\ref{lemma:leading}, of the equations from \eqref{eq-main delta and u^i equation} with $\delta = \gamma$ and $i =np$ for $1 \leq n \leq t_\tau$ are given by
        \begin{equation} \label{eq:quadric}
\sum_{\substack{k + \ell = np \\  k > \mu_{\tau,\alpha}-e
}}  \ell x_{\tau, k} y_{\tau, \ell} = 0
\end{equation}
for each $\tau \in \cJ_0$. If $\mu_{\tau,\alpha} < \mu_{\tau,\beta}$ then the same holds but with the sum running over $\ell > \mu_{\tau,\beta} - e$.
    \end{claim}
    \begin{proof}[Proof of Claim]
        Notice that if $\mu_{\tau,\gamma} \leq e$ then $t_\tau = 0$ and the claim is vacuous. Notice also that if $\mu_{\tau,\gamma} > e$ then the claim makes sense because the $y_{\tau,\ell}$ and $x_{\tau,k}$ for $k > \mu_{\tau,\alpha} -e$ are unsubstituted variables---indeed $k \not\in S_{\tau,\alpha,\operatorname{II}}$ while the assumption $\mu_{\tau,\alpha} \geq \mu_{\tau,\beta}$ means $\mu_{\tau,\beta} \leq e$ and so $k \not\in S_{\tau,\beta,\operatorname{II}}$.

        To prove the claim first note that the left hand term of \eqref{eq-main delta and u^i equation} is unsubstituted when $i = np$ with $1 \leq n \leq t_\tau$. It therefore has new degree $0$ and does not contribute towards the leading term. Also, the proof of Lemma~\ref{lem-sub and grade} shows that if $\mu_{\tau,\alpha} - e >0$ then any Type II substitution for $x_{\tau,i}$ writes this variable as a degree $<i$ expression in the unsubstituted terms (rather than $\leq i$). In particular, we see that if $k \leq \mu_{\tau,\alpha} -e$ then the $\ell x_{\tau,k} y_{\tau,\ell}$ appearing in the right hand term of \eqref{eq-main delta and u^i equation} with $\delta =\gamma$  has new degree $<i=np$ and does not contribute to the leading term.
        
        Consequently, we just have to show that expanding the square bracket term in \eqref{eq-main delta and u^i equation} into monomials and performing the appropriate Type I and II substitutions gives an expression in terms of the unsubstituted variables with degree $< i = np$ for the new grading. This relies on the estimates from Lemma~\ref{lem-sub and grade}.

        First, any such monomial is a product of $Y_{\tau \circ \varphi,\delta',j}$ with $1\leq j \leq i/p$ and a second monomial in the $x_{\tau,k},y_{\tau,k},z_{\tau,k}$ of total $u$-adic degree $i - jp< i$. Lemma~\ref{lem-sub and grade} ensures this second monomial has an expression of total degree $\leq i-jp$ for the $u$-adic grading in the remaining variables, and hence for the new grading also. If $Y_{\tau \circ \varphi,\delta',j}$ is unsolved for then it has degree $0$ in the new grading and so the whole monomial has degree $<i$ as required. We can therefore assume $j + (\mu_{\tau\circ \varphi,\delta'} -e) \in S_{\tau\circ \varphi,\delta',\operatorname{I}}$ and, by Lemma~\ref{lem-sub and grade}, write $Y_{\tau \circ \varphi,\delta',j}$ as an expression of $u$-adic degree (and hence also of new degree) $\leq j + (\mu_{\tau\circ \varphi,\delta'} -e)$. The whole monomial therefore has new degree
        $$
\leq i-pj + j + \mu_{\tau\circ \varphi,\delta'} -e < i
$$
where, the rightmost inequality follows as it did in the last part of the proof of Lemma~\ref{lem-sub and grade}. This finishes the proof. 
    \end{proof}

    Notice that, in contrast to the equations in \eqref{eq-main delta and u^i equation}, the locus in \eqref{eq:quadric} only consists of variables above a single $\tau \in \cJ_0$. If $\Delta_\tau$ denotes the codimension of the locus above $\tau$ then, in view of Lemma~\ref{lemma:leading}, it suffices to show
    \begin{equation}\label{eq-goal inequality}
    D_\tau - \Delta_\tau \leq \mu_{\tau,\alpha}+\mu_{\tau,\beta} + e
    \end{equation}
    with the inequality strict whenever $\mu_{\tau,\delta} < e$ or $\mu_{\tau,\alpha} >e $ or $\mu_{\tau,\beta} >e$. For this it is convenient to work case-by-case:
    
    \emph{Case 1: The easy case} If $\mu_{\tau,\delta} \leq e$ then only Type I substitutions are made, and these eliminate each of the $Y_{\tau,\delta,i}$'s. Thus, the $\tau$-coordinates of the affine space described are the $x_{\tau,i}$, $y_{\tau,i}$, and $z_{\tau,i}$ and $D_\tau = \mu_{\tau,\alpha}+ \mu_{\tau,\beta} + \mu_{\tau,\gamma}$. Since $\Delta_\tau = 0$ the inequality in \eqref{eq-goal inequality}, and its strict refinement, are immediate.
    
    \emph{Case 2: The strictly balanced case} If $\mu_{\tau,\delta} > e$ and $\mu_{\tau,\alpha},\mu_{\tau,\beta} \leq e$ then the only equations unused by the Type I and II substitutions are those in \eqref{eq-main delta and u^i equation} with $\delta = \gamma$ and $i = np$ with $1 \leq n \leq t_\tau$. Therefore, $D_\tau = \mu_{\tau,\alpha} + \mu_{\tau,\beta} + e + t_\tau$ and we need to show that $\Delta_\tau = t_\tau$. 
    
    Without loss of generality, assume $\mu_{\tau, \alpha} \geq \mu_{\tau, \beta}$.  If $t_{\tau} > 0$, then $\mu_{\tau, \gamma} = \mu_{\tau,\alpha} +\mu_{\tau,\beta} \geq e + p-1$. Using the balanced condition we deduce $e \geq \mu_{\tau,\beta} \geq e+p-1 -\mu_{\tau,\alpha} \geq p-1$. If $\mu_{\tau,\beta} = p-1$ then, since $\mu_{\tau,\alpha} \leq e$,  $t_{\tau} = \lfloor 1 + \frac{\mu_{\tau,\alpha} -e}{p-1}\rfloor \leq 1$, in which case the lemma is clear. We therefore assume $\mu_{\tau,\beta} \geq p$, and so $\mu_{\tau,\alpha} \geq p$ also.

We want to replace the equations in question with equations as in Lemma~\ref{lem:dimPoly}. For this, weight the variables by giving $x_{\tau, k}$ weight 1 if $k \equiv p-1, p-2 \mod p$ and 0 otherwise and $y_{\tau, \ell}$ weight 1 if $\ell \equiv 1, 2 \mod p$ and 0 otherwise. The leading term ideal, in the sense of Lemma \ref{lemma:leading}, is generated by the equations
$$
\sum_{0 \leq k \leq n-1}  \bigg( x_{\tau, p-1 + kp} y_{\tau, 1 + (n-k-1)p} + 2 x_{\tau, p-2 + kp} y_{\tau, 2 + (n-k-1)p}\bigg) = 0
$$
for $1 \leq n \leq t_\tau$. These can be repackaged as a congruence  $C_\tau D_\tau = G_\tau H_\tau$ modulo $Z^{t_\tau}$ of polynomials in a variable $Z$, where 
$$
C_\tau = \sum_{k =0}^{r_\tau} y_{\tau, 1 + kp} Z^k, \quad G_\tau = \sum_{k=0}^{r_\tau} 2 y_{\tau, 2 + kp} Z^k, \quad D_\tau = \sum_{k=0}^{s_\tau } x_{\tau,  kp + p-1} Z^k, \quad H_\tau = \sum_{k=0}^{s_\tau } - x_{\tau, kp + p - 2} Z^k
$$
for $r_\tau,s_\tau$ chosen as large as possible so that the coefficients of these polynomials exist as variables. This means that $p s_\tau + p-1 \leq \mu_{\tau,\alpha}-1$ and $2 + p r_\tau  \leq \mu_{\tau,\beta}-1$. Hence, $s_\tau = \lfloor \frac{\mu_{\tau,\alpha}}{p}\rfloor -1$ and $r_\tau = \lfloor \frac{ \mu_{\tau, \beta} - 3}{p} \rfloor$, both of which are $\geq 0$ by the estimates in the first paragraph. 

By Lemma~\ref{lemma:leading} one just needs to show this locus has codimension $t_\tau$. Lemma \ref{lem:dimPoly} asserts this locus has codimension $\operatorname{min}\lbrace 2r_\tau +2,2s_\tau +2, t_\tau \rbrace$. Since $p \geq 5$ we have $r_\tau \geq \lfloor \frac{\mu_{\tau,\beta}}{p} \rfloor -1$. Also $t_\tau \leq \lfloor \frac{\mu_{\tau,\beta}}{p-1} \rfloor$. We will therefore be done if $2\lfloor \frac{x}{p} \rfloor \geq \lfloor \frac{ x}{p-1} \rfloor$ whenever $x \geq p$. But this is clear---if $x = l+pk$ with $0 \leq l \leq p-1$ then $x = l + (p-1)k +k$ and so
$\lfloor \tfrac{ x}{p-1}\rfloor  \leq 1 +k + \lfloor \tfrac{k}{p-1} \rfloor  \leq 2k$
where the last inequality uses that $k \geq 1$ (since $x \geq p$) and is deduced by separating the cases $k=1$ and $k \geq 2$. 

\emph{Case 3: The main unbalanced case} Finally, suppose $\mu_{\tau,\alpha} >e$. Since $\mu_{\tau,\gamma} = \mu_{\tau,\alpha} + \mu_{\tau,\beta} < 2e$  this forces $\mu_{\tau,\beta} <e$. Note that the case where instead $\mu_{\tau,\beta} >e$ is handled in exactly the same way. 

In this case the Type I and II substitutions leave unused the equations from \eqref{eq-main delta and u^i equation} with $\delta = \gamma$ and $i = np$ with $1 \leq n \leq t_\tau$. Also unused are the equations with $\delta = \alpha$ and $i > \mu_{\tau,\alpha} -e$ and $i$ prime to $p$. It follows that 
$$
        D_\tau = e+ \Big\lfloor \frac{\mu_{\tau,\alpha} -e}{p-1} \Big\rfloor + \Big\lfloor \frac{\mu_{\tau,\alpha} -e}{p} \Big\rfloor + \mu_{\tau,\beta} + e + \ t_\tau 
$$
If $\mu_{\tau,\beta} \leq 2$, then $D_\tau$ is already $< \mu_{\tau,\alpha}+\mu_{\tau,\beta} + e$; indeed $\mu_{\tau,\gamma} \leq \mu_{\tau,\alpha}+2$, and so the claim follows from the inequality 
  \[
x > \Big\lfloor \frac{x}{p} \Big\rfloor  +  \Big\lfloor \frac{x}{p-1} \Big\rfloor +  \Big\lfloor \frac{x + 2}{p-1} \Big\rfloor , \qquad x \geq 1
\]
specialised to $x = \mu_{\tau,\alpha} -e$. This inequality holds whenever $p\geq 5$ (to see this check $x=1$ directly and when $x \geq 2$ use the bound $\lfloor r \rfloor \leq r$ for any rational $r$).  Similarly, if $p >e$ then, since $\mu_{\tau,\gamma} <2e$, we have $ \mu_{\tau,\alpha} -e < \mu_{\tau,\gamma}-e \leq p-2$, and so $D_\tau = 2e + \mu_{\tau,\beta}$ and there is nothing to prove. 

Our goal is therefore to show that if $\mu_{\tau,\beta} > 2$ and $e \geq p$ then $
    \Delta_\tau > t_{\tau} + e+ \lfloor \tfrac{\mu_{\tau,\alpha} -e}{p-1} \rfloor + \lfloor \tfrac{\mu_{\tau,\alpha} -e}{p} \rfloor - \mu_{\tau,\alpha}
    $. Just as in the balanced case we do this by introducing a grading by placing $y_{\tau,\ell}$ in degree $1$ if $\ell \equiv 1,2$ modulo $p$ and in degree $0$ otherwise, and placing $x_{\tau,i}$ in degree $1$ if $i \equiv p-1,p-2$ modulo $p$ and $i > \mu_{\tau,\alpha}-e$ and in degree $0$ otherwise.  The leading term ideal, in the sense of Lemma \ref{lemma:leading}, is then generated by equations
$$
\sum_{d_\tau \leq k \leq n-1 }  \bigg( x_{\tau, p-1 + kp} y_{\tau, 1 + (n-k-1)p} + 2 x_{\tau, p-2 + kp} y_{\tau, 2 + (n-k-1)p}\bigg) = 0
$$
where $d_\tau = \lceil \frac{\mu_{\tau,\alpha}-e-p+3}{p} \rceil$ and $1 \leq n \leq t_\tau$. Note, we require $d_\tau \leq k$ to ensure $p-2+kp >\mu_{\tau,\alpha}-e$. As in the proof of the balanced case, these equations can be repackaged as a congruence  $C_\tau D_\tau = G_\tau H_\tau$ modulo $Z^{t_\tau}$ of polynomials in a variable $Z$, where 
$$
C_\tau = \sum_{k =0}^{r_\tau} y_{\tau, 1 + kp} Z^k, \quad G_\tau = \sum_{k=0}^{r_\tau} 2 y_{\tau, 2 + kp} Z^k, \quad D_\tau = \sum_{k=d_\tau}^{s_\tau } x_{\tau,  kp + p-1} Z^k, \quad H_\tau = \sum_{k=d_\tau}^{s_\tau } - x_{\tau, kp + p - 2} Z^k
$$
for $s_\tau = \lfloor \frac{\mu_{\tau,\alpha}}{p}\rfloor -1$ and $r_\tau = \lfloor \frac{ \mu_{\tau, \beta} - 3}{p} \rfloor$. Note $r_\tau \geq 0$ since $\mu_{\tau,\beta} >2$ while $s_\tau \geq 0$ since $\mu_{\tau,\alpha} >e \geq p$. Combining Lemma~\ref{lemma:leading} and Lemma~\ref{lem:dimPoly} therefore gives $\Delta_\tau \geq \operatorname{min}\lbrace 2r_\tau + 2, 2(s_\tau - d_\tau) +2, t_\tau - d_\tau \rbrace$. The following claim (specialised to $y =\mu_{\tau,\beta}$ and $\mu_{\tau,\alpha} = e + x$) therefore finishes the proof. 
\end{proof}

\begin{claim}
 Let $x, y$ be integers less than $e$ with $y \geq 3$ and $x \geq 1$.  Assume $e \geq p \geq 5$.  If $\Delta = \min\{ 2 \lfloor \frac{y - 3}{p} \rfloor + 2, 2\lfloor \frac{x+e}{p} \rfloor - 2 \lceil \frac{ x+3 - p}{p} \rceil + 2,     \lfloor \frac{x+ y}{p-1} \rfloor -  \lceil \frac{ x+3 - p}{p} \rceil \}$, then 
\[
x - \Big\lfloor \frac{x}{p} \Big\rfloor - \Big\lfloor \frac{x}{p-1} \Big\rfloor - \Big\lfloor \frac{x+ y}{p-1} \Big\rfloor + \Delta > 0
\]
\end{claim}  
\begin{proof}[Proof of Claim]
Write $E := x - \lfloor \frac{x}{p} \rfloor - \lfloor \frac{x}{p-1} \rfloor - \lfloor \frac{x+ y}{p-1} \rfloor$ and 
$$
A :=  2 \Big\lfloor \frac{y - 3}{p} \Big\rfloor + 2, \quad B:=   2\Big\lfloor \frac{x+e}{p} \Big\rfloor - 2 \Big\lceil \frac{ x+3 - p}{p} \Big\rceil + 2, \quad C :=     \Big\lfloor \frac{x+ y}{p-1} \Big\rfloor -  \Big\lceil \frac{ x+3 - p}{p} \Big\rceil 
$$
It suffices to show $E+A,E+B$, and $E+C$ are all $>0$. First, look at 
$$
E + C = x- \Big\lfloor \frac{x}{p} \Big\rfloor -\Big\lfloor \frac{x}{p-1} \Big\rfloor - \Big\lceil \frac{x+3-p}{p} \Big\rceil
$$ 
If $x =1$ this is clearly $>0$. If $x \geq 2$ use the bounds $\lfloor r \rfloor \leq r$ and $\lceil r \rceil \leq  r+1$ then 
$$
E + C  \geq x\left( 1 - \frac{2}{p} - \frac{1}{p-1} \right) - \frac{3}{p} \geq \frac{7x}{20} - \frac{3}{5} >0
$$
where the second inequality uses that $p \geq 5$. Second, since $\lfloor \frac{x+y}{p-1}\rfloor \leq \lfloor\frac{x}{p-1} \rfloor + \lfloor\frac{y}{p-1} \rfloor +1 $ we deduce that
$$
E + A = \underbrace{x - \Big\lfloor \frac{x}{p} \Big\rfloor -2\Big\lfloor \frac{x}{p-1} \Big\rfloor}_{> 0 \text{ for  $x \geq 1$}} -  \Big\lfloor \frac{y}{p-1} \Big\rfloor +\underbrace{2  \Big\lfloor \frac{y-3}{p} \Big\rfloor +1 }_{\geq  \lfloor \frac{y}{p-1} \rfloor \text{for $y \geq 3$}}
$$
where the first inequality follows from the bound $\lfloor r \rfloor \leq r$ and the fact $p \geq 5$ and for the second one notes that $y -3 \leq p(k+1)$ for $k = \lfloor \frac{y-3}{p}\rfloor$ and $ (2k+2)(p-1) - (k+1)p -3 =(k+1)(p-2) -3  \geq 0$ since $p-2 \geq 3$. It follows that $E + A >0$. Finally, since $y <e$ we have
$$
E + B \geq \underbrace{x - \Big\lfloor \frac{x}{p} \Big\rfloor - 2\Big\lfloor \frac{x}{p-1} \Big\rfloor}_{>0 \text{ as above}} -  \left(\underbrace{ \Big\lfloor \frac{x+ e-1}{p-1} \Big\rfloor - \Big\lfloor\frac{x}{p-1} \Big\rfloor}_{\leq \lceil \frac{e-1}{p-1}\rceil} \right) + 2\left( \underbrace{ \Big\lfloor \frac{x+e}{p} \Big\rfloor - \Big\lfloor \frac{x+3 -p}{p} \Big\rfloor +1}_{\geq \lfloor \frac{e+p-2}{p} \rfloor} \right)
$$
If we write $e = bp +r$ with $0 \leq r \leq p-1$ then $\lfloor \frac{e+p-2}{p} \rfloor \geq b$, while $\lceil\frac{e-1}{p-1} \rceil = b + \lceil \frac{b + r -1}{p-1} \rceil \leq 2b$ since $\frac{b+r-1}{p-1} \leq 1 + \frac{b-1}{p-1} \leq b$. It follows that $E+B >0$ and we are done.
\end{proof}

\subsubsection{Top-dimensional components} Here we give more precise control on $\bN_\mu^\nabla$ in the case $\mu_\tau = (2e,e,0)$ for each $\tau \in \cJ_0$.

\begin{prop}\label{prop-top dimensional}
    Suppose $\mu_\tau = (2e,e,0)$ for each $\tau \in \cJ_0$. Then there is a dense open locus in $\prod_{\tau \in \cJ_0} L^+G$ over which the fibres of the projection $\bN_\mu^\nabla \rightarrow \prod_{\tau \in \cJ_0} L^+G$ contain a dense open subscheme which is irreducible and of dimension $\sum_{\tau \in \cJ_0} 3e = \sum_{\tau \in \cJ_0} (\mu_{\tau,\alpha} + \mu_{\tau,\beta} + e)$.
\end{prop}

\begin{proof}
We begin as in Proposition~\ref{prop:bounded}, making the same Type I and II substitutions. These eliminate the $Y_{\tau,\alpha}, Y_{\tau,\beta}$, the $z_{\tau, i}$ for $i$ prime to $p$, and the $Y_{\tau, \gamma, np - e}$ for $n > \frac{e}{p-1}$.  The remaining equations are then of the form
\begin{equation} \label{egamma}
Y_{\tau, \gamma, np - e} = \sum_{k + \ell = np}  \ell x_{\tau, k} y_{\tau, \ell} +  (B_\tau \phi(h_{\tau \circ \varphi}^{-1} d_{\tau \circ \varphi}(u) Y_{\tau \circ \varphi}h_{\tau \circ \varphi}) B_\tau^{-1})_{\gamma, np}. 
\end{equation}
where $n \leq \frac{e}{p-1}$. We prove the proposition by identifying an open locus of $(h_\tau)_{\tau \in \cJ_0} \in \prod_{\tau \in \cJ_0} L^+G$ on which \eqref{egamma} can be used to eliminate the $Y_{\tau,\gamma,n}$ for each $n \leq \frac{e}{p-1}$. 

To proceed, set $X_0 = B_\tau \phi(h_{\tau \circ \varphi}^{-1})$ mod $u$ and look at the right hand term in \eqref{egamma}. We claim that, after making the substitutions as above, the difference 
\begin{equation}\label{eq-difference}
(B_\tau \phi(h_{\tau \circ \varphi}^{-1} d_{\tau \circ \varphi}(u) Y_{\tau \circ \varphi}h_{\tau \circ \varphi}) B_\tau^{-1})_{\gamma, np} - \bigg( X_0 \phi \bigg( d_{\tau \circ \varphi}(0) Y_{\tau \circ \varphi} \bigg) X_0^{-1}\bigg) _{\gamma, np}
\end{equation}
can be expressed entirely in terms of the free variables $x_{\tau,i},y_{\tau,i}$ and $z_{\tau,ip}$ and $Y_{\tau',\gamma,n'}$ with $n' < n$. The Type I substitutions are not relevant here because if $Y_{\tau \circ \varphi,\delta,j}$ appears in \eqref{eq-difference} then  $j < n$ and so $j+e < \frac{e}{p-1}+e = \frac{pe}{p-1}$. For the Type II substitutions, note that if $z_{\tau,i}$ appears in \eqref{eq-difference} then $i \leq (n-1)p$. Thus, the Type II substitutions only introduce $Y_{\tau \circ \varphi,\delta,n'}$'s with $n' \leq i/p < n$ and $Y_{\tau,\gamma,i -e}$. But $i \leq np - p \leq n+e -p$ so $i - e < n$.

Next, we examine right hand term of \eqref{eq-difference} and write  
$$
\begin{aligned}
\bigg( X_0 \phi \bigg( d_{\tau \circ \varphi}(0) Y_{\tau \circ \varphi} \bigg) X_0^{-1}\bigg) _{\gamma, np} = F_\gamma  Y_{\tau \circ \varphi, \gamma, n}   +  G_{\alpha} Y_{\tau \circ \varphi, \alpha, n} + G_{\beta}Y_{\tau \circ \varphi, \beta, n}
\end{aligned}
$$
where $F_\gamma,G_\alpha$, and $G_\beta$ are expressions in the entries of $h_{\tau \circ \varphi}(0)$ and $B_\tau(0)$ which are independent of $n$. Notice that, by the same argument as in the previous paragraph, making the Type I and II substitutions expresses each of $G_{\alpha} Y_{\tau \circ \varphi, \alpha, n}$ and $G_{\beta}Y_{\tau \circ \varphi, \beta, n}$ entirely in terms of the free variables $x_{\tau,i},y_{\tau,i}$ and $z_{\tau,ip}$ and $Y_{\tau',\gamma,n'}$ with $n' < n$. It remains to analyse the term $F_\gamma$. For this we first assume $(h_\tau)_{\tau \in \cJ_0} \in \prod_{\tau\in \cJ_0} L^+G$ lies in the open locus where there exists a Gauss factorisation $h_\tau \mod u = a_\tau t_\tau b_\tau$ with $a_{\tau} \in U$ and $b_{\tau} \in U^{-}$ respectively in the upper and lower triangular unipotent subgroups, and $t_\tau \in T$ inside the diagonal torus. Then 
$$
F_{\gamma,\tau} = d_{\tau \circ \varphi}(0) t_{\tau \circ \varphi,\gamma} + F(a_{\tau \circ \varphi}, b_{\tau \circ \varphi}, t_{\tau \circ \varphi}, x_{\tau, 0}, y_{\tau, 0}, z_{\tau,0})
$$
with $F(a_{\tau \circ \varphi}, b_{\tau \circ \varphi}, t_{\tau \circ \varphi}, x_{\tau, 0}, y_{\tau, 0}, z_{\tau,0})$ of degree $\geq 2$ in the variables $a_{\tau \circ \varphi}, b_{\tau \circ \varphi}$ and $t_{\tau \circ \varphi}$. Since $d_{\tau \circ \varphi}(0) t_{\tau \circ \varphi,\gamma} \in \bF^\times$ we can shrink the open locus of $(h_\tau)_{\tau\in \cJ_0} \in \prod_{\tau \in \cJ_0} L^+G$ so that $F_{\gamma, \tau}= F_{0,\tau} + F_1(x_{\tau, 0}, y_{\tau, 0}, z_{\tau,0})$ with $F_{0,\tau} \neq 0$ and $\prod_{\tau \in \cJ_0} F_{0,\tau} \neq 1$. 

Finally, we return to the task described in the first paragraph of the proof and consider the fibres of $\bN_{\mu}^\nabla \rightarrow \prod_{\tau \in \cJ_0} L^+G$ over the open locus just defined. Note that $x_{\tau, 0}, y_{\tau, 0}, z_{\tau,0}$ are free coordinates on this fibre and so there is a dense open locus where $F_{\gamma,\tau} \neq 0$ and $\prod_{\tau \in \cJ_0} F_{\gamma,\tau} \neq 1$. The equations \eqref{egamma} can therefore be rewritten as 
$$
Y_{\tau, \gamma, np - e} = F_{\gamma,\tau }Y_{\tau \circ \varphi,\gamma, n} + G(x_{\tau,i},y_{\tau,i},z_{\tau,ip}, Y_{\tau',\gamma,j}; j <n)
$$
For $n< \frac{e}{p-1}$ we have $np -e < n$ so these equations allow the elimination of $Y_{\tau \circ \varphi,\gamma,n}$ in terms of the $x_{\tau,i},y_{\tau,i},z_{\tau,ip}$, and the unsolved for $Y_{\tau,\gamma,j}$ with $j > \frac{e}{p-1}$. If $n = \frac{e}{p-1}$ (so $e$ is divisible by $p-1$) then $np-e =n$. Iterating these equations then gives
$$
Y_{\tau,\gamma,n} = Y_{\tau,\gamma,n} \prod_{\tau' \in \cJ_0} F_{\gamma,\tau'} + G_\tau(x_{\tau,i},y_{\tau,i},z_{\tau,ip}, Y_{\tau',\gamma,j}; j <n)
$$
Since $\prod_{\tau' \in \cJ_0} F_{\gamma,\tau'} \neq 1$ we can likewise solve for $Y_{\tau,\gamma,\frac{e}{p-1}}$. In conclusion, there is an open locus of $\prod_{\tau \in \cJ_0} L^+G$ over which the fibres of $\bN^\nabla_\mu \rightarrow  \prod_{\tau \in \cJ_0} L^+G$ admits a dense open locus which is smooth of dimension $\sum_{\tau \in \cJ_0} 3e$.
\end{proof}
\subsubsection{Finishing the proof of Theorem~\ref{thm-bounding dimension without convolution}}

For $N>>0$, Proposition~\ref{prop:bounded} gives
$$
    \operatorname{dim} \left(\prod_{\tau \in \cJ_0} \cK_N \right) \backslash \bN_\mu^\nabla \leq \sum_{\tau \in \cJ_0} \left( 3e - n_\tau - m_\tau + \operatorname{dim} \cK_N \backslash L^+G \right)
    $$
with the inequality strict when $\mu$ is unbalanced. Corollary~\ref{cor-open cover} then implies $\left(\prod_{\tau \in \cJ_0} \cK_N \right) \backslash  \widetilde{Y^\nabla(\mu)}$ has a cover by open subschemes whose dimension has the same upper bound. Each open subscheme is furthermore an $\prod_{\tau\in \cJ_0} \cK_N \backslash L^+G$-torsor over an open substack of $Y^\nabla(\mu)$. It follows that $Y_{\leq h}^\nabla(\mu)$ has an open cover by substacks of dimension $\leq \sum_{\tau \in \cJ_0} \left( 3e - n_\tau - m_\tau \right)$ with the inequality strict when $\mu$ is unbalanced.

It remains to consider the case where $\mu_\tau = (2e,e,0)$ for each $\tau \in \cJ_0$. Proposition~\ref{prop-top dimensional} shows that each open in the above cover of $Y_{\leq h}^\nabla(\mu)$ has a unique irreducible component of dimension $\sum_{\tau \in \cJ_0} 3e$, but this does not imply the same holds for $Y_{\leq h}^\nabla(\mu)$. Instead, we note that Lemma~\ref{lem-open covers} produces a surjective morphism
$$
\prod_{\tau \in \cJ_0} G \times  \left(\prod_{\tau \in \cJ_0} \cK_N \right) \backslash \bN_\mu^\nabla \rightarrow \left(\prod_{\tau \in \cJ_0} \cK_N \right) \backslash  \widetilde{Y^\nabla(\mu)}
$$
via $(g_\tau,h_\tau,B_\tau,Y_\tau) \rightarrow (h_\tau u^{\mu_\tau} B_\tau g_\tau, -h_\tau^{-1} Y_\tau h_\tau)$ whose restriction to $\lbrace g \rbrace \times \left(\prod_{\tau \in \cJ_0} \cK_N \right) \backslash \bN_\mu^\nabla$ is an open immersion for any $g \in \prod_{\tau\in \cJ_0} G(\bF)$. Proposition~\ref{prop-top dimensional} asserts the source of this surjection has a unique top dimensional irreducible component. If the target has two top-dimensional irreducible components  then one can choose top-dimensional open subsets $U_1,U_2$ of the target which are disjoint. Continuity implies the preimage of $U_1$ contains a non-empty open subset $\prod_{\tau \in \cJ_0} U_\tau  \times  U^\nabla$ with $U_\tau \subset G$ open and $U^\nabla \subset \left(\prod_{\tau \in \cJ_0} \cK_N \right) \backslash \bN_\mu^\nabla$ open and top-dimensional. Hence, the preimage of $U_1$ is dense in the unique top-dimensional component of the source. The same is true for the preimage of $U_2$, so these preimages intersect. Since the morphism is surjective, this contradicts the assumption that $U_1 \cap U_2 = \emptyset$. 

\subsection{Bounding the fibres of convolution}\label{sec-conv bounds}

\subsubsection{The main argument}\label{sec-conv main arugment}

Now we prove Theorem~\ref{thm-dim bounds for fibres of convolution}. Let $h =2$.  Thus, consider an $\bF$-valued point $(\fM,N_0) \in Y^\nabla(\mu)$ and fix $\tau \in \cJ_0$ with $\mu_\tau$ balanced and $\neq (2e,e,0)$. Recall $m^{-1}_{\leq \eta_\bullet}(\fM)$ then classifies certain sequences
$$
\fM_\bullet: \fM_e \subset \fM_{e-1} \subset \ldots \subset\fM_0
$$
where $\fM_e := \varphi_{\fM}^{-1}(u^{2e}\fM)$ and $\fM_0 := \varphi^*\fM$. Since $(\fM,N_0) \in Y^\nabla(\mu)$ we can choose $\fS_{\bF}$-bases $\gamma_e,\gamma_0$ of $\fM_e$ and $\fM_0$ respectively, so that 
$$
\gamma_{e,\tau} = \gamma_{0,\tau} u^{\mu_\tau^{*}}
$$
where, as in Proposition~\ref{prop-in schubert variety}, $\mu_\tau^* = (\mu_{1,\tau}^* \geq \ldots \geq \mu_{d,\tau}^*)$ with $\mu_{\ell,\tau}^* = h - \mu_{d-\ell,\tau}^*$. Note, we do not require that $\gamma_0$ and $\gamma_e$ are related via the Frobenius on $\fM$ in any specific way. Note also that, since $\mu_\tau \leq (2e,e,0)$ and $\mu^*_\tau = (2e,2e,2e) - w_0(\mu_\tau)$ for $w_0 \in W$ the longest element, we have $\mu^*_\tau \leq (2e,e,0)$. 

For each irreducible component $C \subset m^{-1}_{\leq \eta_\bullet}(\fM)$ there are additionally dominant $\mu_\tau^{(i)} \leq (2i,i,0)$ for $1\leq i \leq e$ such that for generic $\fM_\bullet \in C$ one has $\fM_{i,\tau}$ generated by $\gamma_{0,\tau} g_i u^{\mu_\tau^{(i)}}$ for some $g_i \in L^+G$. In particular $\mu^{(e)} = \mu^*$. In Section~\ref{sec-MV cycles} we will give more control on the possible $g_i$ which can appear. Specifically, we prove:

\begin{prop}\label{prop-MV}
    Fix an irreducible component $C \subset m^{-1}_{\leq \eta_\bullet}(\fM)$ and choose $k \geq 1$. Then there exists $\fM_\bullet \in C$ with $\fM_{i,\tau}$ generated by the inductively defined
    $$
    \gamma_{i,\tau} = \begin{cases}
        \gamma_{i+1,\tau} u^{\mu_\tau^{(i)} - \mu_\tau^{(i+1)}} & \text{if $k+1\leq i \leq e$,} \\
        \gamma_{i+1,\tau} g u^{\mu_{\tau}^{(i)} - \mu_\tau^{(i+1)}} & \text{if $i = k$} \\
    \end{cases}
    $$
    where $g$ can be any element in
    \begin{itemize}
        \item the upper triangular unipotent $U\subset G$ if $\mu_{\tau}^{(k+1)} - \mu_\tau^{(k)} \neq (1,1,1)$.
        \item $U_{\alpha,-1} = \bigg\{ \left( \begin{smallmatrix}
            1 & \frac{x}{u} & \frac{y}{u} \\ 0 & 1 & 0 \\ 0 & 0 & 1 
        \end{smallmatrix} \right)\mid x,y\in \bF \bigg\}$ or $U_{\beta,-1} = \bigg\{ \left( \begin{smallmatrix}
            1 & 0 & \frac{y}{u} \\ 0 & 1 & \frac{x}{u} \\ 0 & 0 & 1 
        \end{smallmatrix} \right)\mid x,y\in \bF \bigg\}$ (the choice depending upon the irreducible component $C$) if $\mu_{\tau}^{(k+1)} - \mu_\tau^{(k)} = (1,1,1)$.
    \end{itemize}
\end{prop}

Under certain assumptions on $N_0$ (it should be sufficiently indivisible by $u$) we will be able to choose $g$ as in the Proposition~\ref{prop-MV} so that $(\fM,N_0,\fM_\bullet)$ does not satisfy condition $(\mathbf{A}_{k+1})$ from Definition~\ref{def-imposing mod p conditions}. The following definition makes this precise:

\begin{defn}\label{def-generic}
    Say $N_0$ is  $i$-generic relative to an $\bF[[u]]$-basis $\gamma_{i,\tau}$ of $\fM_{i,\tau}$ if $u^i N_0^{\varphi}(\gamma_{i,\tau}) = \gamma_{i,\tau} Y^{(i)}$ with
    $$
    Y^{(i)} =  \begin{pmatrix}
       u^i h_1 &  Y_{\alpha}^{(i)} & Y^{(i)}_\gamma \\ 
       0 & u^i h_2 & Y_{\beta}^{(i)} \\
       0 & 0 & u^{i} h_3
    \end{pmatrix}
    $$
    for $Y_{\alpha}^{(i)}$ having $u$-adic valuation $i + 1 - \langle \alpha,\mu_\tau^{(i)}\rangle $ and $Y_{\beta}^{(i)}$  having $u$-adic valuation $i + 1 - \langle \beta,\mu_\tau^{(i)}\rangle $.
\end{defn}

The basic calculation is:
\begin{lemma}\label{lem-compute failure of monodromy}
    Suppose  $1 \leq k < e$ is such that $\mu_\tau^{(k)}= (2k,k,0)$ and $\mu_\tau^{(k+1)} \neq (2(k+1),k+1,0)$. If $N_0$ is $e$-generic then $g$ can be chosen as in Proposition~\ref{prop-MV} so that  either $u^{k+1}N_0^{\varphi}(\fM_{k+1,\tau}) \not\subset u \fM_{k+1,\tau}$ or $u^{k}N_0^{\varphi}(\fM_{k,\tau}) \not\subset u \fM_{k,\tau}$. 
\end{lemma}
In particular, if $N_0$ is $e$-generic and there exists $k\geq 1$ with $\mu_\tau^{(k)} = (2k,k,0)$ then $(\fM,N_0,\fM_\bullet)$ does not satisfy one of $(\mathbf{A}_{k+1})$ or $(\mathbf{A}_k)$ from Definition~\ref{def-imposing mod p conditions}, and so $(\fM,N_0,\fM_\bullet) \not\in Y^{\nabla,\operatorname{conv}}_{\leq \eta}$.
\begin{proof}
    For each $i$ set $\nu_i = \mu_{\tau}^{(i+1)} - \mu_\tau^{(i)} \leq (2,1,0)$ and write $\gamma_{i,\tau} = \gamma_{i+1,\tau} h u^{-\nu_i}$ with $h = 1$ unless $i=k$, in which case $h = g$ for $g$ as in Proposition~\ref{prop-MV} to be chosen shortly. If $u^i N_0^{\varphi}(\gamma_{i,\tau}) = \gamma_{i,\tau} Y^{(i)}$ then
    \begin{equation}\label{eq-computation}
     Y^{(i)} = u^{-1}\bigg[ u^{\nu_i}  h^{-1} Y^{(i+1)} h u^{-\nu_i} + c_\tau(u) u^{i+1} \bigg(  u^{\nu_i} h^{-1} \partial(h)u^{-\nu_i} + u^{\nu_i} \partial(u^{-\nu_i}) \bigg) \bigg].
    \end{equation}
    An immediate consequence of \eqref{eq-computation} when $h=1$ is that $i+1$-genericity of $N_0$ implies $i$-genericity. Indeed, $Y_{\alpha}^{(i)} = u^{\langle \alpha,\nu_i\rangle-1} Y_{\alpha}^{(i+1)}$, and likewise with $\alpha$ replaced by $\beta$. Applying this inductively gives $k+1$-genericity of $N_0$.

    The choice of $k$ gives:
    $$
    \mu_\tau^{(k+1)} \in (2k,k,0) + \lbrace (2,0,1),(1,2,0),(1,0,2),(0,2,1),(0,1,2),(1,1,1)\rbrace
    $$
    The $k+1$-genericity of $N_0$ immediately handles the first two possibilities. Indeed, $\langle \alpha, \mu_\tau^{(k+1)} \rangle = k+2$ in the first case and and $\langle \beta, \mu_\tau^{(k+1)} \rangle = k+2$ in the second. Thus, $Y^{(k+1)}$ has an entry with $u$-adic valuation $0$ and so $u^{k+1}N_0^{\varphi}(\fM_{k+1,\tau}) \not\subset u \fM_{k+1,\tau}$.

    For the remaining cases we have to force failure at the $k$-th level by suitably choosing $g$ as in Proposition~\ref{prop-MV}. First, suppose $\mu_\tau^{(k+1)} = (2k+1,2k,2), (2k,k+2,1)$, or $(2k,k+1,2)$. Since $g$ can be any element in $U$ it can be chosen so that the upper right entry of $u^{-1}g^{-1} Y^{(k+1)} g$ 
    has $u$-adic valuation
    $$
    \operatorname{min}\lbrace v_u(Y_\alpha^{(k+1)}), v_u(Y_\beta^{(k+1)}) \rbrace - 1
    $$
    Accordingly, this top right entry has $u$-adic valuation $1,1$, or $2$. Conjugating this matrix by $u^{\nu_k}$ drops this $u$-adic valuation by $1,1$ or $2$. Plugging this into \eqref{eq-computation} with $i =k$ shows $Y^{(k)}$ has top right entry with $u$-adic valuation $0$. Thus, $u^{k}N_0^{\varphi}(\fM_{k,\tau}) \not\subset u \fM_{k,\tau}$.

    If instead $\mu_\tau^{(k+1)} = (2k+1,k+1,1)$ then $g$ can be chosen in Proposition~\ref{prop-MV} so that the top right entry of $u^{-1}g^{-1} Y^{(k+1)} g$ has $u$-adic valuation 
    $$
    \leq \operatorname{max}\lbrace v_u(Y_\alpha^{(k+1)}), v_u(Y_\beta^{(k+1)}) \rbrace - 2
    $$
    By $k+1$-genericity of $N_0$ this maximum is $0$. Putting this into \eqref{eq-computation} and noting that $\partial(g) = g^{-1}$ again shows $Y^{(k)}$ has top right entry with $u$-adic valuation $0$. Thus, $u^{k}N_0^{\varphi}(\fM_{k,\tau}) \not\subset u \fM_{k,\tau}$.
\end{proof}

Notice that if $\mu_\tau^{(1)} = (1,1,1)$ and $\fM_\bullet \in C$ then, in view of part (2) of Lemma~\ref{lem-props of closed conditions}, $(\fM,N_0,\fM_\bullet)$ does not satisfy $(\mathbf{B}_1)$ from Definition~\ref{def-imposing mod p conditions}. Thus, the following proposition (which we prove in Section~\ref{sec-genericity}) completes the proof of Theorem~\ref{thm-dim bounds for fibres of convolution} in most cases:

\begin{prop}\label{prop-generic locus}
    Suppose that $\mu_{\tau,\alpha} \geq 2$ and $\mu_{\tau,\beta} \geq 2$. Then there exists an open substack $U \subset Y^\nabla(\mu)$ with complement of dimension $< \sum_{\tau' \in \cJ_0} (3e-n_{\tau'} - m_{\tau'})$ such that, for each $\bF$-point $(\fM,N_0) \in U$, there exists $\fS_{\bF}$-bases $\gamma_e$ and $\gamma_0$ of $\fM_e$ and $\fM_0$ such that $\gamma_{e,\tau} = \gamma_{0,\tau} u^{\mu_\tau^{(e)}}$ and such $N_0$ is $e$-generic relative to $\gamma_{e,\tau}$.
\end{prop}

We will show that, under the assumption that $e \equiv 0$ modulo $3$, there are only two balanced $\mu_\tau$ not covered by Proposition~\ref{prop-generic locus}, namely $(\frac{5e}{3},\frac{2e}{3},\frac{2e}{3})$ and $(\frac{4e}{3},\frac{4e}{3},\frac{e}{3})$. For these we use a slightly different notion of genericity and then (a much simpler) variant of Lemma~\ref{lem-compute failure of monodromy}. This is done in Section~\ref{sec-boundary cases}.

\subsubsection{MV cycles and irreducible components}\label{sec-MV cycles}

Here we prove Proposition~\ref{prop-MV}. We begin by recalling some background regarding fibres of convolution in the affine Grassmannian. Recall $U \subset G$ denotes the upper triangular unipotent, and let $LU: A \mapsto U(A((u)))$ denote the loop group of $U$. Let $\mu$ be a dominant cocharacter and $\delta$ any cocharacter. Recall that an \emph{MV-cycle} inside $\operatorname{Gr}_{\leq \mu}$ of type $\delta$ is an irreducible component of the scheme theoretic image of the map
$$
L^+G u^\mu L^+G \cap v^\delta LU \rightarrow \operatorname{Gr}_{\leq \mu}
$$
given by the action map on the base point in $\operatorname{Gr}_{\leq h}$.

Next, fix dominant cocharacters $\lambda,\nu,\mu$ with $\nu \leq \lambda+\mu$ and consider the locus 
    \begin{equation}\label{eq-fibre fibre}
    \lbrace \cE \in \operatorname{Gr}_{\leq \mu} \mid \cE_\nu \subset \cE \text{ defines a point of $\operatorname{Gr}_{\leq \lambda}$ for any trivialisation of $\cE$} \rbrace
    \end{equation}
    where $\cE_\nu \in \operatorname{Gr}_{\leq h}$ is as in Definition~\ref{defn-schubert vars}. It follows from \cite[Theorem 1.3]{Haines06} that this locus is equidimensional of dimension $\langle \rho,\mu + \lambda  - \nu \rangle$ where $\rho$ is half the sum of the positive roots.
Then, \cite[Theorem 8]{And03} says that each irreducible component in \eqref{eq-fibre fibre} is a translate by $-\nu$ of an MV-cycle inside $\operatorname{Gr}_{\leq - w_0(\lambda)}$ of type $\mu - \nu$.  More precisely, each irreducible component is  the scheme theoretic image of an irreducible component under the map
    \begin{equation}\label{eq-MV}
     u^\nu L^+G u^{-\lambda} L^+G \cap   u^{\mu} LU \rightarrow \operatorname{Gr}_{\leq \mu}.
    \end{equation}

    \begin{rmk}
        Not all such MV cycles generally appear as components of \eqref{eq-fibre fibre}. Indeed, the number of MV cycles in \eqref{eq-MV} coincides with the $\nu - \mu$ weight space inside of the highest weight representation $V(\lambda)$, while the number of irreducible components in \eqref{eq-fibre fibre} equals the multiplicity of $V(\nu)$ inside $V(\lambda) \otimes V(\mu)$, see e.g.\ \cite[Theorem 5.3.21]{Zhuinto}. As explained in \cite[Theorem 10]{And03} the latter number is always $\leq$ the former.  However, we will shortly specialise to the quasi-miniscule $\lambda = (2,1,0)$ and for such coweights these numbers coincide, by e.g.\ \cite[Lemme 10.2]{ngopolo}.
    \end{rmk}
    \begin{lemma}\label{lem-MV calc}
        If $\lambda = (2,1,0)$ and $\nu - \mu \leq \lambda$ as above, then each irreducible component 
        $$
         L^+G u^{-\lambda} L^+G \cap u^{\mu-\nu} LU 
        $$
        contains $U  \cdot u^{\mu - \nu}$ where $U$ is:
        \begin{itemize}
        \item the upper triangular unipotent $U\subset G$ if $\nu -\mu \neq (1,1,1)$. 
        \item an open subset of $U_{\alpha,-1} = \bigg\{ \left( \begin{smallmatrix}
            1 & \frac{x}{u} & \frac{y}{u} \\ 0 & 1 & 0 \\ 0 & 0 & 1 
        \end{smallmatrix} \right)\mid x,y\in \bF \bigg\}$ or $U_{\beta,-1} = \bigg\{ \left( \begin{smallmatrix}
            1 & 0 & \frac{y}{u} \\ 0 & 1 & \frac{x}{u} \\ 0 & 0 & 1 
        \end{smallmatrix} \right)\mid x,y\in \bF \bigg\}$ (the choice depending upon the irreducible component $C$) if $\nu -\mu = (1,1,1)$.
    \end{itemize}
    \end{lemma}
\begin{proof}
    This follows from \cite[Lemme 7.4]{ngopolo} and \cite[Corollaire 7.5]{ngopolo}, but it can also be deduced directly as follows. Since there is a factorisation $LU =\cU_0 \cdot L^+U$ for $\cU_0 \subset U[u^{-1}]$ the kernel of $u^{-1} \mapsto 1$, it suffices to compute 
    $$
     L^+G u^{-\lambda} L^+G \cap u^{\mu-\nu} \cU_0
    $$
    where $\nu - \mu$ runs over those cocharacters $\leq (2,1,0)$ which can easily be done on a case-by-case basis. We give two examples:

If $\nu - \mu = (1,1,1)$, then $u^{-1} M \in  \overline{L^+G u^{-(2,1,0)} L^+G}$ if and only if $u M \in \overline{L^+G u^{(2,1,0)} L^+G} $. This implies $u M$ is integral and so
\[
M = \begin{pmatrix}
       1 & xu^{-1} & yu^{-1} \\
        0 & 1 & zu^{-1} \\
        0 & 0 & 1
    \end{pmatrix}.
\]
Furthermore, the $2 x 2$-minors of $uM$ must be divisible by $u$.  This forces $xz =0$.  This gives the two components.  

When $\nu - \mu \in W \cdot (2,1,0)$, then clearly $U_0  \cdot u^{\mu - \nu} \subset L^+G u^{-\mu} L^+G$ so it suffices to show it is irreducible. If $\nu - \mu = (1,2,0)$, then  $u^{(-1,-2, 0)} M \in  L^+G u^{-(2,1,0)} L^+G$ if and only if $u^{(1,0,2)} M \in L^+G u^{(2,1,0)} L^+G$. Since $u^{(1,0,2)} M$ is integral, 
\[
M = \begin{pmatrix}
       1 & xu^{-1} & yu^{-1} \\
        0 & 1 & 0 \\
        0 & 0 & 1
    \end{pmatrix}.
\]
The minor condition is then automatic and so this is irreducible.   
\end{proof}

\begin{proof}[Proof of Proposition~\ref{prop-MV}]
    Consider the closed locus $C(k+1) \subset C$ consisting of $\fM_\bullet$ with $\fM_{i,\tau}$ generated by $\gamma_{0,\tau} u^{\mu_{\tau}^{(i)}}$ for $k+1 \leq i \leq e$.
    Then $\fM_\bullet \mapsto \fM_{k,\tau}$ defines a morphism 
    $$
    C(k+1) \rightarrow \operatorname{Gr}_{\leq \mu_\tau^{(k)}}
    $$
    which factors through the locus \eqref{eq-fibre fibre} with $\nu = \mu_{\tau}^{(k+1)}$, $\mu = \mu_\tau^{(k)}$, and $\lambda = (2,1,0)$. We claim this factorisation surjects onto a union of irreducible components in \eqref{eq-fibre fibre}. This proves the proposition because it produces $\fM_\bullet \in C(k+1)$ with $\fM_{k,\tau}$ generated by $\gamma_{0,\tau} u^{\mu_\tau^{(k+1)}} g u^{\mu_\tau^{(k)} - \mu_\tau^{(k+1)}} = \gamma_{k+1,\tau} g u^{\mu_\tau^{(k)}}$ with $g$ as described in Lemma~\ref{lem-MV calc}.
    
    Since $C(k+1) \rightarrow \operatorname{Gr}_{\leq \mu_\tau^{(k)}}$ is proper, it suffices to show $C(k+1)$ is irreducible and that the image of $C(k+1) \rightarrow \operatorname{Gr}_{\leq \mu_\tau^{(k)}}$ has dimension $\langle \rho, (2,1,0) + \mu_\tau^{(k)} - \mu_\tau^{(k+1)} \rangle$. The basic observation is that on an open neighbourhood of $\cE_{\mu_\tau^{(k+2)}}$ the morphism $C(k+2) \rightarrow \operatorname{Gr}_{\leq \mu_\tau^{(k+2)}}$ is a trivial fibration with fibre $C(k+1)$. Thus, if the claim holds for $C(k+2)$ and $\operatorname{dim}C(k+2)$ equals
    \begin{equation}\label{eq-formula for dimension}
    \sum_{\tau' \neq \tau} (n_{\tau'}+m_{\tau'}) + \langle \rho, (2i,i,0) - \mu_\tau^{(i)} \rangle
    \end{equation}
    when $i=k+2$, then $C(k+1)$ is irreducible and \eqref{eq-formula for dimension} specialised to $i=k+1$ gives $\operatorname{dim}C(k+1)$. Since \eqref{eq-formula for dimension} with $i=k$ is an upper bound for $\operatorname{dim}C(k)$ we conclude that the image of $C(k+1) \rightarrow \operatorname{Gr}_{\leq \mu_\tau^{(k)}}$ has dimension $\langle \rho, (2,1,0) + \mu_\tau^{(k)} - \mu_\tau^{(k+1)} \rangle$ as required.
\end{proof}
\subsubsection{Genericity of $N_0$}\label{sec-genericity}

Here we prove Proposition~\ref{prop-generic locus}. Clearly, the locus in $Y^\nabla(\mu)$ where there exists $\gamma_e,\gamma_0$ as in Proposition~\ref{prop-generic locus} is locally closed. To prove the proposition it therefore suffices to show that the pullback of this locus to 
$$
      \left( \prod_{\tau \in \cJ_0} \cK_N \right) \backslash  \widetilde{Y_{\leq h}^\nabla(\mu) } \cong \widetilde{Y_{\leq h}^\nabla(\mu)}/ \left( \prod_{\tau \in \cJ_0} \cK_N \right)
    $$
    with $N>>0$ (the isomorphism being as described in Corollary~\ref{cor-open cover}) has dimension $\operatorname{dim} Y^\nabla(\mu) + \sum_{\tau \in \cJ_0} \operatorname{dim} L^+G/\cK_N$, with complement of strictly smaller dimension. We do this by producing an open substack of the left hand quotient which lies inside this pullback, and has complement of strictly smaller dimension.

    For this, let $\widetilde{U}$ denote the image of the open immersion $\bN_\mu^\nabla \rightarrow \widetilde{Y^\nabla(\mu)}$
from Lemma~\ref{lem-open covers} in which one takes $w_\tau =1$ for each $\tau \in \cJ_0$. This consists of $(\fM,N_0,\beta^0) \in \widetilde{Y^\nabla(\mu)}$ so that, if $\beta_0 = \beta^0 \otimes 1$ is an $\fS_{\bF}$-basis of $\fM_0$, then
$$
\varphi_{\fM}(\beta_0) = \beta^0 X, \qquad N_0(\beta^0) \equiv \beta^0 \cN \mod u^{e+1}
$$
with $X_\tau = h_\tau u^{\mu_\tau} B_\tau$ and $\cN_\tau = -h_\tau Y_\tau h_\tau^{-1}$
for $(h_\tau,B_\tau,Y_\tau) \in \bN_\mu^\nabla$. If $w_0 \in W$ is the longest element then $\gamma_{0,\tau} := \beta_{0,\tau} B_\tau^{-1} w_0$ is an $\bF[[u]]$-basis of $\fM_{0,\tau}$ and
$$
\gamma_{e,\tau} := \beta_{e,\tau} h_\tau w_0 = \beta_{0,\tau} E(u)^h X^{-1}_\tau h_\tau w_0= \beta_0 B_\tau^{-1} w_0 u^{\mu_\tau^{(e)}} = \gamma_{0,\tau} u^{\mu_\tau^{(e)}}
$$
for the dominant $\mu_\tau^{(e)} = (2e,2e,2e) - w_0\mu_\tau$. Using \eqref{eq-phi o N o phi^{-1}} and \eqref{eq-phi^*N_0 of beta_e} one computes that
$$
u^e N_0^{\varphi}(\gamma_{e,\tau})
\equiv \gamma_{e,\tau}c_\tau(u)\begin{pmatrix}
     eh u^e & -Y_{\tau,\beta} & -Y_{\tau,\gamma} \\ 0 & eh u^e & -Y_{\tau,\alpha} \\ 
     0 & 0 & eh u^e
\end{pmatrix} 
$$
modulo $u^{e+1}\fM_e$. From the equations \eqref{eq-defining Nlambda} defining $\bN_\mu^\nabla$ we see that $Y_{\tau,\alpha}$ has $u$-adic valuation 
$$
\geq e+1 - \mu_{\tau,\alpha} = e+1 - \langle \beta, \mu_\tau^{(e)} \rangle
$$
(since $w_0(\beta) = -\alpha)$, and similarly with $\alpha$ and $\beta$ interchanged. Thus, $N_0$ is $e$-generic with respect to the basis $\gamma_{e,\tau}$ if and only if this is an equality.

\begin{lemma}\label{lem-free coodinates}
    Recall $\mu_\tau$ is balanced.
\begin{enumerate}
    \item If $\mu_{\tau,\alpha} \geq 2$ then the locus in $\bN_\mu^\nabla$ where  $Y_{\tau,\alpha,e+1 - \mu_{\tau,\alpha}} = 0$ has codimension $>0$.
    \item If $\mu_{\tau,\alpha} =0$ then the locus in $\bN_\mu^\nabla$ where  $Y_{\tau,\gamma,e+1 - \mu_{\tau,\gamma}} = 0$ has codimension $>0$.
\end{enumerate}
     Both (1) and (2) also hold with $\alpha$ replaced by $\beta$.
\end{lemma}

\begin{proof}
Recall the analysis in the proof of Proposition \ref{prop:bounded}. If $\mu_{\tau,\alpha} \geq 1$ we have $1 \geq 1 + \mu_{\tau,\alpha}-e$ and  so the Type I substitution at $i=1 \in S_{\tau,\alpha,\operatorname{I}}$ identifies $Y_{\tau, \alpha, e +1 -\mu_{\tau, \alpha}} = x_{\tau, 1}$. Thus, we need to bound the locus where $x_{\tau,1} = 0$.  As in loc.\ cit., we can do this after degenerating to the locus described in the Claim, then further degenerating as in: \emph{The strictly balanced case} of loc.\ cit. Since $\mu_{\tau,\alpha} \geq 2$ we have $x_{\tau,1}$ appearing as a free variable in the last degeneration (since $p\geq 5$) which gives the desired bound.  Note the same argument goes through with $\alpha$ replaced by $\beta$ and $x_{\tau,k}$ and $y_{\tau,l}$ interchanged.

If instead $\mu_{\tau,\alpha} = 0$ then $\mu_{\tau,\gamma} = \mu_{\tau,\beta}$ and so, since $\mu_\tau$ is balanced, $\mu_{\tau,\gamma} = \mu_{\tau,\beta} = e$. Then $i = 1 \in S_{\tau,\gamma,\operatorname{I}}$ and this Type I substitution expresses $Y_{\tau,\gamma,1} = z_{\tau,1} + x_{\tau,1}y_{\tau,0} + x_{\tau,0}y_{\tau,1}$. As explained in: \emph{The easy case} from the proof of Proposition \ref{prop:bounded}, the $z_{\tau,i},x_{\tau,i}$, and $y_{\tau,i}$ are all free variables, so we are done.
\end{proof}

\begin{proof}[Proof of Proposition~\ref{prop-generic locus}]
    Consider the open locus $\widetilde{U}_0 \subset \widetilde{U}$ where $Y_{\tau,\alpha,e+1- \mu_{\tau,\alpha}} \neq 0$ and $Y_{\tau,\beta,e+1- \mu_{\tau,\beta}} \neq 0$. Clearly, the left action of $\prod_{\tau \in \cJ_0} \cK_N$ on $\widetilde{U}$ stabilises  $\widetilde{U}_0$ and so we obtain an open substack of  $\left( \prod_{\tau \in \cJ_0} \cK_N \right) \backslash  \widetilde{U}$. By Lemma~\ref{lem-free coodinates}, it has complement of strictly smaller dimension. By Theorem~\ref{thm-bounding dimension without convolution}, the same is therefore true of its complement in $\left( \prod_{\tau \in \cJ_0} \cK_N \right) \backslash  \widetilde{Y_{\leq h}^\nabla(\mu) }$. Since it lies inside the pullback described at the start of Section~\ref{sec-genericity} we are done.
\end{proof}

\subsubsection{Boundary cases}\label{sec-boundary cases}

To finish the proof of Theorem~\ref{thm-dim bounds for fibres of convolution} we must consider those $\mu_\tau$'s not covered by Proposition~\ref{prop-generic locus}. The following  simple computation is where the assumption $e \equiv 0$ modulo $3$ appear:

\begin{lemma}
    Suppose $r \in \bZ$ is such that $3r = e$. If $\mu_\tau$ is balanced and  $\mu_{\tau,\alpha} \leq 1$ then $\mu_{\tau} = (4r,4r,r)$. If instead $\mu_{\tau,\beta} \leq 1$ then $\mu_\tau = (5r,2r,2r)$.
\end{lemma}

In particular, this rules out the possibility that $\mu_{\tau,\alpha} =1$ or $\mu_{\tau,\beta} =1$, and so part (2) of Lemma~\ref{lem-free coodinates} can be applied. Specifically, we proceed as in Section~\ref{sec-conv main arugment} but work with the following variant of Definition~\ref{def-generic}:

\begin{defn}
    Say $N_0$ is $(i,\alpha)$-generic relative to an $\bF[[u]]$-basis $\gamma_{i,\tau}$ of $\fM_{i,\tau}$ if $u^i N_0^{\varphi}(\gamma_{i,\tau}) = \gamma_{i,\tau} Y^{(i)}$ with
    $$
    Y^{(i)} =  \begin{pmatrix}
       u^i h_1 &  Y_{\alpha}^{(i)} & Y^{(i)}_\gamma \\ 
       0 & u^i h_2 & Y_{\beta}^{(i)} \\
       0 & 0 & u^{i} h_3
    \end{pmatrix}
    $$
    for $Y_{\alpha}^{(i)}$ having $u$-adic valuation $i + 1 - \langle \alpha,\mu_\tau^{(i)}\rangle $ and $Y_{\gamma}^{(i)}$  having $u$-adic valuation $i + 1 - \langle \gamma,\mu_\tau^{(i)}\rangle $. Likewise, define $(i,\beta)$-genericity by interchanging $\alpha$ and $\beta$.
\end{defn}

We then have the following (much simpler analogue) of Lemma~\ref{lem-compute failure of monodromy}:

\begin{lemma}\label{lem-compute failure of monodromy}
    Assume $\mu_{\tau,\alpha} =0$ and apply Proposition~\ref{prop-MV} for any $k \leq e-1$ and $g =1$. If $N_0$ is $(e,\beta)$-generic then $u^{e-1}N_0^{\varphi}(\fM_{e-1,\tau}) \not\subset u \fM_{e-1,\tau}$. Likewise, the same holds if $\alpha$ and $\beta$ are interchanged.
\end{lemma}
\begin{proof}
    Note that $\mu_{\tau,\alpha}=0$ implies $\mu_\tau^{(e)} = (6r,6r,6r) - (r,4r,4r) = (5r,2r,2r)$. The dominance of $\mu_\tau^{(e-1)}$ means the only valid possibilities for $\mu_\tau^{(e-1)}$ are 
    $$
    (5r,2r,2r) - \lbrace (0,1,2), (1,0,2),(1,1,1),(2,0,1)\rbrace
    $$
    Notice that in the first three cases one has $\langle \gamma, \mu_\tau^{(e-1)} \rangle \geq e$, while the fourth case has $\langle \beta, \mu_\tau^{(e-1)} \rangle = e+1$.
    
    The same argument as in the first paragraph of the proof of Lemma~\ref{lem-compute failure of monodromy} shows that $(e,\alpha)$-genericity implies $(e-1,\alpha)$-genericity. Thus, in the first three cases $Y^{(e-1)}_\gamma$ has $u$-adic valuation $\leq e-e= 0$, and in the fourth case $Y^{(e-1)}_\beta$ has $u$-adic valuation $\leq e - (e+1) = -1$. Thus, $u^{e-1}N_0^{\varphi}(\fM_{e-1,\tau}) \not\subset u \fM_{e-1,\tau}$. An identical computation holds if instead $\mu_\tau = (5r,2r,2r)$ and $\mu_\tau^{(e)} = (4r,4r,r)$.
\end{proof}

To finish the proof of Theorem~\ref{thm-dim bounds for fibres of convolution} we only require an analogue of Proposition~\ref{prop-generic locus} describing an open locus on which $N_0$ is $(e,\beta)$-generic when $\mu_{\tau,\alpha} =0$ and $(e,\alpha)$-generic when $\mu_{\tau,\beta} =0$. The discussion of Section~\ref{sec-genericity} constructs such a locus after using part (2) of Lemma~\ref{lem-free coodinates}. 

\section{Kisin varieties}

\subsection{The general construction}

Fix a continuous representation $\overline{\rho}$ of $G_K$ on a finite dimensional $\bF$-vector space, together with a choice of $\bF$-basis $\overline{\alpha}$ of $\overline{\rho}$. We can then consider the framed deformation ring $R_{\overline{\rho}}^{\square}$ 
with universal lifting $\rho^{\operatorname{univ}}: G_K \rightarrow \operatorname{GL}_d(R_{\overline{\rho}}^\square)$ over $\cO$. Following \cite[Corollary 2.27]{Kis08}, write  $ R_{\overline{\rho}}^{\lambda}$  for the unique $\cO$-flat quotient of $R_{\overline{\rho}}^{\square}$ characterised by the property that any $\cO$-algebra homomorphism $R^{\square}_{\overline{\rho}} \rightarrow A$, with $A$  finite flat over $\cO$, factors through $R_{\overline{\rho}}^{\lambda}$ if and only if $\rho^{\operatorname{univ}} \otimes_{R^{\operatorname{univ}}_{\overline{\rho}}} A$ is crystalline with Hodge type $\lambda$. 

The following construction, which first appeared in \cite{mffgs} and \cite{Kis08}, illustrates how to relate $R_{\overline{\rho}}^{\lambda}$ with the $\cY^{\operatorname{cr},\operatorname{conv}}_{\lambda}$ defined in Section~\ref{sec-introducing hodge types}. Here, as we have done previously, we assume $\lambda$ is concentrated in degree $[0,h]$ for some fixed $h \geq 0$.

\begin{const}\label{def-Kisinvar}
    Take $d = \operatorname{dim}_{\bF} \overline{\rho}$ in Section~\ref{sec-setup}. a Hodge type $\lambda$ concentrated in degree $[0,h]$. Then we define $$\widehat{\cL^{\operatorname{cr}}_{\lambda,\overline{\rho}}} := \cY^{\operatorname{cr}}_\lambda \times_{\cX_{d}} \operatorname{Spf}R_{\overline{\rho}}^{\lambda}, \quad \widehat{\cL^{\operatorname{cr},\operatorname{conv}}_{\lambda,\overline{\rho}}} := \cY^{\operatorname{cr},\operatorname{conv}}_\lambda \times_{\cX_{d}} \operatorname{Spf}R_{\overline{\rho}}^{\lambda} $$
    where $\cX_d$ is the formal algebraic stack defined in \cite[Definition 3.2.1]{EGstack}, the morphism $\cY^{\operatorname{cr},\operatorname{conv}}_\lambda \rightarrow \cX_{d}$ is given by $\fM \mapsto \fM \otimes_{A_{\operatorname{inf},A}} 
    W(C^\flat)_A$, and $\operatorname{Spf}R_{\overline{\rho}}^{\lambda} \rightarrow \cX_{d}$ is the morphism from \cite[Proposition 4.8.10]{EGstack}. As explained in \cite[4.5.26]{EGstack}, $\widehat{\cL^{\operatorname{cr}}_{\lambda,\overline{\rho}}}$ can be realised as the $\fm_{R_{\overline{\rho}}^{\lambda}}$-adic completion of a projective $R_{\overline{\rho}}^{\lambda}$-scheme, which we denote $\cL^{\operatorname{cr}}_{\lambda,\overline{\rho}}$.   Since $\cY^{\operatorname{cr}, \operatorname{conv}}_\lambda \rightarrow \cY^{\operatorname{cr}}_\lambda$ is proper, the same argument produces a projective $R_{\overline{\rho}}^{\lambda}$-scheme $\cL^{\operatorname{cr},\operatorname{conv}}_{\lambda,\overline{\rho}}$ whose completion is $\widehat{\cL^{\operatorname{cr},\operatorname{conv}}_{\lambda,\overline{\rho}}}$. 
\end{const}

Suppose $A$ is a finite local $\cO$-algebra with residue field $\bF$. Directly from Construction~\ref{def-Kisinvar} we see that the $A$-points of $\cL_{\lambda,\overline{\rho}}^{\operatorname{cr},\operatorname{conv}}$ functorially identify with pairs $(\fM,\alpha)$ where $\fM \in \cY^{\operatorname{cr,\operatorname{conv}}}_\lambda(A)$ and $\alpha$ is an $A$-basis of $T(\fM) := (\fM \otimes_{\bZ_p} W(C^\flat))^{\varphi=1}$ so that the $\bF$-linear isomorphism
$$
T(\fM) \otimes_A \bF \cong \overline{\rho}
$$
identifying $\alpha \otimes_{A} \bF$ and $\overline{\alpha}$ is $G_K$-equivariant. In particular, this illustrates that the morphism $\cL^{\operatorname{cr},\operatorname{conv}}_{\lambda,\overline{\rho}} \otimes_{\cO} \bF \rightarrow \cY^{\operatorname{cr},\operatorname{conv}}_{\lambda} \otimes_{\cO} \bF$ given by $(\fM,\alpha) \mapsto \fM$ is a formally smooth morphism of algebraic stacks over $\operatorname{Spec}\bF$ (compare with \cite[Lemma 16.6]{B23}). This also shows that $\cL^{\operatorname{cr},\operatorname{conv}}_{\lambda,\overline{\rho}} \rightarrow \operatorname{Spec}R_{\overline{\rho}}^{\lambda}$ becomes an isomorphism after inverting $p$. See, for example, \cite[Corollary 3.3.7]{B23b}.

The following, whose proof is taken from \cite[Corollary 2.4.10]{mffgs}, demonstrates the utility of these constructions.

\begin{prop}\label{prop-connectedness in Kisin var}
    Suppose $x_1,x_2$ are $\cO$-valued points of $\operatorname{Spec}R_{\overline{\rho}}^{\lambda}$ with (necessarily unique) liftings $(\fM_i,\alpha_i) \in \cL_{\lambda,\overline{\rho}}^{\operatorname{cr},\operatorname{conv}}(\cO)$. If $\cY^{\operatorname{cr},\operatorname{conv}}_\lambda \otimes_{\cO} \bF$ is reduced then $x_1,x_2$ lie in the same irreducible component of $\operatorname{Spec}R_{\overline{\rho}}^{\lambda}$ if the images of $(\fM_i,\alpha_i)$ in $\cL_{\lambda,\overline{\rho}}^{\operatorname{cr},\operatorname{conv}} \otimes_{R_{\overline{\rho}}^{\lambda}} \bF$
    lie in the same connected component.
\end{prop}
\begin{proof}
   Since $\cL^{\operatorname{cr},\operatorname{conv}
    }_{\lambda,\overline{\rho}} \otimes_{\cO} \bF \rightarrow \cY^{\operatorname{cr},\operatorname{conv}}_{\lambda} \otimes_{\cO} \bF$ is formally smooth, reducedness of $\cY^{\operatorname{cr},\operatorname{conv}}_\lambda \otimes_{\cO} \bF$  implies the same for $\cL^{\operatorname{cr},\operatorname{conv}}_{\lambda,\overline{\rho}} \otimes_{\cO} \bF$. This implies that any idempotent $e \in R_{\overline{\rho}}^{\lambda}[\frac{1}{p}] = \mathcal{O}(\cL_{\lambda,\overline{\rho}}^{\operatorname{cr},
    \operatorname{conv}}[\frac{1}{p}])$ lies inside $\mathcal{O}(\cL_{\lambda,\overline{\rho}}^{\operatorname{cr},
    \operatorname{conv}})$. If not we could choose a uniformiser $\varpi \in \cO$ and a minimal $n \geq 1$ so that $\varpi^n e \in \mathcal{O}(\cL_{\lambda,\overline{\rho}}^{\operatorname{cr},
    \operatorname{conv}})$; but then the image of $\varpi^n e$ in $\mathcal{O}(\cL_{\lambda,\overline{\rho}}^{\operatorname{cr},
    \operatorname{conv}} \otimes_{\cO}\bF)$ squares to zero, contradicting reducedness. It follows that if  the images of $(\fM_i,\alpha_i)$ in $\cL^{\operatorname{cr},\operatorname{conv}}_{\lambda,\overline{\rho}} \otimes_{\cO} \bF$ lie in the same connected component then the same is true of $x_i[\frac{1}{p}]$ in $\operatorname{Spec}R_{\overline{\rho}}^{\lambda}[\frac{1}{p}]$. Since $R_{\overline{\rho}}^{\lambda}$ is $\cO$-flat with $R_{\overline{\rho}}^{\lambda}[\frac{1}{p}]$ regular \cite[Theorem 3.3.8]{Kis08}, this is equivalent to asking that the $x_i$ lie in the same irreducible component of $\operatorname{Spec}R_{\overline{\rho}}^{\lambda}$. We conclude using the standard fact that the connected components of any proper scheme over $\operatorname{Spec}R_{\overline{\rho}}^{\lambda}$ are in bijection with the connected components in the fibre over the closed point, as follows by combining Stein factorisation \cite[03H2]{stacks-project} with the idempotent lifting described in \cite[09X1]{stacks-project}.    
\end{proof}

\subsection{Torus actions and semisimple points}

Proposition~\ref{prop-connectedness in Kisin var} motivates the problem of determining when two points in $\cL^{\operatorname{cr},\operatorname{conv}}_{\lambda,\overline{\rho}} \otimes_{R_{\overline{\rho}}^{\lambda}} \bF$ lie in the same connected component. The following observation substantially simplifies this process whenever $\overline{\rho} = \bigoplus_{\ell=1}^d \overline{\rho}_\ell$ with each $\overline{\rho}_\ell$ one-dimensional. In this case we assume the basis $\overline{\alpha}$ of $\overline{\rho}$ is compatible with this decomposition.

\begin{lemma}\label{lem-flow}
    Suppose that $\overline{\rho}$ is completely reducible as above. Then $\cL^{\operatorname{cr},\operatorname{conv}}_{\lambda,\overline{\rho}} \otimes_{R_{\overline{\rho}}^{\lambda}} \bF$ admits a $T$-action which, on $\bF$-valued points, 
    is given by 
    $$
    t \cdot (\fM ,\alpha) = (\fM, \alpha t)
    $$
    (recall from the previous section that $\alpha$ here is  an $\bF$-basis of $T(\fM)$, viewed as a row vector). 
\end{lemma}
\begin{proof}
    One just has to check that the given formula is well-defined, i.e. that if the map $T(\fM) \rightarrow \overline{\rho} $ identifying $\alpha$ and $\overline{\alpha}$ is $G_K$-equivariant, then so is the map identifying $\alpha t$ and $\overline{\alpha}$. But this is just the assertion that $t$ normalises $\overline{\rho}$.
\end{proof}

The properness of $\cL^{\operatorname{cr},\operatorname{conv}}_{\lambda,\overline{\rho}} \otimes_{R_{\overline{\rho}}^{\lambda}} \bF$ over $\bF$ therefore ensures every connected component of $\cL^{\operatorname{cr},\operatorname{conv}}_{\lambda,\overline{\rho}} \otimes_{R_{\overline{\rho}}^{\lambda}} \bF$ contains a $T$-fixed $\bF$-valued point. Such $T$-fixed points are very easy to make explicit:

\begin{lemma}\label{lem-sum of rank ones}
    Suppose $(\fM,\alpha)$ is $T$-fixed under the action in Lemma~\ref{lem-flow}. Then, for $1 \leq \ell \leq d$, there are rank one Breuil--Kisin modules $\fM^{(\ell)}$ equipped with convolution structures and crystalline $G_K$-actions so that
    $$
    \fM \cong \bigoplus_{\ell=1}^d \fM^{(\ell)} 
    $$
    compatibly with convolution structures and crystalline $G_K$-actions.
\end{lemma}
\begin{proof}
    A standard fact of the functor $\fM \mapsto T(\fM)$ is the existence of a unique (up to isomorphism) $\varphi,G_{K_\infty}$-equivariant identification
    \begin{equation}\label{eq-another identication}
    \fM\otimes_{\fS} C^\flat \cong T(\fM) \otimes_{\bZ_p} C^\flat
    \end{equation}
    This holds for any Breuil--Kisin module over $\fS_{\bF}$, see for example, \cite[Lemma 4.26]{BMS}. If $\fM$ is furthermore equipped with a $G_K$-action then this identification becomes $G_K$-equivariant for the induced $G_K$-action on $T(\fM)$. 
    
    Then $(\fM,\alpha)$ being $T$-fixed means that, for each $t \in T$, the automorphism of $T(\fM) \otimes_{\bZ_p} C^\flat$ given by $\alpha \mapsto \alpha t$ stabilises $\fM$ under \eqref{eq-another identication}. If $T(\fM) = \bigoplus T_\ell$ is a weight decomposition for the action of a sufficiently generic character $\bG_m \rightarrow T$, then there exists a $\varphi$-equivariant decomposition
    $$
    \fM = \bigoplus_{\ell=1}^d \fM^{(\ell)}, \qquad \fM^{(\ell)} :=  \fM \cap \Big( T_\ell \otimes_{\bF_p} C^\flat \Big)
    $$
    which becomes $G_K$-equivariant after extending scalars to $C^\flat$. After taking the Frobenius twist of \eqref{eq-another identication} we also have that the convolution structure $\fM_\bullet$ on $\fM$ is stabilises by the $T$-action, and so
    $$
    \fM_i = \bigoplus_{\ell}^d \fM_i^{(\ell)} , \qquad 
    \fM_i^{(\ell)} := \fM_i \cap \Big(T_\ell \otimes_{\bZ_p} C^\flat \Big)
    $$
    under the Frobenius twist of \eqref{eq-another identication}. Clearly, each $\fM^{(\ell)}$ and $\fM^{(\ell)}_i$ are $\fS_{\bF}$-projective  and so define rank one Breuil--Kisin modules with convolution structure. Our assumption that $\overline{\rho}$ decomposes as a sum of $1$-dimensional representations compatibly with the basis $\overline{\alpha}$ ensures that each line $T_\ell \subset T(\fM)$ is $G_K$-stable. The extension of this $G_K$-action to  $T_\ell \otimes_{\fS} C^\flat$ therefore induces a  $G_K$-action on $\fM^{(\ell)} \otimes_{\fS} C^\flat$ which is crystalline (since that the sum of these $G_K$-actions is crystalline on $\fM$).
\end{proof}

\begin{cor}\label{cor=N0=0}
    Let $(\fM,\alpha) \in \cL^{\operatorname{cr},\operatorname{conv}}_{\lambda,\overline{\rho}} \otimes_{R_{\overline{\rho}}^{\lambda}} \bF$ and suppose additionally that $h \leq p-1$ so that the morphism $\fM \mapsto (\fM,N_0)$ from Theorem~\ref{thm-embedding} is defined. If $(\fM,\alpha)$ is $T$-fixed and $\overline{\rho}$ is the trivial $d$-dimensional $G_K$-representation then $N_0 \equiv 0$ modulo $u^{e}\fM$.
\end{cor}

\begin{proof}
    Given the formula for $N_0$ in Theorem~\ref{thm-embedding} this follows if we can show $(\fM,\alpha)$ being $T$-fixed implies the divisibility of the crystalline $G_K$-action on $\fM$ described in \eqref{eq-refined Galois} can be strengthened to
    \begin{equation}\label{eq-stronger GKcondition}
        (\tau-1)^n(m) \in  \fM \otimes_{\fS} u^{e} \varphi^{-1}(\mu)^n A_{\operatorname{inf}}
    \end{equation}
for all $n \geq 1$ and $m \in \fM$. By Lemma~\ref{lem-sum of rank ones}, it suffices to show that any crystalline $G_K$-action on a rank one  Breuil--Kisin module $\fM$ over $\fS_{\bF}$ satisfies \eqref{eq-stronger GKcondition}. 

Any such $\fM$ admits at most one crystalline $G_K$-action. Indeed, any such crystalline $G_K$-action extends the $G_{K_\infty}$-action on $T(\fM)$ to a $G_K$-action, and the crystalline $G_K$-action can be recovered from this extension via the identification from \eqref{eq-another identication}. But $T(\fM)$ being one dimensional over $\bF$ ensures there is a unique extension of the $G_{K_\infty}$-action to a $G_K$-action (compare e.g.\ \cite[Lemma 2.2.1]{B21Doc}). 

As explained in e.g.\ \cite[\S1]{Fon90}, \eqref{eq-another identication} actually descends to an identification over $k((u))^{\operatorname{sep}} \subset C^\flat$. In particular, if $T(\fM)$ has a single generator $\beta$ over $\bF$ then the image of $\fM$ under \eqref{eq-another identication} is generated by $\beta f$ for some $f \in k(\!(u)\!)^{\operatorname{sep}} \otimes_{\bF_p} \bF$. If we further assume the $G_{K_\infty}$-action on $T(\fM)$ is trivial then $f$ is $G_{K_\infty}$-fixed and so $f \in k((u)) \otimes_{\bF_p} \bF$. The claimed divisibility in \eqref{eq-stronger GKcondition} therefore reduces to the assertion that $$
(\tau-1)^n(f) \in u^{e}\varphi^{-1}(\mu)^n f A_{\operatorname{inf},\bF} 
$$
For this note that $\tau^n(u^i) \equiv u^i$ modulo $\mu^n u^i A_{\operatorname{inf}}$ (arguing by induction this follows from the case $n=1$) and recall that $\mu$ generates the same ideal of $A_{\operatorname{inf}}$ as $u^e \varphi^{-1}(\mu)$.
\end{proof}

\subsection{Extremal fixed points and semisimple lifts}\label{sec-extremal}

As in the previous section we assume that $\overline{\rho} = \bigoplus_{\ell=1}^d \overline{\rho}_\ell$ compatibly with the basis $\overline{\alpha}$. It follows from Lemma~\ref{lem-sum of rank ones} that if $(\fM,\alpha) \in \cL^{\operatorname{cr},\operatorname{conv}}_{\lambda,\overline{\rho}}(\bF)$ is $T$-fixed then there exists $\fS_{\bF}$-bases $\beta_i$ of $\fM_i$, compatible with the decomposition in Lemma~\ref{lem-sum of rank ones}, so that 
$$
\beta_{i,\tau} = \beta_{i-1,\tau} \begin{pmatrix}
    u^{\nu_{i,\tau,1}^*} & & \\ & \ddots & \\ & & u^{\nu_{i,\tau,d}^*}\\
\end{pmatrix}
$$
for tuples of integers $(\nu_{i,\tau,1}^*,\ldots,\nu_{i,\tau,d}^*) \leq \lambda_{\kappa(i,\tau)}^*$ for each $1 \leq i \leq e$ and $\tau \in \cJ_0$. Note, we don't require for the usual relationship between $\beta_0$ and $\beta_e$ here.

\begin{defn}
    We say that a $T$-fixed point $(\fM,\alpha) \in \cL^{\operatorname{cr},\operatorname{conv}}_{\lambda,\overline{\rho}}(\bF)$ as above is extremal at $\kappa = \kappa(i,\tau)$ if there is a permutation $w \in S_d$ so that  
    $$
 (\nu_{i,\tau,1}^* , \ldots , \nu_{i,\tau,d}^*) = w (\lambda_{\kappa}^*)
    $$
    We say $(\fM,\alpha)$ is extremal if it is extremal at every $\kappa \in \cJ$.
\end{defn}

\begin{lemma}\label{lem-extremal fixed points have lifts}
     Suppose $(\fM,\alpha) \in \cL^{\operatorname{cr},\operatorname{conv}}_{\lambda,\overline{\rho}}(\bF)$ is an extremal $T$-fixed point. Then there exists a crystalline representation $T^\circ$ of $G_K$ with Hodge type $\lambda$ on a finite free $\cO$-module such that 
     \begin{itemize}
         \item $T^\circ$ is a direct sum of $1$-dimensional crystalline representations and $T \otimes_{\cO} \bF \cong \overline{\rho}$.
         \item If $\fM^\circ \in \cY^{\operatorname{cr},\operatorname{conv}}_{\lambda}(\cO)$ is the Breuil--Kisin module associated to $T^\circ$ (with uniquely determined convolution structure) then $\fM^\circ \otimes_{\cO} \bF \cong \fM$.
     \end{itemize}
\end{lemma}

\begin{proof}
    It is a standard fact that any rank $1$ Breuil--Kisin module over $\bF$ with convolution structure lifts to the Breuil--Kisin module associated to a one dimensional crystalline representation. See, for example, \cite[Example 5.3.2]{B23b}. Furthermore, the Hodge type of this character is uniquely determined by the initial Breuil--Kisin module over $\bF$. Since the formation of Breuil--Kisin modules associated to crystalline representations respects direct sums, it follows that any (not necessarily extremal) $T$-fixed $(\fM,\alpha) \in \cL^{\operatorname{cr},\operatorname{conv}}_{\lambda,\overline{\rho}}(\bF)$ is the base change to $\bF$ of an $\fM^\circ \in \cY^{\operatorname{cr},\operatorname{conv}}_\nu(\cO)$ for a Hodge type $\nu$ with $\nu_{\kappa(i,\tau)}$ the dominant conjugate of $\lbrace \nu_{i,\tau,1},\ldots,\nu_{i,\tau,d} \rbrace$ for $\nu_{i,\tau,\ell} := h - \nu_{i,\tau,\ell}^*$.
    If $(\fM,\alpha)$ is extremal then $\nu = \lambda$, which gives the claim.
\end{proof}

\subsection{Potential diagonalisability}\label{sec-potdiag}

\begin{thm}\label{thm-connect}
    Let $\overline{\rho}$ be the trivial $3$-dimensional $\bF$-representation and suppose $(\fM,\alpha) \in \cL^{\operatorname{cr},\operatorname{conv}}_{\eta,\overline{\rho}}(\bF)$ with $\eta_\kappa = (2,1,0)$ for each $\kappa \in \cJ$. If $p\geq 5$ and $e \equiv 0$ modulo $3$ then $(\fM,\alpha)$ lies in the same connected component of $\cL^{\operatorname{cr},\operatorname{conv}}_{\eta,\overline{\rho}} \otimes_{R_{\overline{\rho}}^{\lambda}} \bF$ as an extremal $T$-fixed point.
\end{thm}
\begin{proof}
    Fix $h =2$ and recall in this case that $\nu_\kappa^* = \nu_\kappa$. For a generic choice of character $\eta: \bG_m \rightarrow T$ the point $\operatorname{lim}_{t \rightarrow 0} \eta(t)\cdot (\fM,\alpha)$ is $T$-fixed and lies in the same connected component as $(\fM,\alpha)$. We can therefore assume $(\fM,\alpha)$ is $T$-fixed and so express the convolution structure on $\fM$ using the notation from Section~\ref{sec-extremal}. Choose a minimal $1 \leq i \leq e$ so that $(\fM,\alpha)$ is not extremal at $\kappa = \kappa(i,\tau)$. Then, $\nu_{i,\tau,1} = \nu_{i,\tau,2} = \nu_{i,\tau,3} =1$, and so $\fM_{i,\tau} = u\fM_{i-1,\tau}$.
    Notice that, after part (2) of Lemma~\ref{lem-props of closed conditions}, this is only possible if $i>1$. Minimality of $i$ therefore ensures $(\fM,\alpha)$ is extremal at $\kappa(i-1,\tau)$ and so we can choose an $\fS_{\bF}$-basis $\beta_{i-2,\tau}$ of $\fM_{i-2,\tau}$ so that $\fM_{i-1,\tau} \subset \fM_{i-2,\tau}$ is generated by $\beta_{i-1,\tau} := \beta_{i-2,\tau}\operatorname{diag}(u^{2},u,1)$. This allows us to vary the submodule $\fM_{i-1,\tau} \subset \fM_{i,\tau}$
     by considering the $\bF[t][[u]]$-submodule $\fM_{i-1,\tau}^{(t)} \subset \fM_{i-2,\tau} \otimes_{\bF[[u]]} \bF[t][[u]]$ generated by
    \begin{equation}\label{eq-vary with t}
        \beta_{i-1,\tau}^{(t)}:=   \Big( \beta_{i-2,\tau} \otimes 1 \Big)\begin{pmatrix} 1 & 0 & 0  \\
         0 & 1 & t \\
         0 & 0 & 1 
    \end{pmatrix}\operatorname{diag}(u^2,u,1)= \Big( \beta_{i-1,\tau} \otimes 1 \Big) \begin{pmatrix}
        1 & 0 & 0 \\ 0 & 1 & \frac{t}{u} \\ 0 & 0 & 1
    \end{pmatrix}
    \end{equation}
     Since $\fM_{i,\tau} \otimes_{\bF} \bF[t]$ is contained in $\fM_{i-1,\tau}^{(t)}$ we can define a convolution structure on the Breuil--Kisin module $\fM^{(t)} := \fM \otimes_{\bF} \bF[t]$ over $\bF[t]$ by setting  
    $$
    \fM_{j,\nu}^{(t)} := \begin{cases}
        \fM_{j,\tau'} \otimes_{\bF[[u]]} \bF[t][[u]] & \text{if $(j,\tau') \neq (i-1,\tau)$}\\
        
        \fM_{i-1,\tau}^{(t)} & \text{if $(j,\tau') = (i-1,\tau)$}
    \end{cases}
    $$
    If the morphism in Theorem~\ref{thm-embedding} is given by $\fM \mapsto (\fM,N_0)$ then we can furthermore consider the pair $(\fM^{(t)},N_0^{(t)}) \in Z^{\nabla,\operatorname{conv}}_{\leq h}(\bF[t])$ where $N_0^{(t)} = N_0 \otimes_{\bF[[u]]} \bF[t][[u]]$. We claim that  $(\fM^{(t)},N_0^{(t)}) \in Y^{\nabla,\operatorname{conv}}_\lambda(\bF[t])$. 
    
    Before checking this claim, let us see how it implies the theorem. Corollary~\ref{cor-true = explicit} asserts that $\fM \mapsto (\fM,N_0)$ defines an isomorphism $\cY^{\operatorname{cr},\operatorname{conv}}_\lambda \otimes_{\cO} \bF \cong Y^{\nabla,\operatorname{conv}}_\lambda$. The claim therefore lets us view $(\fM^{(t)},N_0)$ as an $\bF[t]$-valued point of $\cY^{\operatorname{cr},\operatorname{conv}}_\lambda$.  Since $T(\fM^{(t)}) = T(\fM) \otimes_{\bF} \bF[t]$ we can view $(\fM^{(t)},\alpha \otimes_{\bF} \bF[t])$ as an $\bF[t]$-valued point of $\cL^{\operatorname{cr},\operatorname{conv}}_{\lambda,\overline{\rho}} \otimes_{R_{\overline{\rho}}^{\lambda}} \bF$. Specialising to $t=0$ clearly recovers $(\fM,\alpha)$, while if $(\fM^{(\infty)},\alpha \otimes_{\bF} \bF[t])$ denotes the limit as $t 
\rightarrow \infty$ then the chain of submodules
$$
\fM_{i,\tau}^{(\infty)}  \subset \fM_{i-1,\tau}^{(\infty)} \subset \fM_{i-2,\tau}^{(\infty)}
$$
are generated respectively by $\beta_{i-2} \operatorname{diag}(u^3,u^2,u)$ and $\beta_{i-2} \operatorname{diag}(u^2,1,u)$. Thus, $(\fM^{(\infty)},\alpha)$ is $T$-fixed, and extremal at $\kappa = \kappa(i,\tau)$. Iterating this process therefore creates a chain of $\bP^1$'s in $\cL^{\operatorname{cr},\operatorname{conv}}_{\lambda,\overline{\rho}} \otimes_{R} \bF$ connecting $(\fM,\alpha)$ to an extremal $T$-fixed point, which finishes the proof.

Returning to the claim, we have to check $(\fM^{(t)},N_0^{(t)})$ that, after choose $\bF[t]$-bases $\beta_{\bullet}^{(t)}$ of $\fM_\bullet$, the triple $(\fM^{(t)},N_0^{(t)},\beta_\bullet^{(t)})$ satisfies the conditions $(\mathbf{B}_{i-1})$ and $(\mathbf{B}_i)$ from Construction~\ref{con-impose mod p equations} at the embedding $\tau$. We can choose $\beta_\bullet^{(t)}$ so that $\beta_{i-2}^{(t)} = \beta_{i-1} \otimes 1$ and $\beta_{i-1}^{(t)}$ is as in \eqref{eq-vary with t}. We need to compute the restrictions of $u^{i-1}N_0^{(t),\varphi}$ and $u^i N_0^{(t),\varphi}$ to $\fM_{i-2,\tau}^{(t)}$ and $\fM_{i-1,\tau}^{(t)}$ respectively. Then we will check the equation from Example~\ref{exam} vanishes on $\cE_j := \Psi(\fM_{j,\tau}^{(t)},\beta_{j-1})$ for $j=i,i-1$. Corollary~\ref{cor=N0=0} ensures the existence of an $\bF[t]$-basis $(e_1,e_2,e_3)$ of $\varphi^*\fM$ so that $N_0(e_1,e_2,e_3) \equiv 0$ modulo $u^{ep}\varphi^*\fM$. On the other hand, Lemma~\ref{lem-sum of rank ones} ensures $e_1,e_2,e_3$ can be chosen so that $\beta_{i-2} =(u^{r_1}e_1,u^{r_2}e_2,u^{r_3}e_3)$ and $\beta_{i-1} =(u^{r_1+2}e_1,u^{r_2+1}e_2,u^{r_3}e_3)$ for some $r_i \geq 0$. Therefore,
$$
u^{i-1}N_0^{\varphi}(\beta_{i-2}^{(t)}) \equiv \beta_{i-2}^{(t)} u^{i-1}c(u)\Bigg(\begin{smallmatrix}
    r_1 & 0 & 0 \\ 0 & r_2 & 0 \\ 0 & 0 & r_3
\end{smallmatrix}\Bigg), \quad u^iN_0^{(t),\varphi}(\beta_{i-1}^{(t)}) \equiv \beta_{i-1}^{(t)} u^{i}c(u)\Bigg(\begin{smallmatrix}
    r_1+2 & 0 & 0 \\ 0 & r_2 + 1& \frac{t(r_2-r_3-1)}{u}\\ 0 & 0  & r_3
\end{smallmatrix}\Bigg)
$$
modulo $u^{ep+i-1} \varphi^*\fM^{(t)}$ and $u^{ep+i}\varphi^*\fM^{(t)}$ respectively. Thus, these congruences also hold respectively modulo $u^h\fM_{i-2}^{(t)}$ and $u^h \fM_{i-1}^{(t)}$. If $i >2$ then both matrices above are $\equiv 0$ modulo $u^2$, and so substituting $n_{ij}=0$ into the equations from Example~\ref{exam} shows $(\mathbf{B}_{i-1})$ and $(\mathbf{B}_{i})$ hold. If $i=2$ then $r_1=r_2=r_3 =0$ and so, after substituting $n_{23} = -t$ and all other $n_{ij} =0$ in the equations from Example~\ref{exam}, express $(\mathbf{B}_{1})$ and $(\mathbf{B}_2)$ as the vanishing of 
$$
y_{0,1,2}, \quad \text{ respectively }\quad t y_{0,1,4}
$$
as functions on $\Psi(\fM_{i-1,\tau}^{(t)},\beta_{i-2}^{(t)})$ and $\Psi(\fM_{i,\tau}^{(t)},\beta_{i-1}^{(t)})$ respectively. If $\beta_{j}^{(t)} = (e_{1,j},e_{2,j},e_{3,j})$ then the first vanishing asks that $e_{1,i-2}\wedge e_{2,i-2}\wedge e_{3,i-2}$ is zero inside $\bigwedge^3 \fM_{i-2}^{(t)}/ \fM_{i-1}^{(t)}$ and the second asks that $e_{1,i-1}\wedge e_{2,i-1} \wedge ue_{2,i-1}$ is zero inside $\bigwedge^3 \fM_{i-1}^{(t)}/\fM_{i}^{(t)}$. Both these are easy to check using \eqref{eq-vary with t}. 
\end{proof}
Recall, from \cite[\S1.4]{BLGGT14}, that a potentially crystalline representation $\rho:G_K \rightarrow \operatorname{GL}_d(\cO)$ is potentially diagonalisable if there exists a finite extension $L/K$ so that the restriction of $\rho[\frac{1}{p}]$ to $G_L$ corresponds to an $\cO[\frac{1}{p}]$-valued point of $\operatorname{Spec}R_{\overline{\rho}}^{\operatorname{cr},\lambda}$, for some Hodge type $\lambda$ and $\overline{\rho} = \rho \otimes_{\cO} \bF |_{G_L}$, contained in the same connected component as an $\cO[\frac{1}{p}]$-valued point induced by direct sum of $1$-dimensional crystalline representations.
\begin{cor} \label{cor:potdiag}
    Let $K$ be any finite extension of $\bQ_p$ with $p \geq 5$. Then any potentially crystalline $\rho:G_K \rightarrow \operatorname{GL}_3(\cO)$ with Hodge type $(2,1,0)$ is potentially diagonalisable.
\end{cor}
\begin{proof}
    If $\rho:G_K \rightarrow \operatorname{GL}_3(\cO)$ has Hodge type $(2,1,0)$ then so does its restriction to a finite index subgroup of $G_K$. Enlarging $K$ we can therefore assume that $\overline{\rho} = \rho \otimes_{\cO} \bF$ is trivial and $\rho$ is crystalline. We can also assume $e \equiv 0$ modulo $3$. The corollary therefore follows by combining Theorem~\ref{thm-connect}, Proposition~\ref{prop-connectedness in Kisin var}, and Lemma~\ref{lem-extremal fixed points have lifts}.
\end{proof}

\section{Applications} 

In this section, we prove our main results on automorphy lifting, Breuil--M\'ezard conjecture, and the weight part of Serre's conjecture.

\subsection{Global setup} \label{sec:globalpatching} 

Let $F$ be an imaginary CM field with maximal totally real subfield
$F^+$, and let $c \in \Gal(F/F^+)$ denote complex conjugation. Fix a prime $p \geq 5$, $n =3$, and an isomorphism $\iota:\overline{\Q}_p \cong \mathbb{C}$.
  Write $\Sigma_p^+$ (resp. $\Sigma_p$) for the places of $F^+$ (resp. of $F$) lying above $p$.  Assume that all places of $\Sigma_p^+$ split in $F$.

\subsubsection{Algebraic modular forms} \label{sec:modforms} 
We now recall algebraic modular forms.  Let $G_{/F^+}$ be a reductive group which is an outer form for $\GL_3$ which is quasi-split at all finite places of $F^+$  and which splits over $F$.
Suppose that $G(F^+_v) \cong U_3(\R)$ for all $v|\infty$.
Recall from \cite[\S 7.1]{EGH} that $G$ admits a reductive model $\cG$ defined over $\cO_{F^+}[1/N]$, for some $N\in \N$ which is prime to $p$, together with an isomorphism
\begin{equation}
\label{iso integral}
\iota:\,\cG_{/\cO_{F}[1/N]} \stackrel{\iota}{\rightarrow}{\GL_3}_{/\cO_{F}[1/N]}
\end{equation}
which specializes to
$
\iota_w:\,\cG(\cO_{F^+_v})\stackrel{\sim}{\rightarrow}\cG(\cO_{F_w})\stackrel{\iota}{\rightarrow}\GL_3(\cO_{F_w})
$
for all places  $v\in \Sigma_p^+$.

Define $F_p^+:= F^+\otimes_{\Q}\Q_p$ and $\cO_{F^+_p}:=\cO_{F^+}\otimes_\Z\Z_p$. If $W$ is a finite $\cO$-module endowed with a continuous action of $\cG(\cO_{F^+_p})$ and 
$U\leq G(\A_{F^+}^{\infty,p})\times\cG(\cO_{F^+_p})$ is a compact open subgroup, the space of algebraic automorphic forms on $G$ of level $U$ and coefficients in $W$ is the $\cO$-module defined as:
\begin{equation}
S(U,W) := \left\{f:\,G(F^{+})\backslash G(\A^{\infty}_{F^{+}})\rightarrow W\,|\, f(gu)=u_p^{-1}f(g)\,\,\forall\,\,g\in G(\A^{\infty}_{F^{+}}), u\in U\right\}.
\end{equation}
We recall that the level $U$ is said to be \emph{sufficiently small} if for all $t \in G(\bA^{\infty}_{F^+})$, the finite group $t^{-1} G(F^+) t \cap U$
is of order prime to $p$.
The space of algebraic automorphic forms $S(U,W)$ is then endowed with an action of the Hecke algebra $\bT_{\cP}$ where $\cP$ is a set of ``good" finite places of $F$ (see \cite[Section 9.1]{LLLM-models} for details).

\begin{defn} \label{defn:Serreweight}
A \emph{Serre weight} (\emph{for} $\cG$) is an isomorphism class of a smooth, absolutely irreducible representation $V$ of $\cG(\cO_{F^+_p})$ over $\overline{\F}$.
If $v|p$ is a place of $F^+$, a \emph{Serre weight at $v$} is an isomorphism class of a smooth, absolutely irreducible representation $V_v$ of $\cG(\cO_{F^+_v})$. 
  Any Serre weight $V$ for $\cG(\cO_{F^+_p})$ can be written as $V\cong \underset{v \in \Sigma_p^+}{\bigotimes}V_v$ where $V_v$ are Serre weights at $v$.
\end{defn}

Let $\overline{r}:G_F\rightarrow \GL_3(\F)$ be a continuous Galois representation.  Recall that a maximal ideal $\overline{\mathfrak{m}} \subset \bT_{\cP}$ corresponds to $\overline{r}$ if it is the kernel of the system of Hecke eigenvalues $\overline{\alpha}:\bT_{\cP}\rightarrow \F$ satisfying the equality
$$
\det\left(1-\overline{r}^{\vee}(\mathrm{Frob}_w)X\right)=\sum_{j=0}^3 (-1)^j(\mathbf{N}_{F_w/\Qp}(w))^{\binom{j}{2}}\overline{\alpha}(T_w^{(j)})X^j
$$
for all $w\in \cP$.

\begin{defn}
\label{definition modularity}
Let $\overline{r}:G_F\rightarrow \GL_3(\F)$ be a continuous Galois representation and let $V$ be a Serre weight for $\cG$. We say that $\overline{r}$ is \emph{automorphic of weight $V$} (or that $V$ is a Serre weight of $\overline{r}$) if there exists a compact open subgroup $U = U^p \times  \cG(\cO_{F^+,p}) \subset G(\bA^{\infty}_F) $ with $U^p \subset G(\bA^{\infty,p}_F)$ such that
$$
S(U,V)_{\overline{\mathfrak{m}}}\neq0.
$$
wher $\overline{\mathfrak{m}}$ corresponds to $\rbar$. 

We write $W(\rbar)$ for the set of all Serre weights of $\rbar$.
We say that $\rbar$ is \emph{automorphic} if $W(\rbar)\neq \emptyset$.
\end{defn}

\subsection{Automorphy lifting} \label{sec:automorphy}

   We refer to Section 2.1 of \cite{BLGGT14} for the definition of a RAECSDC (regular, algebraic essentially conjugate self dual, cuspidal) automorphic representation $(\pi, \chi)$ of $\GL_3(\mathbb{A}_F)$.   Recall that by Theorem 1.1 of \cite{BLGHT}, we can attach to a RAECSDC representation a continuous semisimple representation
\[
r_{\iota}(\pi):G_{F} \rightarrow \GL_3(\overline{\Q}_p).
\]
Given the Corollary \ref{cor:potdiag}, we immediately deduce the following automorphy lifting theorem:

\begin{thm} \label{thm:lifting}  
Let $F$ be an imaginary CM field with maximal totally real subfield $F^+$, and let
$c \in \Gal(F/F^+)$ denote the non-trivial element. Assume $F$ is split at all places of $F^+$ above $p \geq 5$ and $\zeta_p\notin F$. Let $r : G_F \to \GL_3(\overline{\Q}_p)$ be a continuous irreducible representation such that:
\begin{enumerate}
\item \emph{(odd essential conjugate self-duality)}  
There is an isomorphism $r^c \cong r^\vee \otimes \chi$, and $\chi(c_v) = -1$ for all
$v | \infty$ for a character $\chi : G_{F^+} \to \overline{\Q}_p^\times$.   

\item \emph{(unramified almost everywhere)}  
The representation $r$ is ramified at only finitely many primes.

\item \emph{(minimal regular potentially crystalline)}  
For all places $v | p$, the restriction $r|_{G_{F_v}}$ is potentially crystalline with $\kappa$-Hodge-Tate weights $(2,1,0)$ for all $\kappa : F_v \hookrightarrow \Q_p$. 
\item \emph{(adequate)} $\rbar|_{ G_{F(\zeta_p)}}$ is irreducible and $\rbar( G_{F(\zeta_p)})$ is an adequate subgroup of $\GL_3(\overline{\F})$; and
\item \emph{(residual modularity)} $\rbar \cong \rbar_\iota(\pi)$ for some $\pi$ a regular algebraic conjugate self-dual cuspidal  automorphic representation of $\GL_n(\A_F)$. 
\end{enumerate}
Then $r$ is automorphic, i.e.~$r \cong r_\iota(\pi')$ for some $\pi'$ a RAECSDC automorphic representation of $\GL_3(\A_F)$ \emph{(}of weight $0$ at all places dividing $p$\emph{)}.
\end{thm}

\begin{proof}
This essentially follows from Theorem 4.2.1 of \cite{BLGGT14} with the improvement in Theorem A.1.4 of \cite{BLGG13} related to point (4).  The main difference is that we don't assume $r$ is ordinarily automorphic, and so instead need to show that $\bar r$ satisfies the second condition in (3) of the theorem in \cite{BLGGT14}: namely $\rbar \cong \rbar_\iota(\Pi)$ such that $\Pi$ has level potentially prime to $p$ and $r_\iota(\Pi)|_{G_{F_v}}$ is potentially diagonalizable for all $v | p$. 

To do this, recall that by assumption we have 
\[S(U,V)_{\overline{\mathfrak{m}}}\neq0,\]
for a maximal ideal $\mathfrak{m}\subset \mathbb{T}_{\mathcal{P}}$ corresponding to $\rbar$ and some Serre weight $V$ of $\cG$. 
By Theorem \ref{thm: rational span}, there exists a (regular) $K$-type $\sigma(\tau)$  such that $V$ is a Jordan--H\"older factor of $\overline{\sigma}(\tau)$. 
Let $\sigma^{\circ}(\tau)$ be an $\cO$-lattice in $\sigma(\tau)$.  


\[S(U,\sigma^{\circ}(\tau))_{\overline{\mathfrak{m}}}/\varpi=S(U,\sigma^{\circ}(\tau)/\varpi)_{\overline{\mathfrak{m}}}\neq0.\]
But any system of Hecke eigenvalues supported on $S(U,\sigma(\tau))_{\overline{\mathfrak{m}}}$ gives rise to a RAECSDC automorphic representation $\Pi$ of $\GL_n(\A_F)$ of weight $0$ congruent to $\pi$. By Corollary \ref{cor:potdiag}, $r_\iota(\Pi)|_{G_{F_v}}$  is potentially diagonalizable for all $v | p$. 



\end{proof}


\subsection{Breuil--M\'ezard conjecture}


Fix $K/\Qp$ finite.  We recall the geometric and versal Breuil--M\'ezard conjectures following \cite{LLLM-models} which in turn follows \cite{EGstack, EG14}.    Let $\cX_{3}$ denote the EG stack of $3$-dimensional $p$-adic representations of $G_K$.

Let $\Z[\cX_{3,\mathrm{red}}]$ denote the free abelian group on the irreducible components $\cC_\sigma$ of $\cX_{3,\mathrm{red}}$ parametrized by Serre weights $\sigma$ for $\GL_3(\cO_K)$.
We call elements of $\Z[\cX_{3,\mathrm{red}}]$ \emph{cycles} and, for a Serre weight $\sigma$, call $\cC_\sigma\in \Z[\cX_{3,\mathrm{red}}]$ an \emph{irreducible cycle}.
(One might normally call these \emph{top-dimensional} cycles among cycles of varying dimension, but since we only consider top-dimensional cycles, we omit this adjective.)
A cycle is \emph{effective} if its coefficients are nonnegative.
Let $K_0(\Rep_{\F}(\GL_3(\cO_K)))$ be the Grothendieck group of finite length $\F[\GL_3(\cO_K)]$-modules, or equivalently the free abelian group generated by Serre weights for $\GL_3(\cO_K)$.
If $W$ is a finite length $\F[\GL_3(\cO_K)]$-module, we write $[W] = \sum_\sigma [W:\sigma][\sigma]$ for its image in $K_0(\Rep_{\F}(\GL_3(\cO_K)))$ where $[W:\sigma]$ denotes the multiplicity of a Serre weight $\sigma$ as a Jordan--H\"older factor of $W$.
If $V$ is a finite length $E[\GL_3(\cO_K)]$-module, the class $[\ovl{V}^\circ]$ of the mod $\varpi$ reduction of a $\GL_3(\cO_K)$-stable $\cO$-lattice $V^\circ \subset V$ is independent of the choice of the lattice, and hence we simply denote it by $[\ovl{V}]$.
We then also denote $[\ovl{V}^\circ:\sigma]$ by $[\ovl{V}:\sigma]$.

Recall that an inertial type $\tau$ for $K$ is a representation $I_K \rightarrow \GL_3(E)$  with finite image which extends to a representation of the Weil group of $K$. In what follows, we use $\lambda$ to denote the Hodge type (as opposed to Hodge--Tate weights); this is the reason for the shift by $\eta = (2,1,0)$ in various formulas below.  
Given a pair $(\lambda,\tau)$ where $\lambda\in (\mathbb{Z}^3)^{\cJ}$ is a dominant weight and $\tau$ is an inertial type for $K$,  let $\cX_{\lambda+\eta}^{\tau}$  be the potentially crystalline stack $\cX_{\lambda+\eta}^{\tau}$ parametrizing potentially crystalline representations with Hodge type $\lambda$ and inertial type $\tau$. 
Let $\cZ_{\lambda,\tau}$ denote the cycle 
\[
\sum_\sigma \mu_\sigma(\cX_{\lambda+\eta,\F}^{\tau}) \cC_\sigma
\]
in $\Z[\cX_{3,\mathrm{red}}]$ where $\mu_\sigma(\cX_{\lambda+\eta, \F}^{\tau})$ denotes the multiplicity of $\cC_\sigma$ as an irreducible component of $\cX_{\lambda+\eta,\F}^{\tau}$. 

Given an inertial type $\tau$ for $K$, there is a finite-dimensional smooth representation $\sigma(\tau)$ of $\GL_3(\cO_K)$ over $E$ associated to $\tau$ by the ``inertial local Langlands correspondence'' (see \cite[Theorem 3.7]{CEGS+16} for a characterization). The following conjecture is based on a geometric version of a conjecture of Breuil--M\'ezard (\cite{BM02}, \cite[Conjecture 8.2.2]{EGstack}, \cite[8.1.1]{LLLM-models}).

\begin{conj}[Geometric Breuil--M\'ezard conjecture]\label{conj:S-BM}  Let $\cS$ be a collection of pairs $(\lambda, \tau)$.  
For each Serre weight $\sigma$ for $\GL_3(\cO_K)$, there exists an effective cycle $\cZ_\sigma \in \Z[\cX_{3,\mathrm{red}}]$ such that for all $(\lambda,\tau) \in \cS$, we have
\begin{equation} \label{eqn:BMgeneral}
\cZ_{\lambda,\tau} = \sum_{\sigma} [\ovl{\sigma}(\lambda,\tau):\sigma] \cZ_\sigma.
\end{equation}
\end{conj}

The full conjecture predicts that one can take $\cS$ to be all pairs. 
 In this paper, we will restrict attention to the collection $\cS_{0} = \{(0, \tau) \mid \tau  \text{ is an inertial type for } K \}$. 

\begin{thm} \label{thm:BMconjeta} Assume $p \geq 5$. Then the geometric Breuil--M\'ezard  Conjecture \ref{conj:S-BM} holds for the collection $\cS_0$.  That is, there exist unique cycles $\cZ_\sigma \in \Z[\cX_{3,\mathrm{red}}]$ such that for all $\tau$, we have
\begin{equation} \label{eqn:BMeta}
\cZ_{0,\tau} = \sum_{\sigma} [\ovl{\sigma}(\tau):\sigma] \cZ_\sigma.
\end{equation}
Moreover, the cycles $\cZ_{\sigma}$ are effective.   
\end{thm}

\begin{rmk} Not much is known about these cycles but we can extract some information from previous work. 
\begin{enumerate} 
\item In general, the support of $\cZ_{\sigma}$ includes $\cC_{\sigma}$ (in fact, appearing with multiplicity one).  This follows from the compatibility of $\cZ_\sigma$ with patching and the existence of ordinary globalizations for ordinary Serre weights (for example, as constructed in the proof of \cite[Proposition 2.5.7]{LL-CM}).
\item If $K$ is unramified and $\sigma$ is generic, then $\cZ_{\sigma} = \cC_{\sigma}$ (see Proposition 3.6.1 \cite{LLLM-lattice}).
\item  For $K$ unramified, the second two authors with Daniel Le and Stefano Morra have on-going work to fully determine the cycles $\cZ_{\sigma}$ and hence to make explicit the weight part of Serre's conjecture (see Remark \ref{rmk:BMC}.) 
\end{enumerate}
\end{rmk}

 Recall from \cite[\S 3.3]{GHS} that a collection $\cS$ of pairs $(\lambda, \tau)$ is a \emph{Breuil--M\'ezard system} if the map 
\begin{align*}
\Z[\cS] &\ra K_0(\Rep_{\F}(\GL_3(\cO_K))\\
(\lambda,\tau) &\mapsto [\ovl{\sigma}(\lambda,\tau)]
\end{align*}
has finite cokernel.  
The uniqueness assertion in Theorem \ref{thm:BMconjeta} follows from the fact that $\cS_0$ is a  \emph{Breuil--M\'ezard system} (which in proved in Theorem \ref{thm: rational span}). Indeed, one is forced to take
\[\cZ_\sigma=\sum a_\tau \cZ_{0,\tau}\]
where $a_\tau\in \Q$ are such that 
\[[\sigma]=\sum a_\tau [\overline{\sigma}(\tau)].\]

Our proof of Theorem \ref{thm:BMconjeta} will proceed by patching and passing to versal rings for $\cX_3$.  As explained in \cite[Proposition 8.2.1]{LLLM-models}, taking versal rings for $\cX_3$ recovers the original Breuil--M\'ezard conjecture.  We summarize this discussion (see \emph{loc. cit} for details).  Let $\rhobar:G_K \ra \GL_3(\F)$ and let $\rhobar$ denote the corresponding $\F$-point of $\cX_{3}$. Fix a versal ring $R_{\rhobar}^{\mathrm{ver}}$ for $\cX_3$ at $\rhobar$.  The fiber product $\Spf R_{\rhobar}^{\mathrm{ver}} \times_{\cX_3} \cX_{3,\mathrm{red}}$ is a closed formal subscheme of $\Spf \, R_{\rhobar}^{\mathrm{ver}}$, which we denote by $\Spf \, R_{\rhobar}^{\mathrm{alg}}$. Consider
 \[
i_{\rhobar}: \Spec R_{\rhobar}^{\mathrm{alg}}\ra \cX_{3,\mathrm{red}}.
\]
The map $i_{\rhobar}$ induces a map from the set of irreducible components of $\Spec R_{\rhobar}^{\mathrm{alg}}$ to the set of irreducible components $\cX_{3,\mathrm{red}}$.

Denote by $\Z[\Spec R_{\rhobar}^{\mathrm{alg}}]$ the free abelian group generated by irreducible components of $\Spec R_{\rhobar}^{\mathrm{alg}}$.  The map $i_{\rhobar}$ induces a pullback map $i_{\rhobar}^*: \Z[\cX_{3,\mathrm{red}}] \ra \Z[\Spec R_{\rhobar}^{\mathrm{alg}}]$ obtained by formal completion along components of $\cX_{3,\mathrm{red}}$.   Let $\cZ_{\lambda,\tau}(\rhobar)$ denote the cycle $i_{\rhobar}^*(\cZ_{\lambda,\tau})\in \Z[\Spec R_{\rhobar}^{\mathrm{alg}}]$.  

\begin{conj}[Versal Breuil--M\'ezard conjecture]\label{conj:v-S-BM} Let $\cS$ be a collection of pairs $(\lambda, \tau)$.  
For each Serre weight $\sigma$ for $\GL_3(\cO_K)$, there exist effective cycles $\cZ_\sigma(\rhobar)$ in $\Spec R_{\rhobar}^{\mathrm{alg}}$ such that for all $(\lambda,\tau) \in \cS$, we have
\begin{equation} \label{eq:BMeqn3}
\cZ_{\lambda,\tau}(\rhobar) = \sum_{\sigma} [\ovl{\sigma}(\lambda,\tau):\sigma] \cZ_\sigma(\rhobar).
\end{equation}
\end{conj}
The conjecture doesn't depend on the choice of versal ring (Remark 8.1.6 in \cite{LLLM-models}), so in what follows we will make the universal lifting ring $R^{\square}_{\rhobar}$ our default choice.

Conjecture \ref{conj:S-BM} for a collection $\cS$ clearly implies the versal conjecture for $\cS$ so Theorem \ref{thm:BMconjeta} implies Conjecture \ref{conj:v-S-BM}. 
The converse is also true if one knows the versal conjecture for enough $\rhobar$: 
\begin{prop}  Let $\cP$ be a collection of $\F$-points of $\cX_{3}$. Let $\cS$ be a Breuil--M\'ezard system.   For every Serre weight $\sigma$, assume that $\cP$ contains a $\rhobar \in \cC_{\sigma}$ such that $\cC_{\sigma}$ is smooth at $\rhobar$ and $\rhobar$ does not lie on any other components. If Conjecture \ref{conj:v-S-BM}  for the collection $\cS$ holds with effective cycles $\cZ_\sigma(\rhobar)$ for all $\rhobar$  in $\cP$, then there exists cycles $\cZ_{\sigma} \in \Z[\cX_{3,\mathrm{red}}]$ such that Conjecture \ref{conj:S-BM} holds for $\cS$.  
\end{prop}
\begin{proof} The equivalence of the two conjectures is proved in Remark 8.3.7 in \cite{EGstack} and the discussion before. 
\end{proof}

\begin{thm} \label{thm:versalBMconj} Let $\rhobar \in \cX_3(\F)$ and $p \geq 5$.  Conjecture \ref{conj:v-S-BM} holds
 for $S_{0}$.
\end{thm}
\begin{proof}

We follow the geometric version of the strategy of Gee--Kisin \cite{gee-kisin} formulated in \cite{EG14}.  First, we globalize $\rhobar$.  Since $p \geq 5$, there is an imaginary CM field $F$ and suitable globalization $\rbar:G_{F} \rightarrow \GL_3(\F)$ as in Corollary A.7 in \cite{EG14} (noting that Conjecture A.3 in \emph{loc. cit.} is true by Theorem 6.4.4 in \cite{EGstack}).  In particular we have a global setup as in Section \ref{sec:globalpatching} such that:
\begin{itemize}
\item $\rbar$ is unramified away from $p$;
\item  letting $\Sigma^+_p$ denote places of $F^+$ over $p$, for each $v \in \Sigma^+_p$ there is a place $\tld{v}$ of $F$  lying over $v$ (which we fix) such that $F_{\tld{v}} \cong K$ and $\rbar|_{G_{F_{\tld{v}}}} \cong \rhobar$;
\item $\rbar$ is automorphic.
\end{itemize}
Let $(\lambda, \tau) = (\lambda_v, \tau_v)_{v \in \Sigma_p}$ so that each $\lambda_v \in (\mathbb{Z}^n)^{\Hom(F_{\tld{v}}, \overline{\Q}_p)}$ is a Hodge type and $\tau_v$ an inertial type. Let $\sigma(\lambda,\tau)$ be the finitely generated $E[\GL_3(\cO_{p})]$-module $\otimes_{v} (V(\lambda_v) \otimes \sigma(\tau_v))$ where $V(\lambda_v)$ is the (dual) Weyl module with highest weight $\lambda_v$. 
Let  $R_{\rhobar}^{\lambda_v+ \eta,\tau_v}$ denote the reduced $\cO$-flat quotient of the lifting ring $R_{\rhobar}^{\square}$ such that its $E'$-points are exactly the potentially crystalline representations $\rho: G_L \ra \GL_n(E')$ of Hodge--Tate weights $\lambda_v + \eta$ and inertial type $\tau_v$ for $E'/E$ finite.  Note that $R_{\rhobar}^{\lambda_v+ \eta,\tau_v}$  is the pullback of $\cX^{\tau_v}_{\lambda_v + \eta}$ to  $R_{\rhobar}^{\square}$.  

Set 
\[R_{\rhobar} := \widehat{\bigotimes}_{v|p,\cO} R_{\rhobar}^{\square}\]
\[R_{\rhobar}^{\lambda+\eta,\tau} := \widehat{\bigotimes}_{v|p,\cO} R_{\rhobar_{v}}^{\lambda_{v}+\eta,\tau_{v}}.
\]
 
In this global setup, the patching argument in \cite[\S 2]{CEGS+16} produces a minimal weak patching functor in the sense of \cite[Definition 6.1.1]{LLLM-models}. Let $\Rep_{\cO}(\GL_3(\cO_{F^+_p}))$ denote the category of topological $\cO[\GL_3(\cO_{F^+_p})]$-modules which are finitely generated over $\cO$. There is a formally smooth complete local Noetherian equidimensional flat $\cO$-algebra $R^{\mathrm{aux}}$ with residue field $\F$, and a covariant exact functor
 \[M_\infty:\Rep_{\cO}(\GL_3(\cO_{F^+_p}))\ra \mathrm{Mod}(R_{\infty})\]
where $R_\infty=  R^{\mathrm{aux}}\widehat{\otimes}_{\cO} R_{\rhobar}$, such that
\begin{enumerate}
\item (\cite[Lemma 4.1.18]{CEGS+16}) Let $\sigma^\circ(\lambda,\tau)$ be an $\cO$-lattice in $\sigma(\lambda,\tau)$. Then $M_\infty(\sigma^\circ(\lambda,\tau))$ is supported on $R_{\infty}(\lambda,\tau):=R_\infty\otimes_{R_{\rhobar}}R_{\rhobar}^{\lambda+\eta,\tau}$, and is maximal Cohen-Macaulay over it. Furthermore, $M_\infty(\sigma^\circ(\lambda,\tau))[\frac{1}{p}]$ is locally free of rank $1$ over its support.
\item For all $\sigma \in \JH(\ovl{\sigma}(\lambda,\tau))$, $M_\infty(\sigma)$ is a maximal Cohen--Macaulay module over $R_\infty(\lambda, \tau)/\varpi$ (or is $0$).
\end{enumerate}
In what follows, we will only consider the case $\lambda=0$, and hence suppress it from the notations (e.g we abbreviate $R_\infty(0,\tau)=R_\infty(\tau)$). Let $d$ (resp. $d_v$) be the common relative dimension over $\cO$ of $R_\infty(\tau)$ (resp. $R_{\rhobar}^{\eta,\tau_v}$). 

Let $\tau=(\tau_v)$ be such that $R_\infty(\tau)\neq 0$. 
By Corollary \ref{cor:potdiag}, any closed point of $\Spec\, R_\infty(\tau)[\frac{1}{p}]$ is potentially diagonalizable. Hence for each irreducible component $\cC$ of $\Spec\, R_\infty(\tau)$, we can apply \cite[Theorem 4.4.1]{BLGGT14} with the set $S = \Sigma^+_p$ to produce an automorphic lift $r$ (with respect to our current global setting) of $\rbar$ such that $(r_{\tld{v}})_{v \in \Sigma^+_p}$ induces a point in $\cC$. Since the support of $M_\infty(\sigma^\circ(\tau))$ is a union of irreducible components of $R_\infty(\tau)$ by the maximal Cohen-Macaulay property, we conclude that its support is all of $R_\infty(\tau)$. It follows that $M_\infty(\sigma^\circ(\tau))[\frac{1}{p}]$ is locally free of rank $1$ over $R_\infty(\tau)[\frac{1}{p}]$.
This implies 
\begin{equation} \label{BMproof1}
Z(R_{\infty}(\tau)/\varpi) = Z(M_\infty(\sigma^\circ(\tau)/\varpi)),
\end{equation}
where for $R_\infty$-module $M$ which is supported on $R_\infty(\tau)/\varpi$ for some $\tau$, we let $Z(M)$ denote its associated $d$-dimensional cycle. 

We now fix a place $v_0 \in \Sigma^+_p$. For each $v \in \Sigma^+_p \setminus \{v_0\}$, we fix a choice of inertial type $\tau_v$ such that $R^{\eta,\tau_v}_{\rhobar}\neq 0$ as well as a choice of $\cO$-lattice $\sigma^\circ(\tau_v)$. We have an injective pullback map
\[
j_{\infty}:Z_{d_v}[\Spec R^{\eta, \tau_{v_0}}_{\rhobar}/\varpi]\hookrightarrow
Z_d[\Spec(R_{\infty}/\varpi)]
\] 
induced by tensoring with $R^{\mathrm{aux}} \widehat{\otimes} (\widehat{\otimes}_{v \in \Sigma^+_p \setminus \{ v_0 \}}  R^{\eta,\tau_v}_{\rhobar}/ \varpi)$. Clearly 
\begin{equation} \label{BMproof2}
j_{\infty}(\cZ_{0,\tau_{v_0}}(\rhobar))= j_{\infty} Z(R^{\eta,\tau_v}_{\rhobar}/\varpi) = Z(R_\infty(\tau)/\varpi)
\end{equation}
Now for each Serre weight $\sigma_{v_0}$ for $\GL_3(\cO_K)$, we choose a presentation
\[[\sigma_{v_0}]=\sum a_{\sigma,\tau} [\overline{\sigma}(\tau_{v_0})]\]
in $K_0(\F[\GL_3(\cO_K)])$ and set 
\[\cZ_{\sigma_{v_0}}(\rhobar)=\sum a_{\sigma,\tau} \cZ_{0,\tau_{v_0}}(\rhobar),\]
which is a priori a possibly rational combination of cycles.
Combining equations \eqref{BMproof1}, \eqref{BMproof2}, the exactness of $M_\infty$ and the additivity of cycles in short exact sequences, we deduce
\[j_{\infty}(\cZ_{\sigma_{v_0}}(\rhobar))=Z(M_{\infty}(\sigma_{v_0} \otimes (\otimes_{v \in \Sigma^+_p \setminus \{v_0\}} \sigma^{\circ}(\tau_v)/\varpi))))\]
In particular this shows $\cZ_{\sigma_{v_0}}(\rhobar)$ are effective cycles.
Finally, the system of equations \eqref{eq:BMeqn3} with $\lambda=0$ hold for $\cZ_{\sigma_{v_0}}(\rhobar)$  since they hold after applying the injective map $j_{\infty}$.

\end{proof}



%
%

\subsection{Weight part of Serre's conjecture} \label{sec:WPSC}

Let's return to the global setup with $F/F^+$ as in Section \ref{sec:globalpatching}.  Let $\overline{r}:G_F\rightarrow \GL_3(\F)$ be a continuous Galois representation.   The existence of the cycles $\cZ_{\sigma}$  in Theorem \ref{thm:BMconjeta} allows us to unconditionally state and prove the weight part of Serre's conjecture.   (This connection to the Breuil--M\'ezard conjecture in this generality first appears in \cite{GHS}.)

Continue to assume $p \geq 5$. Let $\rbar_p$ be the collection $\rbar_v := \rbar|_{G_{F_{\widetilde{v}}}}$. For each $v$ and a Serre weight $\sigma_v$ for $G(k_v)$, let $Z_{\sigma_v}$ be the effective Breuil--M\'ezard cycle from Theorem \ref{thm:BMconjeta}.   

\begin{defn} (Definition 3.2.7 \cite{GHS})  Define
\[
W^{\mathrm{BM}}(\rbar_v) = \{ \sigma_v \mid \rbar_v \text{ lies in the support of } Z_{\sigma_v} \}.
\] 
\end{defn}

\begin{thm} \label{thm:WPSC} $($Weight part of Serre's conjecture$)$  Assume $p \geq 5$. Suppose that
\begin{enumerate}
\item if $\rbar$ is ramified at a place $\tld{v}$ of $F$, then $\tld{v}|_{F^+}$ is split in $F$;

\item $\rbar: G_{F(\zeta_p)} \ra \GL_3(\F)$ is an adequate subgroup and $\zeta_p\notin \ovl{F}^{\ker\ad\rbar}$; and
\item $\rbar$ is automorphic in the sense of Definition \ref{definition modularity}. 
\end{enumerate}
Then
\[
\sigma = \otimes_{v } \sigma_v \in W(\rbar) \text{ if and only if }  \sigma_v \in W^{\mathrm{BM}}(\rbar_v).
\]
In particular, the modular weights of $\rbar$ only depend on the local components above $p$. 
\end{thm}

\begin{proof}
We patch in this setup as in  \cite[Appendix A]{LLLM-models} with the property that for any Serre weight $\sigma$ for $\cG$,
\[
M_\infty(\sigma)\neq 0
\quad\Longleftrightarrow\quad
\rbar \text{ is automorphic of weight } \sigma.
\]
In particular, $\sigma\in W(\bar r)$ if and only if $M_\infty(\sigma)\neq 0$.  

By the same argument as in the proof of Theorem \ref{thm:BMconjeta} using \cite[Theorem 4.4.1]{BLGGT14}, for all types $\tau$, $M_{\infty}(\tau)$ has full support on  $\Spec R_{\infty}(\tau)$.  For $\sigma = \otimes_{v} \sigma_v$, we furthermore learn that the support cycle of $M_{\infty}(\sigma)$ is the pullback of   $\prod_{v \in \Sigma_p^+} \cZ_{\sigma_v}$ along $\Spec R_{\infty} \rightarrow \prod_{v \in \Sigma^+_p} \cX_{3, F_{\tld{v}}}$.  

\end{proof}

\begin{rmk}  Theorem \ref{thm:WPSC} is not restricted to the case of algebraic modular forms.  The same argument goes through in any situation where one can perform Taylor--Wiles patching.  For example, we could take a unitary group $G$ which has signature $(2,1)$ at some subset of the infinite places and compact at the others and patch the middle cohomology.  The main technical subtlety in this setting would be the exactness of the patching.  This could be arranged by assuming some genericity assumption at an auxiliary place and applying torsion vanishing results but we don't pursue this here. 
\end{rmk}


\appendix
\section{Reduction of some $\GL_n(\cO_F)$-types (by Andrea Dotto)} \label{appendix}

We fix a $p$-adic field~$F$ with ring of integers~$\cO_F$, residue field~$k_F$, and uniformizer~$\varpi_F$.
We let $\lbar F$ be an algebraic closure of~$F$, we write~$F_n$ for the unramified extension of~$F$ in~$\lbar F$ of degree~$n$, 
and we fix an $F$-linear embedding of~$F_n$ in the matrix algebra~$M_n(F)$, by choosing an $F$-basis of~$F_n$.
Finally, we choose a $p$-adic coefficient field~$E/\bQ_p$ with ring of integers~$\cO$, residue field~$\bF$, and uniformizer~$\varpi$, 
and we assume that $E$ contains $\bQ_p(\zeta_p)$, and that~$k_{F_n}$ embeds in~$\F$. 
The aim of this appendix is to prove a result concerning the reduction mod~$\varpi$ of certain smooth $\cO[\GL_n(\cO_F)]$-representations, whose definition we recall next.{\let\thefootnote\relax\footnote{I 
was supported by a Royal Society University Research Fellowship during the writing of this appendix.
I am grateful to Shaun Stevens for comments on a draft.}}

\begin{defn}\label{defn:types}
Let~$\fs$ be a Bernstein component of the category of smooth~$\lbar E[\GL_n(F)]$-modules, let $J \subset \GL_n(F)$ be a compact open subgroup, and let~$\kappa$ be an irreducible smooth
$E[J]$-module.
\begin{enumerate}
\item $(J, \kappa)$ is a \emph{type} for~$\fs$ if, for all irreducible smooth $\lbar E[\GL_n(F)]$-modules~$\pi$, we have $\Hom_{J}(\kappa, \pi) \ne 0$ if and only if $\pi \in \fs$.
\item $\fs$ is \emph{regular} if, for all irreducible $\pi \in \fs$, the monodromy operator on the Langlands parameter $\rec_{\lbar E}(\pi)$ is equal to zero.
\item We define $\cT \subset K_0(\GL_n(\cO_F), E)$ as the subgroup spanned by irreducible representations~$\tau$ such that $(\GL_n(\cO_F), \tau \otimes_E \lbar E)$ 
is a type for a regular Bernstein component of
$\lbar E[\GL_n(F)]$.
\end{enumerate}
\end{defn}

Note that if~$\tau \in \cT$ is a type for the regular Bernstein component~$\fs$, and $\rho : \Gal(\lbar F/F) \to \GL_n(E)$
is a potentially semistable Galois representation, such that $\operatorname{WD}(\rho)^{\text{F-ss}}$
is the Langlands parameter of an irreducible $\pi \in \fs$,
then~$\rho$ is potentially crystalline, by Definition~\ref{defn:types}~(2).
We will prove the following result.

\begin{thm}\label{thm: rational span}
Assume that~$|k_F| \geq n+1$.
Then the semisimplified mod~$\varpi$ reduction map
\[
r_\varpi: \cT \otimes_\bZ \bQ \to K_0(\GL_n(k_F), \F) \otimes_\bZ \bQ
\]
is surjective.
\end{thm}

We will prove Theorem~\ref{thm: rational span} by exhibiting a spanning set of the right-hand side, and showing that it is contained in the image of~$r_\varpi$, by explicitly constructing
elements of~$\cT$ providing a preimage. We begin by constructing the spanning set.

\begin{lemma}\label{lem: rational spanning set}
Let~$\cS$ be a set of representatives of the $\GL_n(k_F)$-conjugacy classes of pairs~$(T, \chi)$ consisting of the group $T \subset \GL_n(k_F)$ of $k_F$-points of a rational maximal torus,
and an $E$-character $\chi : T \to \lbar E^\times$.
Then $\{r_\varpi (\Ind_T^{\GL_n(k_F)}\chi ): (T, \chi) \in \cS\}$ is a spanning set of $K_0(\GL_n(k_F), \F) \otimes_\bZ \bQ$.
\end{lemma}
\begin{proof}
Since every $p$-regular (i.e.\ semisimple) conjugacy class in~$\GL_n(k_F)$ intersects the $k_F$-points of some rational maximal torus, this lemma can be proved by the same arguments
as for Artin's theorem on induction from cyclic subgroups. Here we offer an alternative argument based on some properties of Deligne--Lusztig characters.

Given $(T, \chi) \in \cS$, we write $\chi \mapsto R_T^{\GL_n}(\chi)$ for the Deligne--Lusztig induction map on virtual characters.
Then, by~\cite[Thm.\ 3.2]{LusztigSrinivasan}, $K_0(\GL_n(k_F), E) \otimes_\bZ \bQ$ is $\bQ$-spanned by $R_T^{\GL_n}(\chi)$,
as~$(T, \chi)$ runs through~$\cS$. 
On the other hand, the Steinberg character of~$\GL_n(k_F)$, denoted~$\St$, is nonzero on all semisimple conjugacy classes of~$\GL_n(k_F)$
(see e.g.\ \cite[Prop.\ 3.1, Rem.\ 3.2]{Lusztigdivisibility}).
Hence a consideration of Brauer characters shows that the operation of tensoring with~$r_{\varpi}(\St)$ induces an isomorphism
\[
- \otimes_\bF r_{\varpi}(\St): K_0(\GL_n(k_F), \bF) \otimes_\bZ \bQ \to K_0(\GL_n(k_F), \bF) \otimes_\bZ \bQ,
\]
and so $r_{\varpi}\bigl( R_{T}^{\GL_n}(\chi) \otimes_E \St \bigr)$ forms a spanning set of $K_0(\GL_n(k_F), \bF) \otimes_\bZ \bQ$ over~$\bQ$.
(In fact, the image of $K_0(\GL_n(k_F), \bF)$ under $- \otimes_\bR r_\varpi(\St)$ is the Grothendieck group of projective $\F[\GL_n(k_F)]$-modules, by~\cite[Thm.\ 1.1]{Lusztigdivisibility}.)
Finally, by~\cite[Prop.\ 7.3]{DeligneLusztig}, the virtual character $R_{T}^{\GL_n}(\chi) \otimes_E \St$ coincides, up to sign, with $\Ind_T^{\GL_n(k_F)}(\chi)$: more precisely,
\[
(-1)^{\sigma(G)-\sigma(T)}(R_{T}^{\GL_n}(\chi) \otimes_E \St) = \Ind_T^{\GL_n(k_F)}(\chi),
\]
where~$\sigma(-)$ denotes the rational rank of a group.
The lemma is an immediate consequence of these considerations. 
\end{proof}

Recall that the $\GL_n(k_F)$-conjugacy classes of rational maximal tori 
in~$\GL_n(k_F)$
are in bijection with partitions~$\cP$ of~$n$, the class corresponding to~$\cP = n_1 \geq n_2 \geq \cdots \geq n_t$ being represented by
\[
T_\cP \coloneqq \prod_{i=1}^t k_{F_{n_i}}^\times.
\] 
Then~$\cS$ will consist of the pairs~$(T_\cP, \chi_\cP)$, as~$\chi_\cP$ runs through $E$-valued characters of~$T_\cP$, and~$\cP$ runs through partitions of~$n$.

Let~$T_\cP^\circ$ be the preimage of~$T_\cP$ under the mod~$\varpi_F$ reduction map $\GL_n(\cO_F) \to \GL_n(k_F)$.
Note that the surjection $T_\cP^\circ \to T_\cP$ is split by the image of the Teichm\"uller lift $[k_{F_{n_i}}^\times]$ 
under our fixed embedding $F_{n_i}^\times \to \GL_{n_i}(F)$,
and so we have a semidirect decomposition
\[
T_\cP^\circ = T_\cP \ltimes (1+\varpi_FM_n(\cO_F)).
\]
By Lemma~\ref{lem: rational spanning set}, to prove Theorem~\ref{thm: rational span} it suffices to show that, for all $(T_\cP, \chi_\cP) \in \cS$,
there exists a character $\chi_\cP^\circ : T_\cP^\circ \to E^\times$ such that $\chi_\cP^\circ |_{T_\cP} = \chi_\cP$,  
and $(T_\cP^\circ, \chi_\cP^\circ)$ is a type for a regular Bernstein component.
Indeed, since~$1+\varpi_F M_n(\cO_F)$ is a pro-$p$ group, we have
\[
r_{\varpi}(\Ind_{T_\cP}^{\GL_n(k_F)}\chi_\cP) = r_{\varpi}(\Ind_{T_\cP^\circ}^{\GL_n(\cO_F)}\chi_\cP^\circ),
\]
and so we conclude that $r_{\varpi}(\Ind_{T_\cP}^{\GL_n(k_F)}\chi_\cP)$ is contained in~$r_{\varpi}(\cT)$, since $(\GL_n(\cO_F), \Ind_{T_\cP^\circ}^{\GL_n(\cO_F)}\chi_\cP^\circ)$ 
is a type for a regular Bernstein component.

\subsection*{Construction of types}
The characters~$\chi_\cP^\circ$ we seek to construct arise as $\GL_n(F)$-covers of Bushnell--Kutzko types 
for cuspidal Bernstein components, of small conductor, of Levi subgroups of $\GL_n(F)$. %
As such, their study goes back at least to~\cite{HoweGLn}, but some of their properties do not seem to be readily available in the literature
in the form we require.
For this reason, we provide details of the construction, although it is surely well-known to experts (our arguments are adapted from~\cite{CarayolGL_n, Rochetypes}).

\begin{defn}\label{defn:properties of sequences}
For any partition~$\cP = n_1 \geq n_2 \geq \cdots \geq n_t$ of~$n$, 
we will consider sequences~$\lbar u_\cP$ of $\lbar u_{i} \in k_{F_{n_i}}^\times, 1 \leq i \leq t,$ with the following two properties:  
\begin{enumerate}
\item $\lbar u_{n_i}$ has trivial stabilizer in $\Gal(k_{F_{n_i}}/k_F)$, and
\item if~$i \ne j$ and~$n_i = n_j$, then~$\lbar u_{n_i}$ and~$\lbar u_{n_j}$ are not conjugate under~$\Gal(k_{F_{n_i}}/k_F)$.
\end{enumerate}
\end{defn}

The following lemma implies the existence of such sequences for any partition~$\cP$.
(This is the only place in this appendix where the assumption that $|k_F|\geq n+1$ is used.)

\begin{lemma}\label{lem:selection lemma}
Assume that $|k_F| \geq n+1$, and let~$\cP$ be a partition of~$n$.
Then there exists a sequence~$\lbar u_\cP$ with properties~(1) and~(2) in Definition~\ref{defn:properties of sequences}.
\end{lemma}
\begin{proof}
Writing $q := |k_F|$, the number of Galois orbits on the set of nonzero primitive elements of $k_{F_{n_i}}/k_F$ is
\begin{equation}\label{eqn:Mobius inversion}
-\delta_{n_i = 1}+n_i^{-1}\sum_{d | n_i} \mu(d)q^{n_i/d}.
\end{equation}
Let~$m$ be the number of times that~$n_i$ occurs in~$\cP$.
We need to show that~\eqref{eqn:Mobius inversion} is at least~$m$. %
This follows from the inequalities
\[
-1+\sum_{d | n_i} \mu(d)q^{n_i/d} \geq \bigl(q^{n_i} - \sum_{d|n_i, d \ne n_i}q^d -1 \bigr) \geq (q-1)q^{n_i-1} \geq q -1 \geq n \geq n_i m. \qedhere
\]
\end{proof}

Given~$u_\cP$ as in Definition~\ref{defn:properties of sequences},
we write $u_{n_i}$ for the Teichm\"uller lift of~$\lbar u_{n_i}$ in~$\cO_{F_{n_i}}^\times$, 
and we regard~$u_{n_i}$ as an element of~$\GL_{n_i}(\cO_F)$ via our fixed embedding $F_{n_i} \to M_{n_i}(F)$. 
We then define
\begin{equation}\label{eqn:lifted sequence}
u_\cP \coloneq (u_{n_1}, \ldots, u_{n_t}) \in %
\GL_{n_1}(F) \times \cdots \times \GL_{n_t}(F).
\end{equation}

To construct our types, we follow a standard procedure of fixing 
an additive character $\psi : F \to \lbar E^\times$ which is trivial on~$\varpi_F \cO_F$ but not on~$\cO_F$. 
Then there is a bilinear pairing
\[
M_n(F) \times M_n(F) \to \lbar E^\times, (A, B) \mapsto \psi \trace(AB),
\]
which identifies $M_n(F)$ with its dual group of smooth characters $M_n(F) \to \lbar E^\times$. %
The annihilator of $\fp_A^i \coloneq \varpi_F^i M_n(\cO_F)$ with respect to this pairing is $\fp_A^{-i+1}$. %
Hence the function
\[
u_\cP^\circ : K_1 \to \lbar E^\times, 1+\varpi_F A \mapsto \psi \trace (u_\cP A)
\]
is a character of $K_1 \coloneq 1+ \varpi_F M_n(\cO_F)$. 
Since $E$ contains the $p$-th roots of unity, it follows that $u_\cP^\circ$ is valued in~$E$.
Since $u_{n_i} \in F_{n_i}^\times$, $u_\cP^\circ$ is invariant under conjugation by $T_\cP$, and so can be extended to a character of $T_\cP^\circ = T_\cP \ltimes K_1$.
We define $(\chi_\cP, u_\cP)^\circ$ to be the unique extension of~$u_\cP^\circ$ to~$T_\cP^\circ$ such that $(\chi_\cP, u_\cP)^\circ |_{T_\cP} = \chi_\cP$.

\begin{lemma}\label{lem:intertwining computation}
    Let $\lbar u_\cP, \lbar v_\cP$ satisfy properties~(1) and~(2) in Definition~\ref{defn:properties of sequences}. 
    \begin{enumerate}
    \item 
    Assume that $u_\cP^\circ, {v}_\cP^\circ$ intertwine in~$\GL_n(F)$. Then there exists a permutation~$\nu \in S_t$ such that 
    $\lbar u_{n_j}$ and~$\lbar v_{n_{\nu(j)}}$ are $\Gal(k_{F_{n_j}}/k_F)$-conjugate for all~$1 \leq j \leq t$.
    \item Let $\chi_\cP : T_\cP \to E^\times$ be a character. 
    Then the $\GL_n(F)$-intertwining set of $(T_\cP^\circ, (\chi_{\cP}, u_\cP)^\circ)$ is contained in $K_1 \bigl(\prod_{i=1}^t F_{n_i}^\times \bigr) K_1$.
    \end{enumerate}
\end{lemma}
\begin{proof}
Let~$g \in \GL_n(F)$, and assume that~$g$ intertwines~$u_\cP^\circ$ and~$v_\cP^\circ$, i.e.\ that the Hecke module
\begin{equation}\label{eqn:Hecke module}
\Hom_{\GL_n(F)}(\cInd_{K_1}^{\GL_n(F)}u_\cP^\circ, \cInd_{K_1}^{\GL_n(F)}v_\cP^\circ)
\end{equation}
has an element supported on $K_1gK_1$.
Then $u_\cP^\circ$ and~$\ad(g)^* v_\cP^\circ$
have the same restriction to $K_1 \cap g^{-1}K_1 g$.
Hence, for all~$x \in K_1 \cap g^{-1}K_1 g$, the equality
\[
\psi \trace(\varpi_F^{-1}u_\cP(x-1)) = \psi \trace(\varpi_F^{-1}v_\cP(gxg^{-1}-1))
\]
holds.
Rearranging, we find that %
\[
\psi \trace ((g^{-1}v_\cP g-u_\cP)\varpi_F^{-1}(x-1)) = 0,
\]
and so $g^{-1}v_\cP g-u_\cP$ is contained in the annihilator of $M_n(\cO_F) \cap g^{-1}M_n(\cO_F) g$, which is $\fp_A + g^{-1}\fp_A g$.
So there exist~$A_u, A_v \in M_n(\cO_F)$ such that
\[
u_\cP-\varpi_F A_u = g^{-1}(v_\cP-\varpi_F A_v)g.
\] 
By Lemma~\ref{lem:conjugacy II} below, there exist~$k_u, k_v \in K_1$ such that 
$\tld u_\cP \coloneq k_u^{-1}(u_\cP-\varpi_F A_u)k_u$ and $\tld v_\cP \coloneq k_v^{-1}(v_\cP-\varpi_F A_v)k_v$
are contained in
$\prod_{j=1}^t\cO_{F_{n_j}}^\times$.
Since $\tld u_\cP$ is congruent modulo $\varpi_F$ to $u_\cP$,
the reduction modulo $\varpi_F$ of 
its $j$-th component
$\tld u_{n_j}$ is $\lbar u_{n_j}$.
Similarly, the reduction modulo~$\varpi_F$ of~$\tld v_{n_j}$ is $\lbar v_{n_j}$.
Hence~$\tld u_{n_j}$, resp.\ $\tld v_{n_j}$ acts on the $j$-th direct summand in the decomposition $F^{\oplus n} = \bigoplus_{j=1}^t F^{\oplus n_j}$
by an endomorphism 
whose characteristic polynomial lifts the characteristic polynomial of $\lbar u_{n_j}$, resp.\ $\lbar v_{n_j}$.
Bearing in mind Definition~\ref{defn:properties of sequences}, we thus see that $\tld u_\cP, \tld v_\cP$ have squarefree characteristic polynomial.
Setting $\tld g \coloneq k_2^{-1}gk_1$, which conjugates~$\tld u_\cP$ to~$\tld v_\cP$ (and so sends eigenspaces of~$\tld u_\cP$ to eigenspaces of~$\tld v_\cP$),
it thus follows that there exists $\nu \in S_t$ such that $\tld g F^{\oplus n_j} = F^{\oplus n_{\nu(j)}}$ for all~$1 \leq j \leq t$.
In turn, this equality implies that $\tld u_{n_j}$ and~$\tld v_{n_{\nu(j)}}$ have the same characteristic polynomial; hence the same is true of 
$\lbar u_{n_j}$ and~$\lbar v_{n_{\nu(j)}}$, which are therefore Galois conjugates.
This concludes the proof of part~(1).

We now prove part~(2). It suffices to prove that if~$g$ intertwines $u_\cP^\circ$ with itself, then $g \in K_1 \bigl(\prod_{i=1}^t F_{n_i}^\times \bigr) K_1$.
To do so, we specialize the discussion above to $\lbar v_\cP \coloneq \lbar u_\cP$, using the same notation.
Definition~\ref{defn:properties of sequences}~(2) implies that necessarily $\nu = 1$. 
It follows that $\tld g \in \prod_{i=1}^t \GL_{n_i}(F)$, and furthermore $\tld g_j \tld u_{n_j} \tld g_j^{-1} = \tld v_{n_j}$, for all~$1 \leq j \leq t$.
Since $F_{n_j} = F[\tld u_{n_j}] = F[\tld v_{n_j}]$, this implies that~$\tld g_j$ normalizes~$F_{n_j}$ in~$\GL_{n_j}(F)$, 
and the resulting Galois automorphism of~$F_{n_j}$ is the identity, since $\tld u_{n_j}$ and~$\tld v_{n_j}$
are the same modulo~$\varpi_F$ (they both coincide with $\lbar u_{n_j} = \lbar v_{n_j}$).
Hence $\tld g_j \in F_{n_i}^\times$, which concludes the proof of part~(2), since $g \in K_1 \tld g K_1$.
\end{proof}

\begin{lemma}\label{lem:conjugacy II}
Let~$\cP$ be a partition of~$n$, let~$\lbar u_\cP$ be a sequence as in Definition~\ref{defn:properties of sequences}, and 
let~$u \coloneq u_\cP$ be as in~\eqref{eqn:lifted sequence}.
Then, for all $A \in M_n(\cO_F)$, there exists $k \in K_1$ such that
\[
k^{-1}(u-\varpi_F A)k \in \prod_{i=1}^{t}\cO_{F_{n_i}}^\times.
\]
\end{lemma}
\begin{proof}
Since $u-\varpi_F A$ is invertible in~$M_n(\cO_F)$,
it suffices to find~$k \in K_1$ such that 
\[
k^{-1}(u-\varpi_F A)k \in \cL \coloneq \prod_{i=1}^{t}\cO_{F_{n_i}}.
\]
We show that, for all~$r \geq 1$ and $u_r \in u + \varpi_F M_n(\cO_F)$ such that $u_r \in \cL + \varpi_F^r M_n(\cO_F)$,
there exists~$k_r \in K_r$ such that $u_{r+1} \coloneq k_r^{-1}u_r k_r$ is contained in $\cL + \varpi_F^{r+1}M_n(\cO_F)$.
Assuming this claim, we let~$u_1 \coloneq u - \varpi_F A$, and we choose~$k_r$ and~$u_r$ recursively; 
note that~$u_{r+1}$ is automatically contained in $u + \varpi_FM_n(\cO_F)$ if so is~$u_r$, so the recursion is possible.
Then, setting $k \coloneq \prod_{r \geq 1} k_r = \lim_{s \to \infty}\prod_{r=1}^sk_r$, we conclude that
$k^{-1}u_1k = \lim_{s \to \infty}u_{s+1}$ is contained in $\cL$, as desired.

Definition~\ref{defn:properties of sequences} implies that $\cL$ is the kernel of the Lie bracket $[u, -] : M_n(\cO_F) \to M_n(\cO_F)$,
and has a $[u, -]$-stable complement~$\cL^\perp$
on which $[u, -]$ acts invertibly.
So our assumption on~$u_r$ implies that there exist $L \in \cL$ and $L^\perp \in \cL^\perp$ such that
\[
u_r = u + \varpi_F L + \varpi_F^r L^\perp.
\]
Let~$M^{\perp} \in \cL^\perp$ be such that $[u, M^{\perp}] = L^{\perp}$. 
Let $k_r^{-1} \coloneq 1+\varpi_F^rM^\perp$.
Then, modulo $\varpi_F^{r+1}M_n(\cO_F)$, we have
\begin{align*}
k_r^{-1}u_rk_r \equiv u_r + \varpi_F^r[M^\perp, u_r] \equiv u_r + \varpi_F^r[M^\perp, u] &\equiv (u+\varpi_F L + \varpi_F^{r}[u, M^\perp]) + \varpi_F^r[M^\perp, u]\\
&\equiv u+\varpi_F L,  
\end{align*}
and since $u + \varpi_FL \in \cL$, we see that $k_r^{-1}u_rk_r \in \cL + \varpi_F^{r+1}M_n(\cO_F)$, as desired.
\end{proof}

The following proposition completes our construction of types, and so concludes the proof of Theorem~\ref{thm: rational span}.

\begin{prop}\label{prop:construction of types}
Let~$\cP = n_1 \geq n_2 \geq \cdots \geq n_t$ be a partition of~$n$, let $\lbar u_\cP$ be as in Definition~\ref{defn:properties of sequences}, and let $\chi_\cP : T_\cP \to E^\times$ be a character.
Then $(T_\cP^\circ, (\chi_\cP, u_\cP)^\circ)$ is a type for a regular Bernstein component of $\lbar E[\GL_n(F)]$.
\end{prop}
\begin{proof}

Assume first that $t = 1$, i.e.\ that $\cP$ has a single nonzero part.
Accordingly, we will replace the symbol~$\cP$ with~$n$.
Lemma~\ref{lem:intertwining computation} implies that the $\GL_n(F)$-intertwining set of~$(T_n^\circ, (\chi_{n}, u_n)^\circ)$
coincides with $K_1F_n^\times K_1 = F_n^\times K_1$.
Since $F_n^\times K_1$ is compact modulo the centre of~$\GL_n(F)$, this implies %
that the compact induction to~$\GL_n(F)$ of any extension of~$(\chi_n, u_n)^\circ$ to~$F_n^\times K_1$
is irreducible and cuspidal.
Since $F_n^\times K_1 = F^\times T_n^\circ$, we conclude that $(T_n^\circ, (\chi_n, u_n)^\circ)$ is a type for a cuspidal Bernstein component $\fs(\chi_{n}, u_{n})$ of
$\lbar E[\GL_n(F)]$.

Furthermore, we claim that if~$u_n, u_n'$ and~$\chi_n, \chi_n'$ are such that $\fs(\chi_n, u_n) = \fs(\chi'_n, u_{n}')$, 
then $u_n$ and~$u'_n$ are $\Gal(k_{F_n}/k)$-conjugate.
In fact, the equality of Bernstein components implies that there exists an irreducible smooth $\lbar E[\GL_n(F)]$ module such that $\Hom_{K_1}(u_n^\circ, \pi)$ and $\Hom_{K_1}((u_n')^\circ, \pi)$
are both nonzero, and so $\pi$ is a common quotient of $\cInd_{K_1}^{\GL_n(F)}u_n^\circ$ and
$\cInd_{K_1}^{\GL_n(F)}(u'_n)^\circ$.
Since these compact inductions are projective objects of the category of smooth $\lbar E[\GL_n(F)]$-modules, we conclude that
\[
\Hom_{\GL_n(F)}(\cInd_{K_1}^{\GL_n(F)}u_n^\circ, \cInd_{K_1}^{\GL_n(F)}(u_n')^\circ) \ne 0,
\]
and so $u_n^\circ, (u_n')^\circ$ intertwine in~$\GL_n(F)$.
Then the claim is a consequence of Lemma~\ref{lem:intertwining computation}.

We now drop the assumption that~$t = 1$, and we let~$\cP$ be any partition of~$n$. 
We will show that $(T_\cP^\circ, (\chi_\cP, u_\cP)^\circ)$ is a $\GL_n(F)$-cover
of
\[
\left ( \prod_{i=1}^t \GL_{n_i}(F), (\chi_{n_1}, u_{n_1})^\circ \boxtimes \cdots \boxtimes (\chi_{n_t}, u_{n_t})^\circ \right )
\]
in the sense of~\cite[Defn.\ 8.1]{BK98}.
We then deduce from~\cite[Thm.\ 8.3]{BK98} that $(T_\cP^\circ, (\chi_\cP, u_\cP)^\circ)$ a type for the parabolically induced Bernstein component
$\fs(\chi_{n_1}, u_{n_1}) \times \cdots \times \fs(\chi_{n_t}, u_{n_t})$ of~$\lbar E[\GL_n(F)]$. 
Since~$i \ne j$ implies that $u_{i}$ and~$u_{j}$ are not Galois conjugates, and so, by the previous paragraph, $\fs(\chi_{n_i}, u_{n_i}) \ne \fs(\chi_{n_j}, u_{n_j})$,
we furthermore see that this component is regular, which concludes the proof of the proposition.

There remains to check that the three conditions in \cite[Defn.\ 8.1]{BK98} are met by $(T_\cP^\circ, (\chi_\cP, u_\cP)^\circ)$.
But the first two are true by construction, and the third condition is a consequence of~\cite[Comments~8.2]{BK98}, since Lemma~\ref{lem:intertwining computation}
implies that (using notation from \cite[Comments~8.2]{BK98})
we have $\cH(G, (\chi_\cP, u_\cP)^\circ)_{\prod_{i=1}^t \GL_{n_i}(F)} = \cH(G, (\chi_\cP, u_\cP)^\circ)$.
\end{proof}

\bibliographystyle{amsalpha}
\bibliography{bibliography.bib}

\end{document}